\documentclass[twoside,11pt]{article}
\usepackage{geometry}
\usepackage{indentfirst}                            
\usepackage[colorlinks]{hyperref}
\hypersetup{linkcolor=blue,filecolor=black,urlcolor=blue, citecolor=black}  
\usepackage[numbers,sort&compress]{natbib} 
\usepackage[perpage,symbol*,marginal]{footmisc}  

\usepackage{graphicx}

\usepackage{amsmath} 
\usepackage{amsfonts}
\usepackage{amssymb} 
\usepackage{bm}           
\usepackage{mathrsfs} 
\usepackage{amsthm}  
\usepackage{accents}  
\usepackage{amsxtra}

\usepackage[titletoc,title]{appendix}
\usepackage{titlesec}
\titleformat{\section}{\large\bfseries\centering}{\thesection}{1em}{}
\titleformat{\subsection}{\it\centering\bfseries}{\thesubsection}{0.5em}{}
\titleformat{\subsubsection}[runin]{\bfseries}{\thesubsubsection.}{0.4em}{}[.]

\DeclareMathOperator{\dive}{div}

\def\p{\partial}

\makeatletter
\def\wideubar{\underaccent{{\cc@style\underline{\mskip10mu}}}}
\def\Wideubar{\underaccent{{\cc@style\underline{\mskip8mu}}}}
\makeatother
\makeatletter
\def\widebar{\accentset{{\cc@style\underline{\mskip10mu}}}}
\def\Widebar{\accentset{{\cc@style\underline{\mskip8mu}}}}
\makeatother

\newcommand*\xbar[1]{%
	\hbox{%
		\vbox{%
			\hrule height 0.4pt 
			\kern0.3ex
			\hbox{%
				\kern-0.1em
				\ensuremath{#1}%
				\kern-0.1em
			}%
		}%
	}%
}

\newcommand{\VERTiii}[1]{{\left\vert\kern-0.3ex\left\vert\kern-0.3ex\left\vert #1
		\right\vert\kern-0.3ex\right\vert\kern-0.3ex\right\vert}}
\newcommand{\VERT}{\vert\kern-0.3ex\vert\kern-0.3ex\vert}
\newcommand{\VERTl}{\left\vert\kern-0.3ex\left\vert\kern-0.3ex\left\vert}
\newcommand{\VERTr}{\right\vert\kern-0.3ex\right\vert\kern-0.3ex\right\vert}
\newcommand{\VERTbig}{\big\vert\kern-0.3ex\big\vert\kern-0.3ex\big\vert}
\newcommand{\VERTBig}{\Big\vert\kern-0.3ex\Big\vert\kern-0.3ex\Big\vert}

\usepackage{color}
\definecolor{DarkRed}{RGB}{139,0,0}
\definecolor{Purple}{RGB}{128,0,128}

\numberwithin{equation}{section}

\newtheorem{lemma}{Lemma}[section]
\newtheorem{proposition}[lemma]{Proposition}
\newtheorem{theorem}{Theorem}[section]

\newtheorem{definition}{Definition}[section]

\newcommand\w[1]{\makebox[2em]{$#1$}}
\newcommand\z[1]{\makebox[3em]{$#1$}}
\begin{document}
	
	\title{\bf Stability of Multidimensional Thermoelastic Contact Discontinuities\let\thefootnote\relax\footnotetext{The research of {\sc Gui-Qiang G. Chen} was supported in part by
the UK Engineering and Physical Sciences Research Council Awards
EP/E035027/1 and EP/L015811/1,
and the Royal Society--Wolfson Research Merit Award (UK).
The research of {\sc Paolo Secchi} was supported in part by
the Italy MIUR Project PRIN 2015YCJY3A-004.
The research of {\sc Tao Wang} was supported in part by
the National Natural Science Foundation of China Grants
11971359 and 11731008.}
	}

	\author{
		{\sc Gui-Qiang~G.~Chen}\thanks{e-mail: chengq@maths.ox.ac.uk}\\
		{\footnotesize Mathematical Institute, University of Oxford, Oxford, OX2 6GG, UK}
		\\[2mm]
		{\sc Paolo~Secchi}\thanks{e-mail: paolo.secchi@unibs.it}\\
		{\footnotesize DICATAM, Mathematical Division, University of Brescia, 25133 Brescia, Italy}\\[2mm]
		{\sc Tao~Wang}\thanks{e-mail: tao.wang@whu.edu.cn}\\
		{\footnotesize School of Mathematics and Statistics, Wuhan University, Wuhan, 430072, China}
	}
	
	\date{\today}

	\maketitle
	

	\begin{abstract}
		We study the system of nonisentropic thermoelasticity describing the motion of thermoelastic nonconductors of heat in two and three spatial dimensions, where the frame-indifferent constitutive relation generalizes that for compressible neo-Hookean materials. Thermoelastic contact discontinuities are characteristic discontinuities for which the velocity is continuous across the discontinuity interface. Mathematically, this renders a nonlinear multidimensional hyperbolic problem with a characteristic free boundary. We identify a stability condition on the piecewise constant background states and establish the linear stability of thermoelastic contact discontinuities in the sense that the variable coefficient linearized problem satisfies {\it a priori} tame estimates in the usual Sobolev spaces under small perturbations. Our tame estimates for the linearized problem do not break down when the strength of thermoelastic contact discontinuities tends to zero. The missing normal derivatives are recovered from the estimates of several quantities relating to physical involutions. In the estimate of tangential derivatives, there is a significant new difficulty, namely the presence of characteristic variables in the boundary conditions. To overcome this difficulty, we explore an intrinsic cancellation effect, which reduces the boundary terms to an instant integral. Then we can absorb the instant integral into the instant tangential energy by means of the interpolation argument and an explicit estimate for the traces on the hyperplane.

		\vspace{4mm}
		
		\vspace{2mm}
	\end{abstract}
	
	\tableofcontents

	
	\section{Introduction}

We study the equations of nonisentropic thermoelasticity in the Eulerian coordinates,
governing the evolution of thermoelastic nonconductors of heat in two and three spatial dimensions.
The constitutive relation under consideration
generalizes that for {\it compressible neo-Hookean materials}
(see {\sc Ciarlet} \cite[p.\;189]{C88MR936420})
and satisfies the necessary {\it frame indifference principle}
(see {\sc Dafermos} \cite[\S 2.4]{D16MR3468916}).
This system can be reduced to a symmetrizable hyperbolic system
on account of the divergence constraints.

Our main interest concerns the stability of {\it thermoelastic contact discontinuities}
that are piecewise smooth, weak solutions with the discontinuity interface,
across which
the mass does not transfer and the velocity is continuous.
The boundary matrix for the free boundary problem
of thermoelastic contact discontinuities
is always singular on the discontinuity interface.
In other words, thermoelastic contact discontinuities are
characteristic discontinuities to the system of thermoelasticity.
As is well-known, characteristic discontinuities, along with shocks and rarefaction waves,
are building blocks of general entropy solutions
of multidimensional hyperbolic systems of conservation laws (see, {\it e.g.}, {\sc Chen--Feldman} \cite{CF18MR3791458}).
Therefore, it is important to analyze the stability of thermoelastic contact discontinuities
when the initial thermodynamic process and interface are perturbed
from the piecewise constant background state.
Mathematically, this renders a nonlinear hyperbolic initial-boundary value problem
with a characteristic free boundary.

Our work is motivated by the results on
3D compressible current-vortex sheets \cite{CW08MR2372810, T09MR2481071},
2D MHD contact discontinuities \cite{MTT15MR3306348,MTT18MR3766987},
and 2D compressible vortex sheets in elastodynamics \cite{CHW17MR3628211,CHW19}.
For ideal compressible magnetohydrodynamics (MHD),
there are two types of characteristic discontinuities:
compressible current-vortex sheets and MHD contact discontinuities,
corresponding respectively to $H\cdot N|_{\varGamma}= 0$ and $H\cdot N|_{\varGamma}\neq 0$,
where $H$ is denoted as the magnetic field, $\varGamma$ as the discontinuity interface,
and $N$ as the spatial normal to $\varGamma$.
{\sc Chen--Wang} \cite{CW08MR2372810,CW12MR3289359} and {\sc Trakhinin} \cite{T09MR2481071}
established  the nonlinear stability of 3D compressible current-vortex sheets independently,
indicating the stabilization effect of non-paralleled magnetic fields
to the motion of 3D compressible vortex sheets.
The local existence of 2D MHD contact discontinuities
was proved by {\sc Morando et al.}~\cite{MTT15MR3306348,MTT18MR3766987}
under the Rayleigh--Taylor sign condition
on the jump of the normal derivative of the pressure through a series of delicate energy estimates.
Notice that the extension of the results in \cite{MTT15MR3306348,MTT18MR3766987}
to 3D MHD contact discontinuities is still a difficult open problem.
For the system of thermoelasticity,
{\sc Chen et al.}~\cite{CHW17MR3628211,CHW19}
recently obtained the linear stability  of the 2D isentropic compressible vortex sheets
associated with the boundary constraint:
${\bm F}\cdot N|_{\varGamma}=0$ for the deformation gradient ${\bm F}$,
by developing the methodology in {\sc Coulombel--Secchi} \cite{CS04MR2095445}.
Comparing with the aforementioned two types of characteristic discontinuities in MHD,
we naturally introduce and investigate the thermoelastic contact discontinuities that correspond to
${\bm F}\cdot N|_{\varGamma}\ne 0$.

The goal of this paper is to explore the stabilizing mechanism in thermoelasticity
such that the thermoelastic contact discontinuities are stable.
More precisely,
{we identify a stability condition on the piecewise constant background states
and establish the linear stability of thermoelastic contact discontinuities in the sense that
the variable coefficient linearized problem satisfies appropriate {\it a priori} tame estimates
under small perturbations.}
In particular, our tame estimates do not break down
when the strength of thermoelastic contact discontinuities tends to zero.
{As far as we know, this is the first rigorous result on the stability of thermoelastic contact discontinuities in the mathematical theory of thermoelasticity.}

In general, for hyperbolic problems with a characteristic boundary,
there is a loss of control on the derivatives
(precisely, on the normal derivatives of the characteristic variables) in {\it a priori} energy estimates.
To overcome this difficulty, it is natural to introduce the Sobolev spaces with conormal regularity,
where two tangential derivatives count as one normal derivative
(see {\sc Secchi}~\cite{S96MR1405665}  and the references therein).
However, for our problem,
we manage to work in  the usual Sobolev spaces,
since the missing normal derivatives of the characteristic variables
can be recovered from the estimates of several quantities relating to the physical constraints.

In the estimate of tangential derivatives, there is a significant new difficulty,
namely the presence of characteristic variables in the boundary conditions,
which is completely different from the previous works such as
\cite{CSW19MR3925528,CW08MR2372810, T09MR2481071,CW12MR3289359,
	MTT15MR3306348,CHW19,CHW17MR3628211,CS04MR2095445}.
New ideas are required to control the boundary integral term
arising in the estimate of tangential derivatives
owing to the complex nature of the boundary conditions.
To address this issue, we utilize a combination of the boundary conditions and
the restriction of the interior equations on the boundary
to exploit an intrinsic cancellation effect.
This cancellation enables us to reduce the boundary term into
the sum of the error term $\mathcal{R}_2$  ({\it cf.}~\eqref{R2.def})
and the instant boundary integral term $\mathcal{R}_3$ ({\it cf.}~\eqref{Q.est1}).

To establish the energy estimates uniform in the strength of the thermoelastic
contact discontinuity for
$\mathcal{R}_2$ and $\mathcal{R}_3$,
we cannot use the boundary conditions for the spatial derivatives of the discontinuity function $\psi$,
owing to the dependence of the coefficients on the strength ({\it cf.}\;\eqref{ELP3.b.4}).
In order to overcome this difficulty, we develop an idea
from {\sc Trakhinin} \cite[Proposition 5.2]{T18MR3721411}
and explore new identities and estimates for the derivatives of $\psi$
with the aid of the interpolation argument.
We make the estimate of $\mathcal{R}_3$ differently for the cases
whether it contains a time derivative.
More precisely, we first consider the case with at least one time derivative.
Thanks to the restriction of the interior equations on the boundary,
the time derivative of the deformation gradient in $\mathcal{R}_3$
can be transformed into the tangential space derivatives of the velocity
({\it cf.}\;\eqref{key2b}).
As a result, the estimate of traces on the hyperplane
({\it cf.}\;Lemma~\ref{lem.trace2}) can be applied
to control the primary term $\mathcal{R}_{31}$
({\it cf.}\;\eqref{R3a.est1}).
Employing the identities and estimates
for the normal derivative of the noncharacteristic variables,
we can reduce the estimate of the instant tangential energy
into that with one less time derivative and one more tangential spatial derivative
({\it cf.}\;\eqref{E.tan.est3}).
Then we are led to deal with the case containing the space derivatives.
For this case,
we derive estimates \eqref{E.tan.est4b} and \eqref{R3.est3d}
by means of the identities and estimates for linearized quantities $(\eta,\zeta)$
({\it cf.}\;\eqref{eta}--\eqref{eta.est'}) and Lemma~\ref{lem.trace2}.
With these estimates in hand,
we can finally obtain the desired estimate for all the tangential derivatives
under the stability condition \eqref{thm.H1} on the background state.
The methods and techniques developed here may be also helpful for other problems
involving similar difficulties.

It is worth noting that
our tame estimates are with a {\it fixed} loss of derivatives
with respect to the source terms and coefficients.
As such,
the local existence and nonlinearly structural stability of thermoelastic contact discontinuities
could be achieved
with resorting to a suitable Nash--Moser iteration scheme as in \cite{CS08MR2423311,CSW19MR3925528}.

Let us also mention some recent results on
the classical solutions and weak--strong uniqueness for
the system of polyconvex thermoelasticity.
{\sc Christoforou et al.}~\cite{CGT18MR3910194} enlarge
the equations of polyconvex thermoelasticity
into a symmetrizable hyperbolic system,
which yields
the local existence of classical solutions of the Cauchy problem
by applying the general theory in \cite[Theorem 5.4.3]{D16MR3468916}.
The convergence in the zero-viscosity limit
from thermoviscoelasticity to thermoelasticity is also provided in \cite{CGT18MR3910194}
by virtue of the relative entropy formulation developed in \cite{CT18MR3799089}.
Moreover, {\sc Christoforou et al.}~\cite{CGT18MR3910194,CGT19MR4026976}
establish the weak--strong uniqueness property
in the classes of entropy weak and measure-valued solutions.

The rest of this paper is organized as follows:
In Section \ref{Sec.Nonlinear},
we introduce the system of thermoelasticity in the Eulerian coordinates,
which can be symmetrizable hyperbolic, via the divergence constraints.
Then we formulate the free boundary problem
and the reduced problem in a fixed domain for thermoelastic contact discontinuities.
It should be pointed out that
no thermoelastic contact discontinuity is possible
for the isentropic process ({\it cf.}~Proposition \ref{pro1.1}).
Section \ref{Sec.Bas} is devoted to stating the main theorem of this paper, Theorem~\ref{thm2}.
Before that, based on an alternative form of the boundary operator,
we deduce the variable coefficient linearized problem around the basic state
(that is, a small perturbation of the stationary thermoelastic contact
discontinuity satisfying suitable constraints).
In Section \ref{sec.Preliminary}, we collect some properties of the Sobolev functions
and notations for later use, including the definitions of fractional Sobolev spaces and norms,
the estimates of the traces on the hyperplane, and the Moser-type calculus inequalities.
To show Theorem~\ref{thm2},
in Section~\ref{sec.Homogenization}, we reduce
the effective linear problem to a problem with homogeneous boundary conditions.
Section~\ref{sec.normal} is dedicated to the proof of Proposition \ref{lem.normal},
{\it i.e.}, the estimate of normal derivatives.
More precisely,
we estimate the noncharacteristic variables $W_{\rm nc}$
and entropies $S^{\pm}$ in \S \ref{sec.normal1}--\S \ref{sec.entropy},
recover the missing $L^2$-norm of
$\p_1 \mathrm{D}_{\rm tan}^{\beta} W_{1}$  and $\p_1 \mathrm{D}_{\rm tan}^{\beta} W_{jd+i+1}$
for $1\leq i \leq d$ and $2\leq j\leq d$
in \S \ref{sec.W1}--\S \ref{sec.normal4},
and complete the proof of Proposition \ref{lem.normal}
by  finite induction in \S \ref{sec.normal5}.
Let us remark that
quantities $\varsigma$, $\eta$, and $\zeta$
({\it cf.}~definitions \eqref{varsigma}, \eqref{eta}, and \eqref{zeta})
are introduced and estimated
to compensate the loss of
the normal derivatives of characteristic variables $W_1$ and $W_{jd+i+1}$.
In Section \ref{sec.Tangential},
we deduce the estimate of tangential derivatives, {\it i.e.}, Proposition \ref{lem.tan}.
For this purpose, we start with the standard energy estimate to introduce
the boundary term $Q$ ({\it cf}.~\eqref{Q.def}) and the instant tangential
energy $\mathcal{E}_{\rm tan}^{\beta}(t)$  ({\it cf}.~\eqref{E.tan.def}).
We present the intrinsic cancellation for $Q$ in \S \ref{sec.tan2}.
Then the boundary integral term can be reduced to
the sum of $\mathcal{R}_2$ (the error term, defined by \eqref{R2.def}) and
$\mathcal{R}_3$ (the instant boundary integral term, {\it cf}.~\eqref{Q.est1}).
After that, we deduce the estimate of $\mathcal{R}_2$ in \S \ref{sec.R2} and
the estimate of $\mathcal{R}_3$ in \S \ref{sec.tan4.a}--\S\ref{sec.tan4.c}.
Proposition \ref{lem.tan} is proved, respectively,
for the two- and three-dimensional cases at the end of \S \ref{sec.tan4.b}
and \S \ref{sec.tan4.c}.
With Propositions \ref{lem.normal} and \ref{lem.tan} in hand,
we conclude the proof of the main theorem in Section \ref{sec.Proof}.
Propositions~\ref{pro1.1} and \ref{pro1} are shown
in Appendices~\ref{App.A} and \ref{App.B}, respectively.

\section{Formulation of the Nonlinear Problems}\label{Sec.Nonlinear}
In this section, we introduce the system of thermoelasticity in the Eulerian coordinates
and formulate the nonlinear problems for thermoelastic contact discontinuities.

\subsection{Equations of Motion}
In the context of elastodynamics,
a body is identified with an open subset $\mathcal{O}$
of the reference space $\mathbb{R}^d$ for $d=2,3$.
A motion of the body over a time interval $(t_1,t_2)$
is a Lipschitz mapping ${x}$ of $(t_1,t_2)\times\mathcal{O}$ to  $\mathbb{R}^d$
such that ${x}(t,\cdot)$ is a bi-Lipschitz homeomorphism of $\mathcal{O}$ for each $t$ in $(t_1,t_2)$.
Every particle $X$ of body $\mathcal{O}$ is deformed to the spatial position $x(t, X)$ at time $t$.

The {\it velocity} $\tilde{v}\in\mathbb{R}^d$ with $i$-th component $\tilde{v}_i$
and the {\it deformation gradient}  $\widetilde{\bm{F}}\in \mathbb{M}^{d\times d}$
with $(i,j)$-th entry $\widetilde{F}_{ij}$ are defined by
\begin{align*}
\tilde{v}_i(t,X):=\frac{\p x_i}{\p t}(t,X), \qquad
\widetilde{F}_{ij}(t,X):=\frac{\p x_i}{\p X_j}(t,X),
\end{align*}
respectively,
where $\mathbb{M}^{m\times n}$ stands for the vector space of real $m \times n$ matrices.
We assume that
map $x(t,\cdot):\mathcal{O}\to\mathbb{R}^d$ is  orientation-preserving so that
\begin{align} \label{detF>0}
\det \widetilde{\bm{F}}(t,\cdot)>0\qquad\, \textrm{in $\mathcal{O}$}.
\end{align}
The compatibility between fields $\tilde{v}$ and $\widetilde{\bm{F}}$ is expressed by
\begin{align}
\label{F.eq.L}
\frac{\p \widetilde{F}_{ij}}{\p t }(t,X)=\frac{\p \tilde{v}_i}{\p {X_j} }(t,X)
\qquad\,  \textrm{for }       i,j=1,\ldots,d.
\end{align}
We need to append the constraints:
\begin{align} \label{inv1.L}
\frac{\p \widetilde{F}_{ij}}{\p X_k}=\frac{\p \widetilde{F}_{ik}}{\p X_j}
\qquad\, \textrm{for }    i,j,k=1,\ldots,d,
\end{align}
in order to guarantee that $\widetilde{\bm{F}}$ is a gradient.
We emphasize that
constraints \eqref{inv1.L} are {\it involutions} to the system of thermoelasticity,
meaning that
constraints \eqref{inv1.L}
are preserved by the evolution via relations \eqref{F.eq.L}, provided that
they hold at the initial time
(see {\sc Dafermos} \cite{D86MR846895}).

We will work in the Eulerian coordinates $(t,x)$.
For convenience,
let us denote by $v=(v_1,\ldots,v_d)^{\mathsf{T}}$ the velocity
and by $\bm{F}=(F_{ij})$ the deformation gradient
in the Eulerian coordinates so that
\begin{align}\notag
v_i(t,x)=\tilde{v}_i(t,X(t,x))\qquad\, {F}_{ij}(t,x)=\widetilde{{F}}_{ij}(t,X(t,x)),
\end{align}
where $X(t,x)$ is the inverse map of $x(t,X)$ for each fixed $t$.

The system of thermoelasticity modeling the motion of thermoelastic nonconductors of heat
consists of the kinematic relations:
\begin{align}\label{F.eq.E}
(\p_t +v_\ell\p_\ell){F}_{ij}=\p_\ell v_i {F}_{\ell j}\qquad\,\,  \textrm{for }   i,j=1,\ldots,d,
\end{align}
and the following conservation laws of mass, linear momentum, and energy (see \cite[\S 2.3]{D16MR3468916}):
\begin{align} \label{conser.laws}
\left\{
\begin{aligned}
&\p_t\rho+\p_{\ell}(\rho v_{\ell})=0,  \\
&\p_t(\rho v_i)+\p_{\ell} (\rho v_{\ell}v_i) = \p_{\ell}  {T}_{i\ell}\qquad\qquad   \textrm{for }   i=1,\ldots,d,\\
&\p_t(\rho \varepsilon+ \tfrac{1}{2}\rho |v|^2)+\p_{\ell}  ((\rho \varepsilon+ \tfrac{1}{2}\rho |v|^2)v_{\ell})=\p_{\ell}  (v_jT_{j\ell}),
\end{aligned}
\right.
\end{align}
where $\p_t:=\frac{\p}{\p t}$ and $\p_\ell:=\frac{\p}{\p x_\ell}$ represent the partial differentials,
$\rho$ is the (spatial) density related with reference density $\rho_{\rm ref}>0$ through
\begin{align} \label{rho.relation.E}
\rho =\rho_{\rm ref}\, (\det \bm{F})^{-1},
\end{align}
symbol  $T_{ij}$ denotes the $(i,j)$-th entry of the Cauchy stress tensor $\bm{T}\in\mathbb{M}^{d\times d}$,
and $\varepsilon$ is the (specific) internal energy.
Equations \eqref{F.eq.E} are directly from the compatibility relations \eqref{F.eq.L}.
In the Eulerian coordinates, constraints \eqref{inv1.L} are reduced to
\begin{align} \label{inv1.E}
F_{\ell k}\p_{\ell} F_{ij}=F_{\ell j}\p_{\ell} F_{ik}\qquad\,  \textrm{for }   i,j,k=1,\ldots,d,
\end{align}
which are the involutions of system \eqref{F.eq.E}--\eqref{conser.laws};
see {\sc Lei--Liu--Zhou} \cite[Remark 2]{LLZ08MR2393434} for instance.
Throughout this paper, we adopt the Einstein summation convention
whereby a repeated index in a term implies the summation over all the values of that index.

For every given thermoelastic medium,
the following constitutive relations hold
(see {\sc Coleman--Noll} \cite{CN63MR153153}):
\begin{align} \notag
\varepsilon=\varepsilon(\bm{F},{ S }),\quad\,
\bm{T} =\bm{T} ^{\mathsf{T}}
= \rho \frac{\p\varepsilon(\bm{F},{ S })}{\p \bm{F}}\bm{F}^{\mathsf{T}},
\quad\, \vartheta :=\frac{\p\varepsilon(\bm{F},{ S })}{\p { S }}>0,
\end{align}
where ${ S }$ and $\vartheta$ represent the (specific) entropy and the (absolute) temperature, respectively.
In this paper, we consider the internal energy functions of the form:
\begin{align}
\varepsilon(\bm{F},{ S })=\sum_{i,j=1}^d \frac{a_{j}}{2}  F_{ij} ^2+e(\rho,{ S }),
\label{varepsilon}
\end{align}
where $a_j$, for $j=1,\ldots,d$, are positive constants.
In view of \eqref{rho.relation.E},
the internal energy $\varepsilon(\bm{F},{ S })$ depends on the deformation gradient $\bm{F}$
only through $\bm{F}^{\mathsf{T}}\bm{F}$. Hence,
relation \eqref{varepsilon} is {\it frame-indifferent}:
\begin{align*}
\varepsilon(\bm{F},{ S })=\varepsilon(\bm{Q}\bm{F},{ S })
\end{align*}
for all $\bm{Q}\in \mathbb{M}^{d\times d}$ with $\bm{Q}\bm{Q}^{\mathsf{T}}=\bm{I}_d$ and $\det \bm{Q}=1$.
Here and below, $\bm{I}_m$ denotes the identity matrix of order $m$.
Moreover, the constitutive relation \eqref{varepsilon}  generalizes
that for the {\it compressible neo-Hookean materials}
(see \cite[p.\;189]{C88MR936420}) to the nonisentropic thermoelasticity.
A direct computation gives
\begin{align}\label{T.bm}
\bm{T}= \rho \bm{F} \mathrm{diag}\,(a_1,\ldots,a_d)\bm{F}^{\mathsf{T}}-p \bm{I}_d,
\end{align}
with
\begin{align}\label{p.def}
p:= \rho^2\frac{\p e(\rho,{ S })}{\p \rho}, \qquad\, \vartheta=\frac{\p e(\rho,{ S })}{\p S}>0,
\end{align}
where $p=p(\rho,{ S })$ is the pressure.
The speed of sound $c=c(\rho,{ S })$ is assumed to satisfy
\begin{align}\label{c.def}
c(\rho,{ S }):=\sqrt{p_{\rho}(\rho,{ S })}>0\qquad\, \textrm{for } \rho>0.
\end{align}
If all of $a_j$ are the same, then the material is {\it isotropic};
otherwise it is {\it anisotropic}
(see \cite[\S 3.4]{C88MR936420}).
In the special case when all of $a_j$ are equal to zero,
{system \eqref{conser.laws} is reduced to the compressible Euler equations in gas dynamics.}
Since this paper concerns the effect of elasticity to the evolution of materials,
we set without loss of generality that $a_j=1$ for all $j$.
We refer to
\cite[Chapters 3--4]{C88MR936420} and \cite[Chapter 2]{D16MR3468916}
for a thorough discussion of the constitutive relations.

For simplicity, the reference density $\rho_{\rm ref}$ is supposed to be unit,
leading to
\begin{gather} \label{inv2.E}
\dive(\rho {F}_j):=
\p_{\ell} (\rho F_{\ell j})=0\qquad\,\textrm{for $j=1,\ldots, d$},
\end{gather}
where ${F}_j $ stands for the $j$-th column of $\bm{F}$; see, {\it e.g.}, \cite[Remark 1]{LLZ08MR2393434}.
By virtue of the divergence constraints \eqref{inv2.E},
we can reformulate \eqref{F.eq.E} and \eqref{inv1.E} into the conservation laws:
\begin{alignat}{2}
\label{F.eq.conser}&\p_t(\rho {F}_{ij})+\p_{\ell}(\rho {F}_{ij}v_{\ell}- v_i \rho {F}_{\ell j})=0
\qquad&&\textrm{for }    i,j=1,\ldots,d,\\[1mm]
\label{inv1.conser}&\p_{\ell}(\rho F_{\ell k}F_{ij}-\rho F_{\ell j}F_{ik} )=0
\quad&& \textrm{for }  i,j,k=1,\ldots,d.
\end{alignat}
In light of \eqref{p.def}--\eqref{inv2.E}, for smooth solutions,
system \eqref{F.eq.E}--\eqref{conser.laws}  is  equivalent to
\begin{subequations}
\label{sym.form}
 \begin{alignat}{2}
 \label{sym.form.1}&(\partial_t+v_{\ell}\p_{\ell})p+\rho c^2\p_{\ell} v_{\ell}=0,&&\\[0.5mm]
 \label{sym.form.2}&\rho (\partial_t+v_{\ell}\p_{\ell}) v_i+\p_i p-    \rho {F}_{ \ell k } \p_{\ell } {F}_{i k }=0
 \qquad &&\textrm{for } i=1,\ldots,d,\\[0.5mm]
 \label{sym.form.3}&   \rho (\partial_t+v_{\ell}\p_{\ell}) {F}_{ij}-   \rho {F}_{\ell j} \p_{\ell} v_i=0
 &&\textrm{for }  i,j=1,\ldots,d,\\[0.5mm]
 \label{sym.form.4} & (\partial_t+v_{\ell}\p_{\ell}) { S }=0.&&
 \end{alignat}
\end{subequations}
Let us take $U=(p,v,{F}_{1},\ldots,{F}_{d},{ S })^{\mathsf{T}}$ as the primary unknowns
and define the following symmetric matrices:
\setlength{\arraycolsep}{3pt}
\begin{align}
\label{A0.def}
A_0(U)&:=\mathrm{diag}\big( 1/(\rho c^2),\,\rho \bm{I}_{d+d^2}, \,1\big),\\[2mm]
\label{Aj.def}
A_i(U)&:= \begin{pmatrix}
v_i/(\rho c^2)  &  \bm{e}_i^{\mathsf{T}}&0&\cdots&0&0\\[0.5mm]
\bm{e}_i& \rho v_i \bm{I}_d&- \rho {F}_{i1}\bm{I}_{d}&\cdots&- \rho {F}_{id}\bm{I}_{d}&0\\[0.5mm]
0&- \rho {F}_{i1}\bm{I}_{d}&   &  & &0\\[-0.5mm]
\vdots &\vdots &  & \rho v_i \bm{I}_{d^2}& &\vdots\\[0.5mm]
0&- \rho {F}_{id}\bm{I}_{d}& &  &   &0\\[0.5mm]
\w{0}&\w{0}&\w{0}&\w{\cdots}&\w{0}&\w{v_i}
\end{pmatrix}
\end{align}
for $i=1,\ldots, d,$
where we denote $\bm{e}_i:=(\delta_{i,1},\ldots,\delta_{i,d})^{\mathsf{T}}$
with $\delta_{i,j}$ being the Kronecker delta.
Then system \eqref{sym.form} reads
\begin{align} \label{vec.form}
A_0(U)\partial_t U+ A_{i} (U)\partial_{i}U=0,
\end{align}
which is symmetric hyperbolic, due to \eqref{rho.relation.E} and \eqref{c.def}.

\subsection{Thermoelastic Contact Discontinuities}
Let $U$ be smooth on each side of a smooth hypersurface
$\varGamma(t):=\{x\in\mathbb{R}^d: x_1=\varphi(t,x')\}$ for $x':=(x_2,\ldots,x_d)$:
\begin{align}\label{U.form}
U(t,x)=\begin{cases}
U^+(t,x)\qquad \textrm{in }\ \varOmega^+(t):=\{x\in\mathbb{R}^d: x_1>\varphi(t,x') \},\\[0.5mm]
U^-(t,x)\qquad \textrm{in }\ \varOmega^-(t):=\{x\in\mathbb{R}^d: x_1<\varphi(t,x') \},
\end{cases}
\end{align}
where $U^{\pm}(t,x)$ are smooth functions in respective domains $\varOmega^{\pm}(t)$.
Then $U$ is a weak solution of \eqref{conser.laws} and \eqref{inv2.E}--\eqref{inv1.conser}
if and only if it is a smooth solution of
\eqref{inv1.E}, \eqref{inv2.E}, and \eqref{vec.form} in domains $\varOmega^{\pm}(t)$,
and the following Rankine--Hugoniot jump conditions hold
at every point of front $\varGamma(t)$:
\begin{subequations}
	\label{RH1}
	\begin{alignat}{3}
	\label{RH1.a}  &[m_{N}]=0,&&\\[0.5mm]
	\label{RH1.b}  &[m_{N}v]+[\rho {F}_{\ell N} {F}_{\ell }]=N[p],&&\\[0.5mm]
	\label{RH1.c}  &[m_{N}(\varepsilon+\tfrac{1}{2}|v|^2)]+[\rho  F_{\ell N}  F_{\ell}\cdot v ]= [p v_N ],\quad &&\\[0.5mm]
	\label{RH1.d}  &[m_{N}{F}_{ij}]+[\rho {F}_{jN}v_i]=0 &&\textrm{for }   i,j=1,\ldots,d,\\[0.5mm]
	\label{inv3.E}  &[\rho {F}_{jN}]=0&&\textrm{for }   j=1,\ldots,d,\\[0.5mm]
	\label{inv4.E}  &[\rho F_{kN}F_{ij}-\rho F_{jN} {F}_{ik}]=0 &&\textrm{for }  i,j,k=1,\ldots,d,
	\end{alignat}
\end{subequations}
where $[g]:=(g^+-g^-)|_{\varGamma(t)}$ stands for the jump across $\varGamma(t)$, and
\begin{align*}
v_N^{\pm}:=v^{\pm}\cdot N,\quad\,\,
{F}_{jN}^{\pm}:={F}_{j}^{\pm}\cdot N,\quad\,\,
m_{N}^{\pm}:=\rho ^{\pm}(\p_t\varphi-v_N^{\pm})
\end{align*}
with $N:=(1,-\p_2\varphi,\ldots,-\p_d\varphi)^{\mathsf{T}}$,
so that $m^\pm_N$ represent the mass transfer {fluxes}.
Also see \cite[\S 3.3]{D16MR3468916}
for the corresponding jump conditions written in the Lagrangian description.

We are interested in discontinuous weak solutions $U$
for which the mass does not {transfer} across the discontinuity interface $\varGamma(t)$:
\begin{align} \label{m.N}
m_N ^{\pm}=\rho ^{\pm}(\p_t\varphi-v_N^{\pm})= 0\qquad\, \textrm{on $\varGamma(t)$}.
\end{align}
Then the matrix:
\begin{align*}
&\big(\p_t\varphi A_0(U)-N_{\ell}A_{\ell}(U)\big)\big|_{\varGamma(t)} \\[1mm] &\quad 
=\left.\begin{pmatrix}
0 &  -N^{\mathsf{T}}&0&\cdots &0&0\\[0.5mm]
-N& \bm{O}_d& \rho  {F}_{1N}\bm{I}_{d}&\cdots & \rho  {F}_{dN} \bm{I}_{d}&0\\[0.5mm]
0& \rho  {F}_{1N}\bm{I}_{d}& &  & &0\\[-0.5mm]
\vdots & \vdots&  & \bm{O}_{d^2}& &\vdots\\[0.5mm]
0& \rho  {F}_{dN}\bm{I}_{d}&  & & &0\\[0.5mm]
\w{0}&\w{0}&\w{0}&\w{\cdots}&\w{0}&\w{0}
\end{pmatrix}\right|_{\varGamma(t)}
\end{align*}
has eigenvalues
\begin{alignat*}{4}
&&&\pm\sqrt{|N|^2+ \rho^2 {F}_{\ell N} {F}_{\ell N}}\qquad  &&  \textrm{with multiplicity 1},\\
&&&\pm \rho\sqrt{ {F}_{\ell N}  {F}_{\ell N}}  &&\textrm{with multiplicity $d-1$},\\
&&&\ 0  &&\textrm{with multiplicity $d^2-d+2$},
\end{alignat*}
where $\bm{O}_m$ denotes the zero matrix of order $m$.
As a result, the boundary matrix on $\varGamma(t)$:
\begin{align*}
A_{\rm bdy}
:=\mathrm{diag}\big(  \p_t\varphi A_0(U^{+})-N_{\ell}A_{\ell}(U^{+}),
-\p_t\varphi A_0(U^{-})+N_{\ell}A_{\ell}(U^{-})\big)\big|_{\varGamma(t)}
\end{align*}
is singular, which implies that the free boundary $\varGamma(t)$ is {\it characteristic}.
In this sense, the weak solution $U$ is a {\it characteristic discontinuity}.

We now reformulate the jump conditions \eqref{RH1} by means of assumption \eqref{detF>0}.
More precisely, from \eqref{detF>0}, we derive
\begin{align} \label{F.N}
\begin{pmatrix}
{F}_{1N}^{\pm}\\[-0.5mm] \vdots \\  {F}_{dN}^{\pm}
\end{pmatrix}
=
\begin{pmatrix}
{F}_{11}^{\pm}\\[-0.5mm] \vdots \\  {F}_{1d}^{\pm}
\end{pmatrix}
-\sum_{\ell=2}^d\p_{\ell}\varphi
\begin{pmatrix}
{F}_{\ell 1}^{\pm}\\[-0.5mm] \vdots \\   {F}_{\ell d}^{\pm}
\end{pmatrix}
\neq 0 \qquad\,\, \textrm{on }\   \varGamma(t).
\end{align}
Consequently, the boundary matrix $A_{\rm bdy}$ on $\varGamma(t)$
has $2d$ negative, $2d$ positive, and $2(d^2-d+2)$ zero eigenvalues.
Since one more boundary condition is needed to determine the unknown interface function $\varphi$,
the correct number of boundary conditions is $2d+1$,
according to the well-posedness theory for hyperbolic problems.
Plugging involutions \eqref{inv3.E} and condition \eqref{m.N} into \eqref{RH1.d} leads to
\begin{align}\notag
{F}_{jN}^+ [v]=0\qquad  \textrm{on $\varGamma(t)$, for $j=1,\ldots,d$.}
\end{align}
Then it follows from \eqref{F.N} that $[v]=0$ on $\varGamma(t)$.
We employ \eqref{inv3.E} and \eqref{m.N} again to rewrite \eqref{RH1.a}--\eqref{RH1.d} as
\begin{equation} \label{BC.E.general}
\p_t\varphi=v_{N}^+, \quad\,\,
[v]=0,
\quad\,\,
\rho^+ {F}_{\ell N}^+[{F}_{\ell}]=N[p] \qquad\,\,  \textrm{on $\varGamma(t)$}.
\end{equation}

\begin{definition} \label{def1.TCD}
A {thermoelastic contact discontinuity} is
a discontinuous weak solution of form \eqref{U.form}
of system \eqref{conser.laws} and \eqref{inv2.E}--\eqref{inv1.conser}
with the boundary conditions \eqref{BC.E.general}.

\end{definition}
We exclude \eqref{inv3.E}--\eqref{inv4.E} from  \eqref{BC.E.general} in order to
prescribe the correct number of boundary conditions for the well-posedness
of the thermoelastic contact discontinuity problem.
On one hand, \eqref{inv3.E}--\eqref{inv4.E}  are involutions inherited from the initial data.
On the other hand, they prevent any thermoelastic contact discontinuity in the {\it isentropic} process.
More generally, we have the following physically relevant result whose proof is postponed to Appendix \ref{App.A}.

\begin{proposition} \label{pro1.1}
	If $[{ S }]=0$ on $\varGamma(t)$, then $[{ U }]=0$ on $\varGamma(t)$ so that
	no thermoelastic contact discontinuity exists.
\end{proposition}

If condition \eqref{detF>0} is ignored on interface $\varGamma(t)$,
then there is another type of characteristic discontinuities
for \eqref{conser.laws} and \eqref{inv2.E}--\eqref{inv1.conser}
with the constitutive relation \eqref{T.bm},
{\it i.e.}, the so-called {\it compressible vortex sheets}
that are associated with the boundary constraints
$({F}_{1N}^{\pm},\ldots,{F}_{dN}^{\pm})|_{\varGamma(t)}= 0$.
In this case, the jump conditions \eqref{RH1}  are reduced to
\begin{align}\notag
\p_t\varphi-v_{N}^+=\p_t\varphi-v_{N}^-=[p]=0\qquad\, \textrm{on $\varGamma(t)$}.
\end{align}
Then the normal velocity and pressure are continuous across front $\varGamma(t)$,
while the tangential components of the velocity can  undergo a jump.
See
\cite{CHW17MR3628211,CHW19}
for the two-dimensional isentropic case  in this regard.

In this paper, we focus on
the thermoelastic contact discontinuity problem corresponding to
the boundary constraints:
\begin{align} \label{inv5.E}
{F}_{1 N}^{\pm}\neq 0, \quad  {F}_{2N}^{\pm}=\cdots={F}_{dN}^{\pm}=0
 \qquad \textrm{on $\varGamma(t)$}.
\end{align}
Then the boundary conditions \eqref{BC.E.general} on $\varGamma(t)$ become
\begin{align}\label{BC.E}
\left\{\begin{aligned}
&\p_t\varphi-v_{N}^+=0,\quad &&[v]=0,\\[0.5mm]
& \rho^+ {F}_{1N}^+[{F}_{11}]=[p], \quad
&&[{F}_{11}\p_i\varphi+F_{i1}]=0\quad \textrm{for $i=2,\ldots, d$}.
\end{aligned}\right.
\end{align}
By virtue of \eqref{inv5.E}, involutions \eqref{inv4.E} are equivalent to
\begin{align} \label{inv4.E'}
[F_j]=0\qquad \textrm{on $\varGamma(t)$, for $j=2,\ldots,d.$}
\end{align}
Since $\varphi$ describing the discontinuity front $\varGamma(t)$
is one of the unknowns,
the thermoelastic contact discontinuity problem is a {\it free boundary problem}.

Taking into account the Galilean invariance of \eqref{inv1.E}, \eqref{inv2.E}, \eqref{vec.form}, and \eqref{RH1},
we choose the following piecewise constant thermoelastic contact discontinuity
as the background state:
\begin{align}
\label{background}
\bar{\varphi}=0,\qquad
\widebar{U}(x)=
\left\{\begin{aligned}
&\widebar{U}^+:=(\bar{p}^+,0,\widebar{\bm{F}}^+, \widebar{ S }^+)\quad &\textrm{if } x_1>0,\\
&\widebar{U}^-:=(\bar{p}^-,0, \widebar{\bm{F}}^-, \widebar{ S }^-)\quad &\textrm{if } x_1<0,
\end{aligned}\right.
\end{align}
where $\widebar{\bm{F}}^{\pm}=\mathrm{diag}\big(\widebar{F}_{11}^{\pm},
 \widebar{F}_{22}, \ldots,  \widebar{F}_{dd}\big)$ and
\begin{align} \label{background1}
\bar{p}^{\pm}=p( \bar{\rho}^{\pm},\widebar{ S }^{\pm} ),\quad
  \bar{\rho}^{+}\widebar{F}_{11}^+[\widebar{F}_{11}]= [\bar{p}]
 \qquad\,\, \textrm{for $\bar{\rho}^{\pm}:=(\det \widebar{\bm{F}}^{\pm})^{-1}$},
\end{align}
in keeping with \eqref{rho.relation.E}, \eqref{p.def}, and \eqref{inv5.E}--\eqref{inv4.E'}.
Without loss of generality,  we assume  that
the principal stretches $\widebar{F}_{11}^{\pm},$ $\widebar{F}_{22},$ $\ldots,$ $\widebar{F}_{dd}$
are positive constants with
$\widebar{F}_{11}^{+}>\widebar{F}_{11}^{-}$.
We point out that
each of the background deformations
is either a {\it dilation} or a {\it simple extension}
when $\widebar{F}_{22}=\cdots=\widebar{F}_{dd}$
(see {\sc Truesdell--Toupin} \cite[\S 43--\S 44]{TT60MR0118005}).

\subsection{Reduced Problem in a Fixed Domain}

It is more convenient to convert the free boundary problem for thermoelastic contact discontinuities into
a problem in a fixed domain.
To this end, we replace unknowns $U^{\pm}$, being smooth in $\varOmega^{\pm}(t)$, by
\begin{align} \label{transform}
U^{\pm}_{\sharp}(t,x):=U(t,\varPhi^{\pm}(t,x),x'),
\end{align}
where we take the lifting functions $\varPhi^{\pm}$
as in {\sc M{{\'e}}tivier} \cite[p.\;70]{M01MR1842775} to have the form:
\begin{align}     \label{varPhi}
\varPhi^{\pm}(t,x):=\pm x_1+\chi(\pm x_1)\varphi(t,x'),
\end{align}
with $\chi\in C^{\infty}_0(\mathbb{R})$ satisfying
\begin{align} \label{chi}
\chi\equiv 1\quad \textrm{on $[-1,1]$},  \qquad\,\,
\|\chi'\|_{L^{\infty}(\mathbb{R})} <1.
\end{align}
The cut-off function $\chi$ is introduced as in \cite{M01MR1842775,MTT18MR3766987}  to
avoid the assumption in the main theorem that the initial perturbations have compact support.
This change of variables is admissible on the time interval $[0,T]$
as long as $\|\varphi\|_{L^{\infty}([0,T]\times\mathbb{R}^{d-1})}\leq \frac{1}{2}$.

The existence of thermoelastic contact discontinuities
amounts to constructing solutions $U^{\pm}_{\sharp}$,
which are smooth in the fixed domain $\varOmega:=\{x\in\mathbb{R}^d:x_1>0\}$,
of the following initial-boundary value problem:
\begin{subequations}   \label{TCD0}
\begin{alignat}{2}
\label{TCD0.a}
&\mathbb{L}(U^{\pm},\varPhi^{\pm}) :=L(U^{\pm},\varPhi^{\pm})U^{\pm} =0
&  \qquad &\textrm{if }   x_1>0,\\
\label{TCD0.b}
&\mathbb{B}(U^{+},U^{-},\varphi)=0
& \qquad  &\textrm{if }  x_1=0,\\
\label{TCD0.c}
&(U^{+},U^{-},\varphi)=(U^{+}_0,U^{-}_0,\varphi_0)
& \qquad &\textrm{if } t=0,
\end{alignat}
\end{subequations}
where index $``\sharp"$ has been dropped for notational simplicity.
Thanks to transformation \eqref{transform}, operator $L(U,\varPhi)$ is given by
\begin{align}
\label{L.def}
L(U,\varPhi):=A_0(U)\partial_t+\widetilde{A}_1(U,\varPhi)\partial_1+\sum_{i=2}^d A_i(U)\partial_i,
\end{align}
where $A_i(U)$, for $i=0,\ldots, d$, are defined by \eqref{A0.def}--\eqref{Aj.def}, and
\begin{align}  \notag
\widetilde{A}_1(U, \varPhi):=
\frac{1}{\partial_1\varPhi}\Big(A_1(U)-\partial_t\varPhi A_0(U)-\sum_{i=2}^d\partial_i\varPhi A_i(U)\Big).
\end{align}
According to \eqref{BC.E}, the boundary operator $\mathbb{B}$ reads
\begin{align}
\label{B.bb.def}
\mathbb{B}(U^+,U^-,\varphi):=
\begin{pmatrix}
\partial_t\varphi-v_N^+ \\[0.5mm]
[v] \\[0.5mm]
 [p]-\rho^+{F}_{1N}^+ [{F}_{11}]\\[0.5mm]
[{F}_{11}\p_2\varphi+{F}_{21}]\\
\vdots\\
[{F}_{11}\p_d\varphi+{F}_{d1}]
\end{pmatrix}.
\end{align}

{The boundary matrix
$\mathrm{diag}(-\widetilde{A}_1(U^+, \varPhi^+),-\widetilde{A}_1(U^-, \varPhi^-))$
for problem \eqref{TCD0}}
has $2d$ negative eigenvalues (``incoming characteristics'')
on boundary $\partial\varOmega:=\{x\in\mathbb{R}^d:x_1=0\}$.
As discussed before, the correct number of boundary conditions is $2d+1$,
which is just the case in \eqref{TCD0.b}.

In accordance with \eqref{rho.relation.E}--\eqref{inv1.E}, \eqref{inv3.E}--\eqref{inv4.E}, and \eqref{inv5.E},
we assume that the initial data \eqref{TCD0.c} satisfy
\begin{alignat}{4}
\label{rho.relation}& \rho^{\pm}=( {\rm det\,}\bm{F}^{\pm})^{-1}
&&& \quad&\textrm{if } x_1\geq 0,\\
\label{inv1}& F_{\ell k}^{\pm} \p_{\ell} ^{\varPhi^{\pm}} F_{ij}^{\pm} -F_{\ell j}^{\pm} \p_{\ell} ^{\varPhi^{\pm}} F_{ik}^{\pm}=0
\quad &&\textrm{for }  i,j,k=1,\ldots,d, & \quad&\textrm{if }  x_1> 0,\\
\label{inv3}&[\rho {F}_{jN}]=0
&&\textrm{for }  j=1,\ldots,d, &\quad&\textrm{if }  x_1=0,\\
\label{inv4}& [\rho F_{kN}F_{ij}-\rho F_{jN} {F}_{ik}]=0
&&\textrm{for }   i,j,k=1,\ldots,d,&\quad&\textrm{if } x_1=0,\\
\label{inv5}&{F}_{jN}^{\pm} =0
&&\textrm{for }  j=2,\ldots,d, &\quad&\textrm{if } x_1=0.
\end{alignat}
Here and below, to simplify the notation, we denote the partial differentials
with respect to the lifting function $\varPhi$  by
\begin{align} \label{differential}
\partial_t^{\varPhi}:=\partial_t-\frac{\partial_t\varPhi}{\partial_1\varPhi}\p_1,\ \
\partial_1^{\varPhi}:=\frac{1}{\partial_1\varPhi}\partial_1,\ \
\partial_i^{\varPhi}:=\partial_i-\frac{\partial_i\varPhi}{\partial_1\varPhi}\partial_1\ \ \mbox{for $i=2,\ldots,d.$}
\end{align}

The following proposition manifests that identities \eqref{rho.relation}--\eqref{inv5}
are involutions in the straightened coordinates.
See Appendix \ref{App.B} for the proof.

\begin{proposition}\label{pro1}
 For each sufficiently smooth solution of problem \eqref{TCD0} on the time interval $[0,T]$,
 if constraints \eqref{rho.relation}--\eqref{inv5} are satisfied at the initial time,
 then these constraints and
 \begin{align}
 \label{inv2}\partial_{\ell}^{\varPhi^{\pm}} (\rho^{\pm}{F}_{\ell  j}^{\pm})=0\qquad
\textrm{if $x_1>0,$  for $j=1,\ldots,d$,}
 \end{align}
 hold for all $t\in[0,T]$.
\end{proposition}

Relations \eqref{inv2} are involutions in the straightened coordinates
corresponding to the divergence constraints \eqref{inv2.E},
from which we can pass from the Eulerian to the Lagrangian formulation
of the thermoelastic contact discontinuity problem.


\section{Linearized Problem and Main Theorem}\label{Sec.Bas}

In this section we introduce the basic state $(\mathring{U}^{\pm} ,\mathring{\varphi} )$
that is a small perturbation of the stationary thermoelastic contact discontinuity
$(\widebar{U}^{\pm},\bar{\varphi})$ given in
\eqref{background}--\eqref{background1}.
Then we perform the linearization
and state the main theorem of this paper.

\subsection{Basic State}
We denote
$\varOmega_T:=(-\infty,T)\times \varOmega$ and $\omega_T:=(-\infty,T)\times\partial \varOmega$
for any real number $T$.
Let the basic state $(\mathring{U}^{\pm},\mathring{\varphi})$ with
$\mathring{U}^{\pm}:=(\mathring{p}^{\pm},\mathring{v}^{\pm},\mathring{\bm{F}}^{\pm},\mathring{ S }^{\pm})^{\mathsf{T}}$
be sufficiently smooth.
According to form \eqref{varPhi}, we introduce the notations:
\begin{alignat}{4} \label{bas.relation}
\widebar{\varPhi}^{\pm}&:=\pm x_1,&\ \
\mathring{\varPhi}^{\pm}& :=\pm x_1+\mathring{\varPsi}^{\pm},&\ \ 
\mathring{\varPsi}^{\pm} &:=\chi(\pm x_1)\mathring{\varphi}(t,x'),\\
\label{N.ring}
\mathring{v}_{N}^{\pm}&:=\mathring{v}^{\pm}\cdot \mathring{N}^{\pm},&\ \
{\mathring{{F}}}_{jN}^{\pm}&:=\mathring{{F}}_{j}^{\pm}\cdot \mathring{N}^{\pm},&\ \
\mathring{N}^{\pm}&:=(1,-\partial_2\mathring{\varPhi}^{\pm},\ldots,-\partial_d\mathring{\varPhi}^{\pm})^{\mathsf{T}},
\end{alignat}
where $\chi\in C_0^{\infty}(\mathbb{R})$ satisfies \eqref{chi}, and
$\mathring{{F}}_{j}^{\pm}$  are the $j$-th columns of $\mathring{\bm{F}}^{\pm}$.

Perturbations $\mathring{V}^{\pm}:=\mathring{U}^{\pm}-\widebar{U}^{\pm}$
and $\mathring{\varphi}$ are supposed to satisfy
\begin{align} \label{bas.bound}
\|\mathring{V}^{\pm}\|_{H^6(\varOmega_T)}+\|\mathring{\varphi}\|_{H^6(\omega_T)}\leq K
\end{align}
for a sufficiently small positive constant $K\leq 1$, so that
\begin{align}
\pm \partial_1\mathring{\varPhi}^{\pm}\geq \frac{1}{2} \,\qquad \textrm{on $\xbar{\varOmega_T}$},
\label{bas.positive}
\end{align}
thanks to the Sobolev embedding $H^6(\varOmega_T)\hookrightarrow W^{3,\infty}(\varOmega_T)$.
We assume further that the basic state $(\mathring{U}^{\pm},\mathring{\varphi})$
satisfies constraints \eqref{rho.relation}, \eqref{TCD0.b}, and \eqref{inv3}--\eqref{inv5}, {\it i.e.},
\begin{subequations} \label{bas.cc}
\begin{alignat}{4}
\label{bas.c1}&\mathring{\rho}^{\pm}=(\det\mathring{\bm{F}}^{\pm})^{-1},\quad \mathring{p}^{\pm}=p(\mathring{\rho}^{\pm},\mathring{ S }^{\pm})
&  &\textrm{if } x_1\geq 0,\\[0.5mm]
\label{bas.c2}&\mathbb{B}\big(\mathring{U}^+,\mathring{U}^-,\mathring{\varphi}\big)=0
& &\textrm{if } x_1= 0,\\[0.5mm]
\label{bas.c3a}& [\mathring{\rho}\mathring{F}_{jN}]=0
\qquad \qquad \qquad \qquad  \; \textrm{for }  j=1,\ldots,d,
&  &\textrm{if } x_1= 0,\\[0.5mm]
\label{bas.c4}& [\mathring{\rho}\mathring{F}_{kN}\mathring{F}_{ij}-\mathring{\rho}\mathring{F}_{jN}\mathring{F}_{ik}]=0
 \qquad  \textrm{for }  i,j,k=1,\ldots,d,
&\quad &\textrm{if } x_1= 0,\\[0.5mm]
\label{bas.c3b}&  \mathring{F}_{jN}^{\pm}=0
\qquad \qquad \qquad \qquad \quad \ \textrm{for }  j=2,\ldots,d,
&  &\textrm{if } x_1= 0.
\end{alignat}
\end{subequations}
Under \eqref{bas.c3a} and \eqref{bas.c3b}, relations \eqref{bas.c4} are equivalent to
\begin{align}
[\mathring{F}_{j}]=0\qquad \textrm{on $\p\varOmega$, \ \ for $j=2,\ldots,d$}.
\label{bas.c5}
\end{align}
Moreover, we assume that the basic state satisfies
\begin{align} \label{bas.c6}
\Big(\p_t+\sum_{\ell=2}^d\mathring{v}_{\ell}^{\pm}\p_{\ell}\Big) \mathring{ {F}}_{j}^{\pm}
-\sum_{\ell=2}^d\mathring{ {F}}_{\ell j}^{\pm}\p_{\ell} \mathring{v}^{\pm}=0\qquad  \textrm{on $\p\varOmega$,\ \ for $j=2,\ldots,d$,}
\end{align}
which will play an important role in the estimate of the tangential derivatives,
especially in the proof of Lemma \ref{lem.key}.
As a matter of fact, constraints \eqref{bas.c6}
come from restricting the interior equations for $F_{j}^{\pm}$
on boundary $\p\varOmega$ and utilizing \eqref{bas.c2}--\eqref{bas.c3b}.

Before performing the linearization,
we give an alternative form of the boundary operator $\mathbb{B}$ defined in \eqref{B.bb.def},
which will be essential for providing the cancellation effect in the estimate of the tangential derivatives.
More precisely, by virtue of \eqref{inv5}, we observe
\begin{align}   \label{key}
\det \bm{F}^{\pm}= \varrho(\bm{F}^{\pm})^{-1}  F_{1N}^{\pm} \qquad\, \textrm{on $\p\varOmega$},
\end{align}
where $\varrho(\bm{F})$ is the scalar function defined by
\begin{align}    \label{varrho}
\varrho(\bm{F}):=\left\{
\begin{aligned}
&{F_{22}^{-1}}\qquad &&\textrm{if }  d=2,\\
&(F_{22}F_{33}-F_{23}F_{32})^{-1}\qquad &&\textrm{if } d=3.
\end{aligned}
\right.
\end{align}
In particular, for the background state \eqref{background}, we have
\begin{align}
\label{varrho.bar}
\varrho(\widebar{\bm{F}}^{\pm})=\left\{
\begin{aligned}
&{\widebar{F}_{22}^{-1}}\qquad &&\textrm{if }  d=2,\\
& \widebar{F}_{22}^{-1}\widebar{F}_{33}^{-1}\qquad &&\textrm{if } d=3.
\end{aligned}
\right.
\end{align}
Combine \eqref{key} with \eqref{rho.relation} and use \eqref{inv3} to obtain
\begin{align} \label{key1}
\rho^{\pm} F_{1N}^{\pm} = \varrho(\bm{F}^{+})= \varrho(\bm{F}^{-})\qquad\, \textrm{on $\p\varOmega$},
\end{align}
which yields
\begin{align}
\label{B.bb.def2}
\mathbb{B}(U^+,U^-,\varphi)=
\begin{pmatrix}
\partial_t\varphi-v_N^+ \\[0.5mm]
[v] \\[0.5mm]
 [p]-\varrho(\bm{F}^{+}) [{F}_{11}]\\[0.5mm]
[{F}_{11}\p_2\varphi+{F}_{21}]\\
\vdots\\
[{F}_{11}\p_d\varphi+{F}_{d1}]
\end{pmatrix}.
\end{align}
Furthermore, from \eqref{bas.c1}, \eqref{bas.c3b}, and \eqref{bas.c5}, we infer
\begin{align} \label{key1.bas}
\mathring{\rho}^{\pm} \mathring{F}_{1N}^{\pm} = \varrho(\mathring{\bm{F}}^{+})
= \varrho(\mathring{\bm{F}}^{-})\qquad\, \textrm{on $\p\varOmega$}.
\end{align}
As a result, constraints \eqref{bas.cc}--\eqref{bas.c6}
are equivalent to constraints  \eqref{bas.c1}--\eqref{bas.c2} and \eqref{bas.c3b}--\eqref{bas.c6}.

\subsection{Linearization and Main Theorem}

Let us now deduce the linearized problem based on identity \eqref{B.bb.def2}.
For this purpose, we consider families
$(U^{\pm}_{\epsilon},\,\varPhi^{\pm}_{\epsilon})
=(\mathring{U}^{\pm}+\epsilon V^{\pm},\,\mathring{\varPhi}^{\pm}+\epsilon \varPsi^{\pm})$,
where
\begin{align}  \label{varPsi.def}
 \varPsi^{\pm}(t,x):=\chi(\pm x_1)\psi(t,x').
\end{align}
The linearized operators are given by
\begin{align*}
\left\{\begin{aligned}
&\mathbb{L}'\big(\mathring{U}^{\pm},\mathring{\varPhi}^{\pm}\big)(V^{\pm},\varPsi^{\pm})
:=\left.\frac{\mathrm{d}}{\mathrm{d}\epsilon}
\mathbb{L}\big(U^{\pm}_{\epsilon}, \varPhi^{\pm}_{\epsilon}\big)\right|_{\epsilon=0},\\
&\mathbb{B}'\big(\mathring{U}^{\pm},\mathring{\varphi}\big)(V,\psi)
:=\left.\frac{\mathrm{d}}{\mathrm{d}\epsilon}
 \mathbb{B}(U^{+}_{\epsilon},U^{-}_{\epsilon},\varphi_{\epsilon})\right|_{\epsilon=0},
\end{aligned}\right.
\end{align*}
where  $V:=(V^+,V^-)^{\mathsf{T}}$,
and $\varphi_{\epsilon}:=\mathring{\varphi}+\epsilon \psi$
denotes the common trace of $\varPhi^{\pm}_{\epsilon}$ on boundary  $\partial\varOmega$.
A standard calculation leads to
\begin{align} \label{L.prime}
\mathbb{L}'(U,\varPhi)(V,\varPsi)=
L(U,\varPhi)V+\mathcal{C}(U,\varPhi)V - \frac{1}{\partial_1 \varPhi}\big(L(U,\varPhi)\varPsi \big)\partial_1 U,
\end{align}
where $L(U,\varPhi)$ is given in \eqref{L.def}, and
$\mathcal{C}(U,\varPhi)$ is the zero-th order operator defined by
\begin{align}\label{C.cal}
  \mathcal{C}(U,\varPhi)V:= V_{\ell}\frac{\p A_0(U)}{\p {U_{\ell}} } \p_t U
  +  V_{\ell}\frac{\p \widetilde{A}_1(U,\varPhi)}{\p {U_{\ell}} }  \p_1 U
+\sum_{i=2}^d V_{\ell} \frac{\p A_i(U)}{\p {U_{\ell}} } \p_i U .
\end{align}
Thanks to the alternative form \eqref{B.bb.def2},
we compute
\begin{align}
\mathbb{B}'\big(\mathring{U}^{\pm},\mathring{\varphi}\big)(V,\psi)=
\begin{pmatrix}
\big(\p_t+\sum_{i=2}^d\mathring{v}_i^+\p_i\big)\psi -v^+ \cdot \mathring{N} \\[0.5mm]
[v]\\[0.5mm]
[p]-\varrho(\mathring{\bm{F}}^{+}) [{F}_{11}]
-[\mathring{F}_{11}]\p_{F_{ij}}\varrho(\mathring{\bm{F}}^{+}) F_{ij}^{+} \\[0.5mm]
[{F}_{11}\p_2\mathring{\varphi}+{F}_{21}]+[\mathring{F}_{11}] \p_2\psi \\
\vdots\\
[{F}_{11}\p_d\mathring{\varphi}+{F}_{d1}]+[\mathring{F}_{11}] \p_d\psi
\end{pmatrix}
\label{B.prime}
\end{align}
with $\varrho(\bm{F})$ defined by \eqref{varrho}.

As in {\sc Alinhac} \cite{A89MR976971}, applying the ``good unknowns'':
\begin{align} \label{good}
\dot{V}^{\pm}:=V^{\pm}-\frac{\varPsi^{\pm}}{\partial_1 \mathring{\varPhi}^{\pm}}\partial_1\mathring{U}^{\pm},
\end{align}
we calculate ({\it cf.} \cite[Proposition\,1.3.1]{M01MR1842775})
\begin{align} \notag
&\mathbb{L}'(\mathring{U}^{\pm}, \mathring{\varPhi}^{\pm})(V^{\pm},\varPsi^{\pm})\\
&\quad =L(\mathring{U}^{\pm}, \mathring{\varPhi}^{\pm})\dot{V}^{\pm}
+\mathcal{C}( \mathring{U}^{\pm},\mathring{\varPhi}^{\pm})\dot{V}^{\pm}
+\frac{\varPsi^{\pm}}{\partial_1\mathring{\varPhi}^{\pm}}
\partial_1\big(L(\mathring{U}^{\pm},\mathring{\varPhi}^{\pm} )\mathring{U}^{\pm}\big).
\label{Alinhac}
\end{align}
In view of the nonlinear results obtained in \cite{A89MR976971, CS08MR2423311,CSW19MR3925528},
we neglect the zero-th order terms in $\varPsi^{\pm}$ of \eqref{Alinhac}
and consider the {\it effective linear problem}:
\begin{subequations} \label{ELP}
 \begin{alignat}{3}
 \label{ELP.1}
 &\mathbb{L}'_{e,\pm} \dot{V}^{\pm}=f^{\pm}
 &\qquad &\textrm{if } x_1>0,\\
  \label{ELP.2}
  &\mathbb{B}'_e (\dot{V},\psi)=g
 &&\textrm{if } x_1=0,\\
 \label{ELP.3}
 &(\dot{V},\psi)=0   &&\textrm{if } t<0,
 \end{alignat}
\end{subequations}
where we abbreviate $\dot{V}:=(\dot{V}^+,\dot{V}^-)^{\mathsf{T}}$ and denote
\begin{align}
\label{Le.bb}&\mathbb{L}'_{e,\pm}  \dot{V}^{\pm} :=
L(\mathring{U}^{\pm}, \mathring{\varPhi}^{\pm})\dot{V}^{\pm}
+\mathcal{C}( \mathring{U}^{\pm},\mathring{\varPhi}^{\pm})\dot{V}^{\pm},\\[2mm]
\label{Be.bb}&\mathbb{B}'_e (\dot{V},\psi) :=
\begin{pmatrix}
\big(\p_t+\sum_{i=2}^d\mathring{v}_i^+\p_i\big)\psi
-\dot{v}^+\cdot \mathring{N}^{+}-\partial_1 \mathring{v}_N^+ \psi  \\[1mm]
[\dot{v}]+  \psi (\p_1 \mathring{v}^++ \p_1 \mathring{v}^-) \\[1mm]
[\dot{p}]-\varrho(\mathring{\bm{F}}^{+}) [\dot{F}_{11}]
-[\mathring{F}_{11}]\p_{F_{ij}}\varrho(\mathring{\bm{F}}^{+}) \dot{F}_{ij}^{+} + \mathring{b}_1 \psi \\[1mm]
[\dot{{F}}_{11}\p_2\mathring{\varphi}+\dot{{F}}_{21}]+ [\mathring{F}_{11}]\p_2 \psi  +\mathring{b}_2 \psi \\
\vdots\\
[\dot{{F}}_{11}\p_d\mathring{\varphi}+\dot{{F}}_{d1}]+ [\mathring{F}_{11}]\p_d \psi +\mathring{b}_d \psi
\end{pmatrix},
\end{align}
with operators $L$ and $\mathcal{C}$ defined by \eqref{L.def} and \eqref{C.cal}.
In \eqref{Be.bb}, we denote
\begin{align*}
&\mathring{b}_1:= \p_1 \mathring{p}^++\p_1 \mathring{p}^--\varrho(\mathring{\bm{F}}^{+}) (\p_1 \mathring{F}_{11}^+ +\p_1 \mathring{F}_{11}^-)
-[\mathring{F}_{11}]\p_{F_{ij}}\varrho(\mathring{\bm{F}}^{+}) \p_1 \mathring{F}_{ij}^{+},\\
&\mathring{b}_i:=(\p_1 \mathring{{F}}_{11}^++\p_1 \mathring{{F}}_{11}^-)\p_i\mathring{\varphi}+
(\p_1 \mathring{{F}}_{i1}^++\p_1 \mathring{{F}}_{i1}^-)\qquad
\textrm{for $i=2,\ldots,d$}.
\end{align*}
The explicit form \eqref{Be.bb} results from the identity
$\mathbb{B}'_e(\dot{V},\psi) =\mathbb{B}'\big(\mathring{U}^{\pm},\mathring{\varphi}\big)(V,\psi).$
We write $\mathring{V}:=(\mathring{V}^+,\mathring{V}^-)^{\mathsf{T}}$,
$\mathring{\varPsi}:=(\mathring{\varPsi}^+,\mathring{\varPsi}^-)^{\mathsf{T}}$,
$\mathbb{L}'_e \dot{V}:=(\mathbb{L}'_{e,+} \dot{V}^+,\mathbb{L}'_{e,-}\dot{V}^-)^{\mathsf{T}}$,
${f}:=({f}^+,{f}^-)^{\mathsf{T}}$, etc.
to avoid overloaded expressions.

We now state the main result of this paper.

\begin{theorem}\label{thm2}
Let $T>0$ and $s\in \mathbb{N}_+$ be fixed.
Assume that the background state \eqref{background}--\eqref{background1}
	satisfies the stability condition{\rm :}
\begin{align}
\label{thm.H1}
&\frac{[\widebar{F}_{11}]}{ \widebar{F}_{11}^+} < 
\left\{
\begin{aligned}
& \widebar{F}_{22}^2 (\widebar{F}_{11}^+)^{-2}&\qquad & \textrm{if } d=2,\\[1mm]
&\widebar{C }^{-1} &&\textrm{if } d=3,
\end{aligned}\right.
\end{align}
with
\begin{align}
\notag \qquad \quad & \widebar{C}:= 
{   \begin{aligned}
	\Big(1+\frac{\widebar{F}_{22}^2}{\widebar{F}_{33}^2} \Big)^{1/2}
	\bigg\{
	\max(1,\, \frac{(\widebar{F}_{11}^+)^2}{\widebar{F}_{22}^{2}})
	+
	\max ( 1,\,\frac{(\widebar{F}_{11}^+)^2}{\widebar{F}_{33}^{2}})
	\frac{\widebar{F}_{33}}{\widebar{F}_{22}}
	\bigg\},
	\end{aligned}}
\end{align}
and that perturbations $(\mathring{V}^{\pm},\mathring{\varphi})\in H^{s+2}(\varOmega_T)\times H^{s+2}(\omega_T)$
satisfy constraints \eqref{bas.bound}--\eqref{bas.c6}.
Then there exist positive constants $K_0$ and $C_0$, uniformly bounded
even when $[\widebar{F}_{11}]$ tends to zero, such that,
for all $K \leq K_0$ and $(\dot{V}^{\pm},\psi)\in H^{s+1}(\varOmega_T)\times H^{s+3/2}(\omega_T)$
vanishing in the past,
\begin{alignat}{3}
\notag &
\|\dot{V}\|_{H^{1}(\varOmega_T)}+\|\psi\|_{H^{3/2}(\omega_T)}&\\ &\quad 
\leq C_0
\Big\{\|\mathbb{L}'_e \dot{V}\|_{H^{1}(\varOmega_T)}+
 \|\mathbb{B}'_e(\dot{V},\psi) \|_{H^{3/2}(\omega_T)}
\Big \} &\quad  \textrm{if } s=1,
	\label{thm2.est2}\\
\notag
&\|\dot{V}\|_{H^{s}(\varOmega_T)}+\|\psi\|_{H^{s+1/2}(\omega_T)}&\\
&\notag \quad 	\leq C_0
	\Big\{ \|\mathbb{L}'_e \dot{V}\|_{H^{s}(\varOmega_T)}+\|\mathbb{B}'_e(\dot{V},\psi)\|_{H^{s+1/2}(\omega_T)}&\\
	&\notag \qquad \qquad
	+\left(\|\mathbb{L}'_e \dot{V}\|_{H^3(\varOmega_T)} +\|\mathbb{B}'_e(\dot{V},\psi)\|_{H^{7/2}(\omega_T)}\right)\ &\\
	&\qquad \qquad\quad \times\left(\|\mathring{V}^{\pm}\|_{H^{s+2}(\varOmega_T)} +\|\mathring{\varphi}\|_{H^{s+2}(\omega_T)}  \right)
	\Big\} &\quad  {\textrm{if } s\geq 3.}
	\label{thm2.est}
	\end{alignat}
\end{theorem}

{Notice that the $H^2(\varOmega_T)\times H^{5/2}(\omega_T)$--estimate of $(\dot{V},\psi)$ follows from \eqref{thm2.est} with $s=3$.}
We remark that the tame estimates \eqref{thm2.est2}--\eqref{thm2.est} present
no loss of regularity with respect to the interior source term $\mathbb{L}'_e \dot{V}$,
while there is a loss of one derivative with respect to the boundary source term $\mathbb{B}'_e(\dot{V},\psi)$.
It should also be pointed out that estimate \eqref{thm2.est}
is with a {\it fixed} loss of regularity with respect to the coefficients,
which offers a way to establish
the nonlinear stability of thermoelastic contact discontinuities
by a suitable Nash--Moser iteration scheme.
The dropped terms in \eqref{Alinhac} will be considered as error terms
at each Nash--Moser iteration step.
Moreover, Theorem \ref{thm2} provides the tame estimates in the usual Sobolev spaces $H^s$
for the solutions and source terms vanishing in the past,
which corresponds to the nonlinear problem with zero initial data.
The case with general initial data is postponed to the nonlinear analysis that involves
the construction of so-called approximate solutions.


\section{Sobolev Functions and Notations}\label{sec.Preliminary}
In this section, we first state the definitions of some fractional Sobolev spaces and norms for self-containedness.
Then we prove two important estimates for the traces of  $H^1(\mathbb{R}^{n+1}_+)$--functions
on the hyperplane $\{y\in\mathbb{R}^{n+1}:y_1=0\}$ with $\mathbb{R}_+^{n+1} :=\{y\in\mathbb{R}^{n+1}:y_1>0 \}$.
We also present the Moser-type calculus inequalities and the notations for later use.

\subsection{Fractional Sobolev Spaces and Norms}

We first give the definitions of the Sobolev spaces and norms for general domains;
see also {\sc Tartar} \cite{T07MR2328004} for more details.

\begin{definition}\label{def.Sobolev1}
	Let $\mathcal{O}$ be an open subset of $\mathbb{R}^n$ with $n\in\mathbb{N}_+$.
	For every nonnegative integer $m$, the Sobolev space $H^m(\mathcal{O})$ is defined by
	\begin{align*}
	H^m(\mathcal{O}):=\{u\in L^2(\mathcal{O})\,:\,\p^{\alpha} u\in L^2(\mathcal{O})\
	\textrm{ for all } \alpha\in \mathbb{N}^n \textrm{ with } |\alpha|\leq m \},
	\end{align*}
	equipped with the norm{\rm :}
	\begin{align} \label{norm.def}
	\|u\|_{H^m(\mathcal{O})}:=
	\Big(\sum_{|\alpha|\leq m} \int_{\mathcal{O}} |\p^{\alpha} u(y)|^2 \mathrm{d}y \Big)^{\frac{1}{2}},
	\end{align}
	where $\alpha:=(\alpha_1,\ldots,\alpha_n)\in\mathbb{N}^n$ denotes a multi-index,
	\begin{align*}
	|\alpha|:=\alpha_1+\cdots+\alpha_n,\qquad\,\,
	\p^{\alpha} u(y):=\frac{\p^{|\alpha|}}{\p y_1^{\alpha_1}\cdots \p y_n^{\alpha_n}}u(y).
	\end{align*}
	For each real number $s\geq 0$ that is not an integer, the fractional Sobolev space
    $H^s(\mathcal{O})$ and its norm $\|\cdot\|_{H^s(\mathcal{O})}$ can be defined by interpolation
    between $H^{\lfloor s\rfloor}(\mathcal{O})$ and $H^{\lfloor s\rfloor+1}(\mathcal{O})$
    $($see {\rm \cite[\S 22]{T07MR2328004}}$)$,
	where $\lfloor s\rfloor$ denotes the greatest integer less than or equal to $s$.
\end{definition}
Next we present an alternative definition of the Sobolev space $H^s(\mathbb{R}^n)$ via the Fourier transform.

\begin{definition}\label{def.Sobolev2}
	For each real number $s\geq 0$,  we define
	\begin{align*}
	H^s(\mathbb{R}^n):=\left\{u\in L^2(\mathbb{R}^n)\,:\,|\xi|^s \mathcal{F} u (\xi)\in L^2(\mathbb{R}^n) \right\},
	\end{align*}
	where $\mathcal{F} u$ denotes the Fourier transform of $u$; in particular,
	\begin{align*}
	\mathcal{F} u(\xi):=\int_{\mathbb{R}^n} u(y) \mathrm{e}^{-2\pi  \mathrm{i} \, y\cdot\xi}\,\mathrm{d}y \qquad
	\textrm{for $u\in L^1(\mathbb{R}^n)$}.
	\end{align*}
	The negative-order Sobolev spaces $H^{-s}(\mathbb{R}^n)$ are defined by duality as
	$$
H^{-s}(\mathbb{R}^n):=(H^s(\mathbb{R}^n) )'\qquad\, \textrm{for all $s\geq 0$}.
$$
\end{definition}

Let us recall that
\begin{alignat}{3} \label{Fourier.id1}
\mathcal{F}(\p^{\alpha} u )&=(2\pi \mathrm{i}\, \xi)^{\alpha}\mathcal{F}u &&
\textrm{for all $u\in L^2(\mathbb{R}^n)$},\\
\label{Fourier.id2}
\int_{\mathbb{R}^n} u\;\! \xbar{w}\,\mathrm{d}y 
&=\int_{\mathbb{R}^n} \mathcal{F}u \, \xbar{\mathcal{F} w}\,\mathrm{d}y &\qquad&
  \textrm{for all $u,w\in L^2(\mathbb{R}^n)$},
\end{alignat}
where $\xbar{w}$  denotes the complex conjugation of $w$.
Using identities \eqref{Fourier.id1}--\eqref{Fourier.id2},
we can show that Definition \ref{def.Sobolev1} is equivalent to Definition \ref{def.Sobolev2}
for $s\in\mathbb{N}$ and $\mathcal{O}=\mathbb{R}^n$.
Furthermore, we refer to \cite{T07MR2328004}  for
the equivalence between Definition \ref{def.Sobolev1} and Definition \ref{def.Sobolev2}
for fractional Sobolev spaces $H^s(\mathbb{R}^n)$.

\subsection{Traces on the Hyperplane}

The following lemma is to characterize
the traces of $H^{1}(\mathbb{R}_+^{n+1} )$--functions on the hyperplane $\{y\in\mathbb{R}^{n+1}:y_1=0\}$.
\begin{lemma}\label{lem.trace1}
	Any function $u\in H^{1}(\mathbb{R}_+^{n+1} )$
	has a trace $w$ on the hyperplane $\{y\in\mathbb{R}^{n+1}:y_1=0\}$  such that
	$w$ belongs to $H^{1/2}(\mathbb{R}^{n} )$ and satisfies
	\begin{align}  \label{trace.est1}
	\int_{\mathbb{R}^n}(1+4\pi^2|\xi'|^2)^{\frac{1}{2}} \left|\mathcal{F}w(\xi' ) \right|^2  \mathrm{d} \xi'
	\leq   \|u \|_{H^1(\mathbb{R}_+^{n+1} )}^2 .
	\end{align}
\end{lemma}
\begin{proof}
	We first extend $u\in H^{1}(\mathbb{R}_+^{n+1})$ to be defined in $\mathbb{R}^{n+1}$ by setting
	\begin{align*}
	E u(y_1,\,y'):=
	\left\{\begin{aligned}
	& u(y_1,\,y')\qquad && \textrm{if } y_1>0,\\
	& u(-y_1,\,y')\qquad && \textrm{if } y_1<0,
	\end{aligned}\right.
	\end{align*}
	for all $y':=(y_2,\ldots\,,y_{n+1})\in\mathbb{R}^n$.
In view of \cite[Lemma 12.5]{T07MR2328004},
	we obtain that $E u\in H^1(\mathbb{R}^{n+1})$.
	A direct computation yields
	\begin{align} \label{exten.est}
	\left\{\begin{aligned}
	&\|E u\|_{L^2(\mathbb{R}^{n+1})}\leq \sqrt{2}\|u\|_{L^2(\mathbb{R}_+^{n+1})},\\
	&\|\p_y E u\|_{L^2(\mathbb{R}^{n+1})}\leq \sqrt{2}\|\p_{y} u\|_{L^2(\mathbb{R}_+^{n+1})}.
	\end{aligned}\right.
	\end{align}
	By virtue of \eqref{exten.est}, it suffices to prove that, for all rapidly decreasing $C^{\infty}$--function $\tilde{u}\in \mathscr{S}(\mathbb{R}^{n+1}) $,
	\begin{align}  \label{trace.est1'}
	\int_{\mathbb{R}^n}(1+4\pi^2|\xi'|^2)^{\frac{1}{2}} \left|\mathcal{F}w(\xi') \right|^2  \mathrm{d} \xi'
	\leq   \frac{1}{2}\|\tilde{u} \|_{H^1(\mathbb{R}^{n+1} )}^2
	\end{align}
	with $w$ defined by $w(y'):=\tilde{u}(0,y')$ for $y'\in\mathbb{R}^n$.
	According to \cite[Lemma 15.11]{T07MR2328004}, we have
	\begin{align*}
	\mathcal{F} w(\xi')=\int_{\mathbb{R}}\mathcal{F} \tilde{u}(\xi_1,\xi')\mathrm{d}\xi_1 \qquad\,\textrm{for $\xi'\in\mathbb{R}^n$},
	\end{align*}
	which, along with the Cauchy--Schwarz inequality, implies
	\begin{align*}
	|\mathcal{F} w(\xi')|^{2}\leq \int_{\mathbb{R}} (1+4\pi^2|\xi|^2)
	\left|\mathcal{F}\tilde{u}(\xi_1,\xi')\right|^2\, \mathrm{d} \xi_1
	\int_{\mathbb{R}}  \frac{\mathrm{d}\xi_1 }{1+4\pi^2|\xi|^2}.
	\end{align*}
Performing the change of variable: $\xi_1=t{(1+4\pi^2|\xi'|^2)^{1/2}}$, we obtain
	\begin{align*}
	\int_{\mathbb{R}}  \frac{\mathrm{d}\xi_1 }{1+4\pi^2|\xi|^2}
	=(1+4\pi^2|\xi'|^2)^{-\frac{1}{2}}\int_{\mathbb{R}}  \frac{\mathrm{d} t}{1+4\pi^2 t^2}
	=\frac{1}{2}(1+4\pi^2|\xi'|^2)^{-\frac{1}{2}}.
	\end{align*}
	Combine the two estimates above to infer
\begin{align}\notag
	\int_{\mathbb{R}^n}(1+4\pi^2|\xi'|^2)^{\frac{1}{2}} \left|\mathcal{F}w(\xi') \right|^2  \mathrm{d} \xi'
	\leq   \frac{1}{2}\int_{\mathbb{R}^{n+1}} (1+4\pi^2|\xi|^2) \left|\mathcal{F}\tilde{u}(\xi_1,\xi')\right|^2\mathrm{d} \xi,
\end{align}
from which  we can deduce \eqref{trace.est1'} by means of \eqref{Fourier.id1}--\eqref{Fourier.id2}.
\qed\end{proof}

The next lemma will be crucial for reducing the boundary integrals to the volume ones
in the estimate of tangential derivatives.
\begin{lemma} \label{lem.trace2}
	If $n\in\mathbb{N}_+$  and $u_1,u_2 \in H^{1}(\mathbb{R}_+^{n+1})$, then
	\begin{align} \notag 
	&\left| \int_{\mathbb{R}^n} {u}_1 \frac{\p  u_2}{\p y_j}(0,y')\,  \mathrm{d} y' \right|
	\\
	& \label{trace.est2} \qquad \leq \|u_1 \|_{H^1(\mathbb{R}_+^{n+1} )}\|u_2 \|_{H^1(\mathbb{R}_+^{n+1} )}\quad\textrm{for } j=2,\ldots,n+1.
	\end{align}
\end{lemma}
\begin{proof}
	In light of  \eqref{Fourier.id1}--\eqref{Fourier.id2},   we have
	\begin{align*}
	\left| \int_{\mathbb{R}^n}{u}_1 \frac{\p  u_2}{\p y_j}(0,y')\,  \mathrm{d} y' \right|
	&=\left| \int_{\mathbb{R}^n}  \xbar{\mathcal{F} {\widebar{u}}_1}\,
	2\pi \mathrm{i}\,\xi_j \mathcal{F}  u_2(0,\xi')\,  \mathrm{d} \xi ' \right|        \\
	&\leq \left( \int_{\mathbb{R}^n}  (1+4\pi^2 |\xi'|^2)^{\frac{1}{2}}
	\left|\mathcal{F}u_1 (0,\xi ')\right|^2\,  \mathrm{d} \xi ' \right)^{\frac{1}{2}}\\
	&\quad  \times \left( \int_{\mathbb{R}^n}   (1+4\pi^2 |\xi'|^2)^{-\frac{1}{2}} 4\pi^2 \xi_j^2
	\left|\mathcal{F}u_2 (0,\xi ')\right|^2\,  \mathrm{d} \xi '\right)^{\frac{1}{2}},
	\end{align*}
which, combined with  \eqref{trace.est1}, leads to \eqref{trace.est2}.
\qed\end{proof}

\subsection{Moser-Type Calculus Inequalities}

We present the following Moser-type calculus inequalities
that will be repeatedly employed in the subsequent analysis.

\begin{lemma}[Moser-type calculus inequalities]
	Let $\mathcal{O}$ be an open subset of $\mathbb{R}^n$ with Lipschitz boundary for $n\in\mathbb{N}_+$.
	Assume that  $b\in C^{\infty}(\mathbb{R})$ and
	$u,w\in L^{\infty}(\mathcal{O}) \cap H^m(\mathcal{O})$ for an integer $m>0$.
	
	\noindent {\rm (a)}
	If $|\alpha|+|\beta|\leq m$ and $b(0)=0$, then
	\begin{alignat}{3}
	\label{Moser1}
	&\|\p^{\alpha} u\p^{\beta} w\|_{L^{2}}+ \| uw\|_{H^m}
	\leq C \|u\|_{H^{m}}\|w\|_{L^{\infty}}+C \|u\|_{L^{\infty}}\|w\|_{H^{m}},\\
	\label{Moser2}
	&\|b(u)\|_{H^{m}}\leq C(M_0)\|u\|_{H^{m}}.
	\end{alignat}
	
	\noindent {\rm (b)}
	If $|\alpha+\beta+\gamma|\leq m$, then
	\begin{align}    \label{Moser3}
	\|\p^{\alpha}[\p^{\beta},b(u)]\p^{\gamma}w\|_{L^{2}}
	\leq        C(M_0)
	\big(\|w\|_{H^{m}}+\|u\|_{H^{m}}\|w\|_{L^{\infty}}\big).
	\end{align}
	\mbox{\quad\ }Moreover, if $u\in W^{1,\infty}(\mathcal{O})$, then
	\begin{align}
	\|\p^{\alpha}[\p^{\beta},b(u)]\p^{\gamma}w\|_{L^{2}}
	&\leq      C(M_1)
	\big(\|w\|_{H^{m-1}}+\|u\|_{H^{m}}\|w\|_{L^{\infty}}\big). \label{Moser4}
	\end{align}
	Here we write $\|\cdot\|_{L^{p}}:=\|\cdot\|_{L^{p}(\mathcal{O})}$,
	$\|\cdot\|_{H^{m}}:=\|\cdot\|_{H^{m}(\mathcal{O})}$,
	and $\|\cdot\|_{W^{1,\infty}}:=\|\cdot\|_{W^{1,\infty}(\mathcal{O})}$
	for notational simplicity, and
	$M_0$ and $M_1$ are positive constants such that
	$\|u\|_{L^{\infty}}\leq M_0$ and $\|u\|_{W^{1,\infty}}\leq M_1$.
	 As usual,
	$$
     [a,b]c:=a(bc)-b(ac)
    $$
	denotes the notation of commutator.
\end{lemma}

\begin{proof}
	We refer to {\sc Stein} \cite[\S VI.3--\S VI.4]{S70MR0290095}
	for reducing the analysis of this lemma to the case when $\mathcal{O}=\mathbb{R}^n$.
	See {\sc Alinhac--G{{\'e}}rard} \cite[pp.\;84--89]{AG07MR2304160}
	for the detailed proof of assertion (a) when $\mathcal{O}=\mathbb{R}^n$.
	Here we give the proof of \eqref{Moser3}--\eqref{Moser4} by means of \eqref{Moser1}--\eqref{Moser2}.
	It follows from \eqref{Moser1} that
	\begin{align}
	\notag \big\|\p^{\alpha}[\p^{\beta},u]\p^{\gamma}w\big\|_{L^{2}}
	&\le C \sum_{\alpha'\leq \alpha} \sum_{\substack{0< \beta'\leq \beta}}
	\big\|\p^{\alpha'} \p^{\beta'} u\,\p^{\alpha-\alpha'} \p^{\beta-\beta'}\p^{\gamma} w\big\|_{L^{2}}\\
	\label{Moser3'}
	&\le C \|u\|_{H^{m}}\|w\|_{L^{\infty}}+C \| u\|_{L^{\infty}}\|w\|_{H^{m}},\\[3mm]
	\notag \big\|\p^{\alpha}[\p^{\beta},u]\p^{\gamma}w\big\|_{L^{2}}
	&\le C \sum_{\alpha'\leq \alpha} \sum_{\substack{\beta''\leq \beta'\leq \beta\\ |\beta''|=1}}
	\big\|\p^{\alpha'} \p^{\beta'-\beta''}\big(\p^{\beta''} u\big)\,\p^{\alpha-\alpha'} \p^{\beta-\beta'}\p^{\gamma} w\big\|_{L^{2}}\\
	\label{Moser4'}
	&\le C \|u\|_{H^{m}}\|w\|_{L^{\infty}}+C \|u\|_{W^{1,\infty}}\|w\|_{H^{m-1}}.
	\end{align}
	Combining  \eqref{Moser4'} with \eqref{Moser2} yields
	\begin{align*}
	&\|\p^{\alpha}[\p^{\beta},b(u)]\p^{\gamma}w\|_{L^{2}} 
    =\|\p^{\alpha}[\p^{\beta},b(u)-b(0)]\p^{\gamma}w\|_{L^{2}}\\
	&\quad \le C \|b(u)-b(0)\|_{H^{m}}\|w\|_{L^{\infty}}+C\|b(u)-b(0)\|_{W^{1,\infty}}\|w\|_{H^{m-1}}\\
	&\quad \leq C(M_1) (\|w\|_{H^{m-1}}+\|u\|_{H^{m}}\|w\|_{L^{\infty}}).
	\end{align*}
	Inequality \eqref{Moser3} can be proved similarly from \eqref{Moser2} and \eqref{Moser3'}.
\qed\end{proof}

\subsection{Notations}
For convenience, we collect the following notations.
\begin{itemize}
\item[(i)] We will use letter $C$ to denote any universal positive constant.
Symbol $C(\cdot)$ denotes any generic positive constant
depending only on the quantities listed in the parenthesis.
Notice that constants $C$ and $C(\cdot)$ may vary at different occurrence.
We denote $A\lesssim B$ (or $B\gtrsim A$) if
$A \leq CB$ holds uniformly for some universal  positive constant $C$.
Symbol $A\sim B$ means that both $A\lesssim B$ and $B\lesssim  A$ hold.
	
\item[(ii)] Letter $d$ always denotes the spatial dimension.
	{Both the two and three dimensional cases ($d=2,3$) are considered.}
	Symbol $\varOmega$ stands for the half-space $\{x\in\mathbb{R}^d: x_1>0 \}$.
	Boundary $\p\varOmega:=\{x\in\mathbb{R}^d: x_1=0 \}$ is identified to $\mathbb{R}^{d-1}$.
	We write $\varOmega_t:=(-\infty,t)\times \varOmega$
	and $\omega_t:=(-\infty,t)\times\p\varOmega.$
	
\item[(iii)]  Symbol $\mathrm{D}$ will be used to denote
	\begin{align*}
	\mathrm{D}:=(\p_t,\p_1,\ldots,\p_d),
	\end{align*}
where $\p_t:=\frac{\p }{\p t}$ and $ \p_{\ell}:=\frac{\p }{\p x_{\ell}}$ are the partial differentials.
For any multi-index $\alpha=(\alpha_0,\alpha_1,\ldots,\alpha_d)\in \mathbb{N}^{d+1}$,
we define
\begin{align*}
	\mathrm{D}^{\alpha}:=\p_t^{\alpha_0}\p_1^{\alpha_1}\cdots\p_d^{\alpha_d}, \qquad 
	|\alpha|:=\alpha_0+\alpha_1+\cdots+\alpha_d.
\end{align*}
For $m\in\mathbb{N}$, we denote $\mathrm{D}^{m} :=\{\mathrm{D}^{\alpha}: |\alpha|= m\}$.
	
\item[(iv)]  Denote $\mathrm{D}_x:=(\p_1,\ldots,\p_d)$ as the gradient vector and
$\mathrm{D}_{\rm tan}:=(\p_t,\p_2,\ldots,\p_d)$ as the tangential derivative.
 We write
\begin{alignat*}{2}
	\mathrm{D}_{\rm tan}^{\beta}:=\p_t^{\beta_0}\p_2^{\beta_2}\cdots\p_d^{\beta_d},\qquad 
	|\beta|:=\beta_0+\beta_2+\cdots+\beta_d,
\end{alignat*}
for any multi-index $\beta=(\beta_0,\beta_2,\ldots,\beta_d)\in \mathbb{N}^d$.
We denote $\mathrm{D}_{x'}:=(\p_2,\ldots,\p_d)$.
\item[(v)]
For any nonnegative integer $m$, we introduce
\begin{align} \label{VERT.norm}
&\VERT u(t) \VERT_{m}
:=\Big(\sum_{|\alpha|\leq m} \|\mathrm{D}^{\alpha} u(t) \|_{L^{2} (\varOmega)}^2 \Big)^{1/2},\\
\label{VERT.tan}
&\VERT u(t)\VERT_{{\rm tan},\,m}
:= \Big(\sum_{|\beta|\leq m}\| \mathrm{D}_{\rm tan}^{\beta} u(t) \|_{L^2(\varOmega)}^2\Big)^{1/2},\\[1.5mm]
\label{C.ring}
&\mathring{\rm C}_m
:={1+\| (\mathring{V},  \mathring{\varPsi})\|_{H^{m}(\varOmega_T)}^2,}
\end{align}
so that
our formulas will be much shortened in the calculations.
	
\item[(vi)]  Recall
the partial differentials with respect to functions $\mathring{\varPhi}^{\pm}$
from the notations in \eqref{differential} to obtain
	\begin{align*}
	\p_t^{\mathring{\varPhi}^{\pm}}+\mathring{v}^{\pm}_{\ell}\p_{\ell}^{\mathring{\varPhi}^{\pm}}
	=\p_t+\mathring{w}^{\pm}_{\ell}\p_{\ell},
	\end{align*}
	where
\begin{align} \label{w.ring}
\mathring{w}^{\pm}_1
:=\frac{1}{\p_1\mathring{\varPhi}^{\pm}}(\mathring{v}_{N}^{\pm}-\p_t\mathring{\varPhi}^{\pm}),
    \qquad\,\,
	\mathring{w}^{\pm}_i:=\mathring{v}_i^{\pm} \quad\textrm{for $i=2,\ldots,d$}.
\end{align}
In view of condition \eqref{bas.c2}, we have
\begin{align} \label{w1.ring.bdy}
\mathring{w}^{\pm}_1=0 \qquad \textrm{on $\p\varOmega$}.
\end{align}
Let us  define
\begin{align} \label{partial0}
\p_0:=\p_t+\sum_{i=2}^{d}\mathring{v}_i^+\p_i  \qquad \textrm{on $\xbar{\varOmega_T}$},
\end{align}
which coincides with $\p_t+\mathring{w}^{\pm}_{\ell}\p_{\ell}$
on boundary $\p\varOmega$ as a result of \eqref{bas.c2} and \eqref{w1.ring.bdy}.

\item[(vii)]  For any nonnegative integer $m$,
a generic and smooth matrix-valued function of
$\{(\mathrm{D}^{\alpha} \mathring{V},\mathrm{D}^{\alpha}\mathring{\varPsi}):|\alpha|\leq m\}$
is denoted by $\mathring{\rm c}_m$,
and by $\underline{\mathring{\rm c}}_m$ if it vanishes at the origin.
The exact forms of $\mathring{\rm c}_m$ and $\underline{\mathring{\rm c}}_m$
may be different at each occurrence.
For instance,  the equations for $p^{\pm}$ in \eqref{ELP.1} can be written as
	\begin{align}
	\notag
	(\p_t+\mathring{w}_{\ell}^{\pm}\p_{\ell})p^{\pm}+\mathring{\rho}^{\pm}\mathring{c}_{\pm}^2\p_{\ell}^{\mathring{\varPhi}^{\pm}} v_{\ell}^{\pm}
	=\mathring{\rm c}_0 {f}+\underline{\mathring{\rm c}}_1 V,
	\end{align}
since $\mathcal{C}(\mathring{U},\mathring{\varPhi})$ are
$C^{\infty}$--functions of $(\mathring{V},\mathrm{D}\mathring{V}, \mathrm{D} \mathring{\varPsi})$ vanishing at the origin.
\end{itemize}

\section{Partial Homogenization and Reformulation}\label{sec.Homogenization}

It is more convenient to reformulate problem \eqref{ELP} into the case with homogeneous boundary conditions.
 To this end,
 noting that $g=\mathbb{B}'_{e}(\dot{V},\psi)\in H^{s+1/2}(\omega_T)$ vanishes in the past,
 we employ the trace theorem to find a regular function ${V}_{\natural}=({V}_{\natural}^+,{V}_{\natural}^-)^{\mathsf{T}}\in H^{s+1}(\varOmega_T)$
 vanishing in the past such that
 \begin{align}
 {\mathbb{B}'_e({V}_{\natural},0)\big|_{\omega_T}=g},\quad \ 
 \|{V}_{\natural} \|_{H^{m}(\varOmega_T)}\lesssim \|g\|_{H^{m-1/2}(\omega_T)}\ \  \mbox{for $m=1,\ldots, s+1.$}
 \label{V.natural.est}
 \end{align}
 Then the new unknowns $V_{\flat}^{\pm}:=\dot{V}^{\pm}-{V}_{\natural}^{\pm}$ solve problem \eqref{ELP} with zero boundary source term
 and new internal source terms $\tilde{f}^{\pm}$:
 \begin{subequations} \label{ELP2}
  \begin{alignat}{3}   \label{ELP2.1}
  &\mathbb{L}'_{e,\pm} V^{\pm}=\tilde{f}^{\pm}\qquad\, &&\textrm{if } x_1>0,\\
  &\mathbb{B}'_e (V,\psi)=0\qquad  &&\textrm{if } x_1=0, \label{ELP2.2}\\
  &({V},\psi)=0\qquad  &&\textrm{if } t<0, \label{ELP2.3}
  \end{alignat}
 \end{subequations}
where we have dropped index ``$\flat$'' for simplicity of notation,
operators $\mathbb{L}'_{e,\pm}$ and $\mathbb{B}'_e$ are defined by \eqref{Le.bb}--\eqref{Be.bb},
and
\begin{align} \label{f.tilde}
\tilde{f}^{\pm}:=f^{\pm}-L(\mathring{U}^{\pm},\mathring{\varPhi}^{\pm}){V}_{\natural}^{\pm}-\mathcal{C}(\mathring{U}^{\pm},\mathring{\varPhi}^{\pm}){V}_{\natural}^{\pm}    .
\end{align}

We introduce new unknowns $W^{\pm}$ in order to distinguish the noncharacteristic variables
from the others for problem \eqref{ELP2}. More precisely, we define
\begin{align}
\label{W.def.d}
\left\{ \begin{aligned}
&W_1^{\pm}:=p^{\pm},\quad 
W_2^{\pm}:=v^{\pm}\cdot \mathring{N}^{\pm},\quad 
W_{j+1}^{\pm}:=v_j^{\pm},\\
&W_{d+2}^{\pm}:=p^{\pm}-\mathring{\rho}^{\pm}\mathring{{F}}_{1N}^{\pm}{F}_{11}^{\pm},\quad
W_{d+j+1}^{\pm}:=\partial_j \mathring{\varPhi}^{\pm}{F}_{11}^{\pm}+ {F}_{j1}^{\pm},\\
& W_{jd+i+1}^{\pm}:=F_{ij}^{\pm},\quad
W_{d^2+d+2}^{\pm}:=S^{\pm}
\quad \textrm{for }i=1,\ldots,d,\ j=2,\ldots,d,
\end{aligned}\right.
\end{align}
where $\mathring{N}^{\pm}$ and $\mathring{{F}}_{1N}^{\pm}$ are given in \eqref{N.ring}.
Equivalently, we set
\begin{align}
\notag 
W^{\pm}:=\mathring{J}_{\pm}^{-1}V^{\pm},
\qquad  \mathring{J}_{\pm}:={J}(\mathring{U}^{\pm}, \mathring{\varPhi}^{\pm}),
\end{align}
where $J({U},  {\varPhi})$ is the $C^{\infty}$--function of $(U,\mathrm{D}\varPhi)$ defined as
 \begin{align} \notag
 {J}({U},  {\varPhi}):=
 {\small \begin{pmatrix}1 & 0 & 0 & 0 & 0 & 0  \\[2mm]
  0 & 1 & \partial_2{\varPhi}    & 0 & 0 & 0  \\[2mm]
  0 & 0 & 1 & 0 & 0 & 0    \\[2mm]
  \dfrac{1}{ {\rho}{{F}}_{1N}}  & 0 & 0 & -\dfrac{1}{{\rho}{{F}}_{1N}} & 0 & 0  \\[3.5mm]
  -\dfrac{\partial_2{\varPhi}}{   {\rho}{{F}}_{1N}} & 0 & 0  & \dfrac{\partial_2{\varPhi}}{{\rho}{{F}}_{1N}}  & 1 & 0  \\[3.5mm]
  \w0 & \w0 & \w0 & \w0 & \w0 & \w{\bm{I}_3}
  \end{pmatrix}}
 \qquad\,\, \textrm{if $d=2$},
 \end{align}
and
\begin{align} \notag
  {J}({U},  {\varPhi}):=
   {\small \begin{pmatrix}1 & 0 & 0 & 0 & 0 & 0 & 0 & 0  \\[2mm]
  0 & 1 & \partial_2{\varPhi}  & \partial_3{\varPhi}  & 0 & 0 & 0 & 0  \\[2mm]
  0 & 0 & 1 & 0 & 0 & 0 & 0 & 0  \\[2mm]
  0 & 0 & 0 & 1 & 0 & 0 & 0 & 0  \\[3mm]
  \dfrac{1}{ {\rho}{{F}}_{1N}} & 0 & 0 & 0 & -\dfrac{1}{{\rho}{{F}}_{1N}}& 0 & 0 & 0  \\[3.5mm]
  -\dfrac{\partial_2{\varPhi}}{   {\rho}{{F}}_{1N}} & 0 & 0 & 0 & \dfrac{\partial_2{\varPhi}}{{\rho}{{F}}_{1N}}  & 1& 0 & 0  \\[3.5mm]
  -\dfrac{\partial_3{\varPhi}}{   {\rho}{{F}}_{1N}} & 0 & 0 & 0 &
  \dfrac{\partial_3{\varPhi}}{{\rho}{{F}}_{1N}}  & 0 & 1& 0  \\[3.5mm]
  \w0 & \w0 & \w0 & \w0 & \w0 & \w0 & \w0 & \w{\bm{I}_7}
  \end{pmatrix}}
 \qquad\,\, \textrm{if $d=3$}.
  \end{align}

In terms of the new unknowns $W^{\pm}$, we obtain the equivalent formulation of problem \eqref{ELP2.1} as
 \begin{align} \label{ELP3.a}
 \mathring{\mathcal{A}}_0^{\pm}\partial_t W^{\pm}+\sum_{j=1}^{d}\mathring{\mathcal{A}}_j^{\pm}\partial_j W^{\pm}+\mathring{\mathcal{A}}_4^{\pm}W^{\pm}=\mathring{J}_{\pm}^{\mathsf{T}}\tilde{f}^{\pm}
 \qquad\,\, \textrm{in }\ \varOmega_T,
 \end{align}
 where
 $\mathring{\mathcal{A}}_{i}^{\pm}:=\mathcal{A}_i(\mathring{U}^{\pm}, \mathring{\varPhi}^{\pm})$,
 for $i=0,\ldots,d$, with
 \begin{align} \label{A.cal}
\left\{
\begin{aligned}
 \mathcal{A}_1(U,\varPhi)&:=J({U},  {\varPhi})^{\mathsf{T}}\widetilde{A}_1({U},  {\varPhi})J({U},  {\varPhi}),\\[1.5mm]
  \mathcal{A}_j(U,\varPhi)&:={J}({U},  {\varPhi})^{\mathsf{T}}A_j(U)J({U},  {\varPhi})\qquad\textrm{for }  j=0,2,\ldots,d,\\[1.5mm]
 \mathcal{A}_4({U},  {\varPhi})&:=J({U},  {\varPhi})^{\mathsf{T}}\! \left(L({U},  {\varPhi})J({U},  {\varPhi})+
\mathcal{C}({U},  {\varPhi})J({U},  {\varPhi})\right).
\end{aligned}\right.
 \end{align}
Note that the coefficient matrices $\mathring{\mathcal{A}}_j^{\pm}$,
for $j=0,\ldots,d$, are symmetric, and $\mathring{\mathcal{A}}_0^{\pm}$ are positive definite.
In particular, a straightforward calculation gives
\begin{align}
\label{A0.cal}
\mathcal{A}_0(\widebar{U}^{\pm},\widebar{\varPhi}^{\pm})
&=
  {\small \begin{pmatrix}
 \dfrac{1}{\bar{\rho}^{\pm} \bar{c}_{\pm}^2 } +\dfrac{1}{\bar{\rho}^{\pm} (\widebar{F}_{11}^{\pm})^2}
 &0 &-\dfrac{1}{\bar{\rho}^{\pm} (\widebar{F}_{11}^{\pm})^2} &0 &0\\[4mm]
 0&\bar{\rho}^{\pm}\bm{I}_{d}& 0&0&0\\[2mm]
 -\dfrac{1}{\bar{\rho}^{\pm} (\widebar{F}_{11}^{\pm})^2} &0 &
 \dfrac{1}{\bar{\rho}^{\pm} (\widebar{F}_{11}^{\pm})^2} &0&0\\[4mm]
 0&0&0&\bar{\rho}^{\pm}\bm{I}_{d^2-1}&0\\[3mm]
 \w{0}&\w{0}&\w{0}&\w{0}&\w{1}
  \end{pmatrix}},\\[2mm]
\label{A2.cal}
\mathcal{A}_2(\widebar{U}^{\pm},\widebar{\varPhi}^{\pm})
&=
{\small \begin{pmatrix}
	0&\bm{e}_2^{\mathsf{T}} &0 &0 &0\\[1mm]
	\bm{e}_2& \bm{O}_{d}&  \bm{O}_{d}& -\bar{\rho}^{\pm}  \widebar{F}_{22}\bm{I}_{d}&0\\[1mm]
	0 & \bm{O}_{d} &	 &&\\[1mm]
	0& -\bar{\rho}^{\pm}  \widebar{F}_{22}\bm{I}_{d}& & \bm{O}_{d^2+1}&\\[1mm]
	\w{0}&\w{0}&\w{\ }&\w{\ } &\w{\ }
	\end{pmatrix}},
\end{align}
and, for the three-dimensional case,
\begin{align}
\label{A3.cal}
\mathcal{A}_3(\widebar{U}^{\pm},\widebar{\varPhi}^{\pm})
&=
{\small \begin{pmatrix}
	0&\bm{e}_3^{\mathsf{T}} &0&0 &0 &0\\[1mm]
\bm{e}_3& \bm{O}_{3}&\bm{O}_{3}&\bm{O}_{3}& -\bar{\rho}^{\pm}  \widebar{F}_{33}\bm{I}_{3}&0\\[1mm]
0 & \bm{O}_{3} &\bm{O}_{3}	& \bm{O}_{3}&\bm{O}_{3}&0\\[1mm]
0 & \bm{O}_{3} &\bm{O}_{3}	& \bm{O}_{3}&\bm{O}_{3}&0\\[1mm]
0& -\bar{\rho}^{\pm}  \widebar{F}_{33}\bm{I}_{3}&\bm{O}_{3} &\bm{O}_{3}  &\bm{O}_{3} &0\\[1mm]
\z{0}&\z{0}&\z{0}&\z{0}&\z{0} &\z{0}
	\end{pmatrix}},
\end{align}
where
$\bar{\rho}^{\pm}:=(\det \widebar{\bm{F}}^{\pm})^{-1}$ are the background densities,
and $\bar{c}_{\pm}:= p_{\rho}(\bar{\rho}^{\pm},\widebar{ S }^{\pm})^{1/2}$ are  the background speeds of sound.
The explicit expressions \eqref{A0.cal}--\eqref{A3.cal} will be used
in the estimate of tangential derivatives.

We now compute the exact form of $\mathring{\mathcal{A}}_{1}^{\pm}$ on boundary $\p\varOmega$, which is necessary
for deriving the energy estimate of tangential derivatives.
We first infer from \eqref{bas.c2}  and  \eqref{bas.c3b}
that matrices $\widetilde{A}_1(\mathring{U}^{\pm},\mathring{\varPhi}^{\pm})$ satisfy
 \begin{align}
 \left.\widetilde{A}_1(\mathring{U}^{\pm},\mathring{\varPhi}^{\pm})\right|_{x_1=0}
 =\pm
   \left.{\small \begin{pmatrix}
  0 & (\mathring{N}^{\pm})^{\mathsf{T}}& 0 & 0 \\[1.5mm]
  \mathring{N}^{\pm} & \bm{O}_d & -\mathring{\rho}^{\pm}\mathring{F}_{1N} ^{\pm}\bm{I}_d & 0 \\[1.5mm]
  0 & -\mathring{\rho}^{\pm}\mathring{F}_{1N} ^{\pm}\bm{I}_d   & \bm{O}_d & 0 \\[1.5mm]
   0& 0   &  0 & {\bm{O}_{d^2-d+1 }}
  \end{pmatrix}}\right|_{x_1=0}.
 \label{A1.tilde.bdy}
 \end{align}
In light of \eqref{A1.tilde.bdy}, we can decompose the boundary matrices $\mathring{\mathcal{A}}_1^{\pm} $ as
\begin{align}\label{A1.cal.decom}
\mathring{\mathcal{A}}_1^{\pm}=\mathring{J}_{\pm}^{\mathsf{T}} \widetilde{A}_1(\mathring{U}^{\pm},\mathring{\varPhi}^{\pm})\mathring{J}_{\pm}=\mathring{\mathcal{A}}_{1a}^{\pm}+\mathring{\mathcal{A}}_{1b}^{\pm}
\qquad\,\textrm{with $\left.\mathring{\mathcal{A}}_{1b}^{\pm}\right|_{x_1=0}=0$},
\end{align}
where
\begin{align} \label{A.1a}
\mathring{\mathcal{A}}_{1a}^{\pm}:=
\pm\begin{pmatrix}\w0 &\w0 & \w0 & \w0 \\[1.5mm]
0 & \bm{O}_d & \bm{A}(\mathring{U}^{\pm},\mathring{\varPhi}^{\pm}) & 0    \\[1.5mm]
0 & \bm{A}(\mathring{U}^{\pm},\mathring{\varPhi}^{\pm})& \bm{O}_d  & 0    \\[1.5mm]
0 & 0 & 0 & \bm{O}_{d^2-d+1}
\end{pmatrix},
\end{align}
with
\begin{align} \label{A.bm}
\bm{A}(U,\varPhi):=
\mathrm{diag}\,(1,\,-  \rho F_{1N}\bm{I}_{d-1}).
\end{align}
The explicit expression of $\mathring{\mathcal{A}}_{1b}^{\pm}$ is of no interest. According to
the kernels of matrices $\mathring{\mathcal{A}}_{1a}^{\pm}$, we denote by
\begin{align}
\label{W.nc}
W_{\rm nc}^{\pm}:=(W_{2}^{\pm}, \ldots,W_{2d+1}^{\pm})^{\mathsf{T}}
\end{align}
the noncharacteristic parts of unknowns $W^{\pm}$,
and by
$$
W_{\rm c}^{\pm}:=(W_{1}^{\pm}, W_{2d+2}^{\pm},  \ldots,W_{d^2+d+2}^{\pm})^{\mathsf{T}}
$$
the characteristic parts of $W^{\pm}$.

We reformulate the boundary conditions \eqref{ELP2.2} for unknowns $W^{\pm}$ into
\begin{subequations} \label{ELP3.b}
 \begin{alignat}{2}
 \label{ELP3.b.1} &\p_0\psi= W_2^++\underline{\mathring{\rm c}}_1 \psi &\quad & \textrm{on }\ \omega_T  ,\\
 \label{ELP3.b.2}  &[W_{i+1}]=\underline{\mathring{\rm c}}_1\psi
 & \textrm{for } i=1,\ldots,d,\quad   & \textrm{on }\ \omega_T ,\\
 \label{ELP3.b.3} &[W_{d+2}]=[\mathring{F}_{11}] \p_{F_{ij}} \varrho(\mathring{\bm{F}}^+)F_{ij}^++\underline{\mathring{\rm c}}_1\psi\ \  &\quad & \textrm{on }\ \omega_T,\\
 \label{ELP3.b.4} &[W_{d+j+1}]=-[\mathring{F}_{11}] \p_j \psi+\underline{\mathring{\rm c}}_1\psi
  &\textrm{for } j=2,\ldots,d, \quad & \textrm{on }\ \omega_T  ,
 \end{alignat}
\end{subequations}
where $ \varrho(\bm{F})$ and $\p_0$ are defined by \eqref{varrho} and \eqref{partial0}, respectively.
Here we recall that symbol $\underline{\mathring{\rm c}}_m$ denotes a generic and smooth matrix-valued function of $\{(\mathrm{D}^{\alpha}\mathring{V},\, \mathrm{D}^{\alpha}\mathring{\varPsi}): |\alpha|\leq m \}$ vanishing at the origin.
It is worth mentioning that the boundary conditions \eqref{ELP3.b} depend upon the traces of $W^{\pm}$
not only through the noncharacteristic variables $W_{\rm nc}^{\pm}$ but also through
the characteristic variables $F_{ij}^+$ for $i,j=2,\ldots,d,$
which is a different situation from the standard one (see, {\it e.g.}, \cite[\S 4.1]{B-GS07MR2284507}).

\section{Estimate of the Normal Derivatives} \label{sec.normal}
This section is devoted to the proof of the following proposition.

\begin{proposition}   \label{lem.normal}
If the assumptions in Theorem {\rm \ref{thm2}} are satisfied, then
\begin{align}
 \VERT  W (t)\VERT_{s}^2
\lesssim \VERT W(t)\VERT_{{\rm tan},s}^2
 + \|    (   \tilde{f},W)\|_{H^s(\varOmega_t)}^2 +\mathring{\rm C}_{s+2}
\|    ( \tilde{f},W)\|_{L^{\infty}(\varOmega_t)}^2 ,
\label{normal.est}
\end{align}
where $\VERT  \cdot \VERT_{s}$, $ \VERT \cdot \VERT_{{\rm tan},s}$, and $\mathring{\rm C}_{s+2} $
are defined by \eqref{VERT.norm}--\eqref{C.ring}, respectively.
In addition,
\begin{align}
 \VERT  W (t)\VERT_{1} ^2
\lesssim  \VERT W(t)\VERT_{{\rm tan},1}^2+
 \|  ( \tilde{f},W)\|_{H^1(\varOmega_t)}^2  .
\label{normal.est'}
\end{align}
\end{proposition}

In this section, we let
$\beta=(\beta_0,\beta_2,\ldots,\beta_d)\in \mathbb{N}^d$ be a multi-index
with  $|\beta|\leq s-1$.
The proof of this proposition is divided into the following five subsections.

\subsection{Estimate of the Noncharacteristic Variables}\label{sec.normal1}

In view of \eqref{ELP3.a} and \eqref{A1.cal.decom}--\eqref{A.1a}, we have
\begin{align}
\notag
\begin{pmatrix}
0\\
\p_1 W_{\rm nc}^{\pm}\\ 0
\end{pmatrix}
=\,&
-\mathring{\bm{B}}^{\pm} \mathring{\mathcal{A}}_0^{\pm}  \partial_t W^{\pm}
-\sum_{j=2}^d \mathring{\bm{B}}^{\pm}\mathring{\mathcal{A}}_j^{\pm}   \partial_j W^{\pm}\\
&
-\mathring{\bm{B}}^{\pm}\mathring{\mathcal{A}}_{1b}^{\pm}\partial_1 W^{\pm}
-\mathring{\bm{B}}^{\pm}\mathring{\mathcal{A}}_4^{\pm}W^{\pm}
+\mathring{\bm{B}}^{\pm}\mathring{J}_{\pm}^{\mathsf{T}}\tilde{f}^{\pm},
\label{Wnc.id}
\end{align}
where $\mathring{\bm{B}}^{\pm} :=\pm\bm{B}(\mathring{U}^{\pm},\mathring{\varPhi}^{\pm})$,
and $ \bm{B}(U,\varPhi)$ is defined by
\begin{align} \label{B.bm}
\bm{B}(U,\varPhi)
:=\begin{pmatrix}\w0 &\w0 & \w0 & \w0 \\[1mm]
0 & \bm{O}_d & \bm{A}(U,\varPhi)^{-1} & 0    \\[1mm]
0 & \bm{A}(U,\varPhi)^{-1}& \bm{O}_d  & 0    \\[1mm]
0 & 0 & 0 & \bm{O}_{d^2-d+1}
\end{pmatrix},
\end{align}
with $\bm{A}(U,\varPhi)$ given in \eqref{A.bm}.

Noting that  $ \bm{B}(U,\varPhi)$ and $\mathcal{A}_j(U,\varPhi)$
are  $C^{\infty}$--functions of $(U,\mathrm{D}\varPhi)$ for $j=0,\ldots,d$,
we  apply operator $\mathrm{D}_{\rm tan}^{\beta}:=\p_t^{\beta_0}\p_2^{\beta_2}\cdots\p_d^{\beta_d}$
to identity \eqref{Wnc.id} and deduce
\begin{align}
\notag
\| \p_1 \mathrm{D}_{\rm tan}^{\beta} W_{\rm nc} \|_{L^2(\varOmega)}^2\lesssim \;&
\|  \mathrm{D}_{\rm tan}^{\beta}  (\mathring{\rm c}_1 \mathrm{D}_{\rm tan} W)\|_{L^2(\varOmega)}^2
+\|  \mathrm{D}_{\rm tan}^{\beta}  (\mathring{\bm{B}} \mathring{\mathcal{A}}_{1b} \partial_1 W )\|_{L^2(\varOmega)}^2
\\
\label{Wnc.est1}
&+\|  \mathrm{D}_{\rm tan}^{\beta}  (\mathring{\bm{B}}\mathring{\mathcal{A}}_4 W)\|_{L^2(\varOmega)}^2
+\|  \mathrm{D}_{\rm tan}^{\beta}  (\mathring{\bm{B}}\mathring{J}^{\mathsf{T}}\tilde{f})\|_{L^2(\varOmega)}^2.
\end{align}
Here we recall that $ {\mathring{\rm c}}_m$ denotes a generic and smooth matrix-valued function
of $\{(\mathrm{D}^{\alpha}\mathring{V}, \mathrm{D}^{\alpha}\mathring{\varPsi}): |\alpha|\leq m \}$.

We integrate by parts to obtain
\begin{align}
 \VERT u(t) \VERT_{m-1}^2
\label{MTT.inequ}
\lesssim  \sum_{|\alpha|\leq m-1}
\int_{\varOmega_t} |\mathrm{D}^{\alpha} u(\tau,x)||\p_t \mathrm{D}^{\alpha} u(\tau,x)| \mathrm{d }x \mathrm{d}\tau  \lesssim \|u\|^2_{H^{m}(\varOmega_t)}.
\end{align}
By virtue of \eqref{MTT.inequ} and the Moser-type calculus inequality \eqref{Moser4}, we have
\begin{align}
\notag  \|  \mathrm{D}_{\rm tan}^{\beta}  (\mathring{\rm c}_1 \mathrm{D}_{\rm tan} W)\|_{L^2(\varOmega)}^2
 &\lesssim \|\mathring{\rm c}_1  \mathrm{D}_{\rm tan}^{\beta}  \mathrm{D}_{\rm tan} W+
 [\mathrm{D}_{\rm tan}^{\beta}  ,\,\mathring{\rm c}_1] \mathrm{D}_{\rm tan} W\|_{L^2(\varOmega)}^2\\[0.5mm]
\notag &  \lesssim    \VERT W\VERT_{{\rm tan},s}^2
+ \|   [\mathrm{D}_{\rm tan}^{\beta}  ,\,\mathring{\rm c}_1] \mathrm{D}_{\rm tan} W\|_{H^1(\varOmega_t)}^2\\
&    \lesssim  \VERT W\VERT_{{\rm tan},s}^2
+ \| W\|_{H^s(\varOmega_t)}^2
+ \mathring{\rm C}_{s+2} \|    W\|_{L^{\infty}(\varOmega_t)}^2 .
\label{Wnc.est1a}
\end{align}

Since $ \bm{B}(U,\varPhi)$ and  $J(U,\varPhi)$ are $C^{\infty}$--functions of $(U,\mathrm{D}\varPhi)$,
and $\mathcal{A}_4({U},  {\varPhi})$ is a $C^{\infty}$--function of $(U,\mathrm{D}\varPhi,\mathrm{D}U,\mathrm{D}^2\varPhi)$,
we use \eqref{MTT.inequ} and the Moser-type calculus inequality \eqref{Moser3} to obtain
\begin{align}
\notag &\|\mathrm{D}_{\rm tan}^{\beta}(\mathring{\bm{B}} \mathring{\mathcal{A}}_4 W )\|_{L^2(\varOmega)}^2
+\|\mathrm{D}_{\rm tan}^{\beta}(\mathring{\bm{B}} \mathring{J}^{\mathsf{T}}\tilde{f})\|_{L^2(\varOmega)}^2\\
 &\qquad \lesssim \|\mathring{\rm c}_2  W \|_{H^s(\varOmega_t)}^2
+\|\mathring{\rm c}_1  \tilde{f} \|_{H^s(\varOmega_t)}^2 \notag \\
\label{Wnc.est1b}
&\qquad  \lesssim  \|  ( \tilde{f},W)\|_{H^s(\varOmega_t)}^2
+ \mathring{\rm C}_{s+2} \|    ( \tilde{f},W)\|_{L^{\infty}(\varOmega_t)}^2 .
\end{align}

Notice  from \eqref{bas.relation} and \eqref{bas.bound}
that the $W^{2,\infty}(\varOmega_T)$--norm of  $(\mathring{V},\,\mathring{\varPsi})$
is bounded by $CK$ for some positive constant $C$ depending only on $\chi$.
In view of \eqref{A1.cal.decom}, we have
$$\|\p_1(\mathring{\bm{B}}^{\pm}\mathring{\mathcal{A}}_{1b}^{\pm})\|_{L^{\infty}(\varOmega_T)}
\lesssim \|\mathring{\rm c}_2\|_{L^{\infty}(\varOmega_T)} \lesssim 1,
\qquad \mathring{\bm{B}}^{\pm}\mathring{\mathcal{A}}_{1b}^{\pm}\big|_{x_1=0}=0.
$$
Then we integrate by parts to obtain
\begin{align}\label{Wnc.est1e}
\big\|\big(\mathring{\bm{B}}^{\pm}\mathring{\mathcal{A}}_{1b}^{\pm}\big)(\cdot,x_1,\cdot)\big\|_{L^{\infty}([0,T]\times \mathbb{R}^{d-1})}
\lesssim \sigma(x_1)\qquad\,\, \textrm{for $x_1\geq 0$},
\end{align}
where $\sigma$ is an increasing function of $x_1$ satisfying
\begin{align}  \label{sigma}
\sigma=\sigma(x_1)  \in C^{\infty}(\mathbb{R}),\qquad
\sigma(x_1)=\begin{cases}
x_1&\,\,\,\, \textrm{for $0\leq x_1\leq 1$},\\
2&\,\,\,\, \textrm{for $x_1\geq 4$}.
\end{cases}
\end{align}
Utilizing the estimate above along with \eqref{MTT.inequ} and \eqref{Moser4},
we infer
\begin{align}
\notag& \|\mathrm{D}_{\rm tan}^{\beta }(\mathring{\bm{B}}\mathring{\mathcal{A}}_{1b}\partial_1 W)\|_{L^2(\varOmega)}^2\\
\notag &\qquad \lesssim \big\| \mathring{\bm{B}}\mathring{\mathcal{A}}_{1b}\mathrm{D}_{\rm tan}^{\beta } \partial_1 W
+[\mathrm{D}_{\rm tan}^{\beta},\mathring{\bm{B}}\mathring{\mathcal{A}}_{1b}]\partial_1 W \big\|_{L^2(\varOmega)}^2\\[0.5mm]
\notag &\qquad
\lesssim   \big\|\sigma \mathrm{D}_{\rm tan}^{\beta } \partial_1 W  \big\|_{L^2(\varOmega)}^2
+\big\| [\mathrm{D}_{\rm tan}^{\beta },\mathring{\bm{B}}\mathring{\mathcal{A}}_{1b}]\partial_1 W \big\|_{H^1(\varOmega_t)}^2\\
&\qquad\lesssim  \|\sigma\p_1\mathrm{D}_{\rm tan}^{\beta }W\|_{L^2(\varOmega)}^2
+\| W\|_{H^s(\varOmega_t)}^2
+ \mathring{\rm C}_{s+2} \|W\|_{L^{\infty}(\varOmega_t)}^2.
\label{Wnc.est1c}
\end{align}
Apply operator $\sigma \p_1^k\mathrm{D}_{\rm tan}^{\beta'}$ with $k+|\beta'|\leq s$
to system \eqref{ELP3.a}  and  employ the standard arguments of the energy method to deduce
\begin{align}\label{sigma.est'}
  & \|\sigma \p_1^k\mathrm{D}_{\rm tan}^{\beta'} W(t)\|_{L^2(\varOmega)}^2
 \lesssim   \|  ( \tilde{f},W)\|_{H^1(\varOmega_t)}^2\qquad   \textrm{for $k+|\beta'|\leq 1$},
\\
\label{sigma.est}
&\|\sigma \p_1^k\mathrm{D}_{\rm tan}^{\beta'} W(t)\|_{L^2(\varOmega)}^2
\lesssim \|  ( \tilde{f},W)\|_{H^s(\varOmega_t)}^2
+ \mathring{\rm C}_{s+2} \|    ( \tilde{f},W)\|_{L^{\infty}(\varOmega_t)}^2
 \   \textrm{for $k+|\beta'|\leq s$}.
\end{align}
Plugging \eqref{Wnc.est1a}--\eqref{Wnc.est1b}, \eqref{Wnc.est1c},
and \eqref{sigma.est} into \eqref{Wnc.est1} implies
\begin{align}
\notag
&\sum_{|\beta|\leq s-1 }\| \p_1 \mathrm{D}_{\rm tan}^{\beta} W_{\rm nc}(t) \|_{L^2(\varOmega)}^2  \\
&\qquad \lesssim  \VERT W(t)\VERT_{{\rm tan},s}^2  + \|  ( \tilde{f},W)\|_{H^s(\varOmega_t)}^2
+ \mathring{\rm C}_{s+2} \|    ( \tilde{f},W)\|_{L^{\infty}(\varOmega_t)}^2 .
\label{Wnc.est}
\end{align}
Moreover, from \eqref{Wnc.est1} with $\beta=0$, \eqref{Wnc.est1e}, and \eqref{sigma.est'},
we have
\begin{align}
\| \p_1   W_{\rm nc}(t) \|_{L^2(\varOmega)}^2
\lesssim  \VERT W(t)\VERT_{{\rm tan},1}^2 +\|  ( \tilde{f},W)\|_{H^1(\varOmega_t)}^2.
\label{Wnc.est'}
\end{align}

\subsection{Estimate of the Characteristic Variables $S^{\pm}$}\label{sec.entropy}

The next lemma gives the estimate of the characteristic variables $W_{d^2+d+2}^{\pm}$
that are entropies $S^{\pm}$.

\begin{lemma}\label{lem.S}
 If the assumptions in Theorem {\rm \ref{thm2}} are satisfied, then
 \begin{align}
 &\VERT S^{\pm}(t)\VERT_{s}^2 \lesssim \|  ( \tilde{f},W)\|_{H^s(\varOmega_t)}^2
 + \mathring{\rm C}_{s+2} \|    ( \tilde{f},W)\|_{L^{\infty}(\varOmega_t)}^2,
  \label{S.est}
\\
&\VERT S^{\pm}(t)\VERT_{1}  \lesssim \|  ( \tilde{f},W)\|_{H^1(\varOmega_t)} .
\label{S.est'}
\end{align}
\end{lemma}

\begin{proof}
Since matrices $\mathcal{C}(\mathring{U}^{\pm},\mathring{\varPhi}^{\pm})$
are $C^{\infty}$--functions of $(\mathring{V},\mathrm{D}\mathring{V}, \mathrm{D}\mathring{\varPsi})$ vanishing at the origin,	
 we can write the equations for $S^{\pm}$ in \eqref{ELP2.1} as
 \begin{align*}
 (\p_t+\mathring{w}^{\pm}_{\ell}\p_{\ell})S^{\pm}
 = \mathring{\rm c}_0 \tilde{f}+  \underline{\mathring{\rm c}}_1 W \qquad\,\, \textrm{in }\ \varOmega,
 \end{align*}
 where $\mathring{w}^{\pm}_{\ell}$, $\ell=1,\ldots,d$, are given in \eqref{w.ring}.
 Let $\alpha:=(\alpha_0,\alpha_1,\ldots,\alpha_d)\in\mathbb{N}^{d+1}$ be any multi-index
 with $|\alpha|:=\alpha_0+\alpha_1+\cdots+\alpha_d\leq s$.
 Apply operator $\mathrm{D}^{\alpha}:=\p_t^{\alpha_0}\p_1^{\alpha_1}\cdots\p_d^{\alpha_d}$ to the  equations above
 and multiply the resulting identities by $\mathrm{D}^{\alpha} S^{\pm}$ respectively to find
 \begin{align*}
 &\p_t \big|\mathrm{D}^{\alpha} S^{\pm}\big|^2
 +\p_{\ell} \big(\mathring{w}^{\pm}_{\ell} \big|\mathrm{D}^{\alpha}S^{\pm}\big|^2\big)
 -\p_{\ell} \mathring{w}^{\pm}_{\ell} \big|\mathrm{D}^{\alpha}S^{\pm}\big|^2\\
 &\qquad =2\mathrm{D}^{\alpha} S^{\pm}\big(
 \mathrm{D}^{\alpha} (\mathring{\rm c}_0 \tilde{f})
 +  \mathrm{D}^{\alpha}(\mathring{\rm c}_1 W)
 -[\mathrm{D}^{\alpha},\, \mathring{w}^{\pm}_{\ell}]\p_{\ell} S^{\pm} \big).
 \end{align*}
Note that the $W^{2,\infty}(\varOmega_T)$--norm of  $(\mathring{V},\,\mathring{\varPsi})$
is bounded by $CK$ for some positive constant $C$ depending only on $\chi$.
 By virtue of \eqref{w1.ring.bdy},
 we can obtain \eqref{S.est}--\eqref{S.est'} by integrating the last identities
 over $\varOmega_t$ and applying the Moser-type
 calculus inequalities \eqref{Moser3}--\eqref{Moser4}.
\qed\end{proof}

\subsection{Estimate of the Characteristic Variables $W_1^{\pm}$}\label{sec.W1}
To compensate the loss of the normal derivatives of the characteristic variables $W^{\pm}_1=p^{\pm}$,
inspired by involutions \eqref{inv2},  we introduce linearized divergences $\varsigma^{\pm}$ by
\begin{align} \label{varsigma}
\varsigma^{\pm}:=
\p_{i }^{\mathring{\varPhi}^{\pm}}\big(
\mathring{c}_{\pm}^{-2}\mathring{{F}}_{i 1}^{\pm}p^{\pm}
+\mathring{\rho}^{\pm}{F}_{i 1}^{\pm}\big),
\end{align}
where $\p_{i}^{\mathring{\varPhi}^{\pm}}$, $i=1,\ldots,d,$ are defined by \eqref{differential},
and $\mathring{c}_{\pm}:= p_{\rho}(\mathring{\rho}^{\pm},\mathring{ S }^{\pm})^{1/2}$ are the basic speeds of sound.
See {\sc Trakhinin} \cite{T18MR3721411} for a slightly different definition of the linearized divergences.

Then we obtain the following estimate for $\varsigma^{\pm}$.

\begin{lemma}\label{lem.varsigma}
 If the assumptions in Theorem {\rm \ref{thm2}} are satisfied, then
 \begin{align} \label{varsigma.est}
& \VERT \varsigma^{\pm}(t) \VERT_{s-1}^2
\lesssim
 \|(\tilde{f} ,W)\|_{H^{s}(\varOmega_t) }^2+ \mathring{\rm C}_{s+2}\|(\tilde{f} ,W)\|_{L^{\infty} (\varOmega_t)}^2,
\\
 \label{varsigma.est'}
 &\| \varsigma^{\pm}(t) \|_{L^2(\varOmega)}
\lesssim
  \|(\tilde{f} ,W)\|_{H^{1}(\varOmega_t) }.
\end{align}
\end{lemma}

\begin{proof}
The equations for $\bm{F}^{\pm}$ and  $p^{\pm}$  in \eqref{ELP2.1} read
\begin{align}
\label{F.equ}&
(\p_t+\mathring{w}_{\ell}^{\pm}\p_{\ell}){F}_{ij}^{\pm}
- \mathring{{F}}_{\ell j}^{\pm}\p_{\ell}^{\mathring{\varPhi}^{\pm}} v_i^{\pm}
=\mathring{\rm c}_0  \tilde{f}+\underline{\mathring{\rm c}}_1 W,\\
\label{p.equ}&
(\p_t+\mathring{w}_{\ell}^{\pm}\p_{\ell})p^{\pm}
+\mathring{\rho}^{\pm}\mathring{c}_{\pm}^2\p_{\ell}^{\mathring{\varPhi}^{\pm}} v_{\ell}^{\pm}
=\mathring{\rm c}_0\tilde{f}+\underline{\mathring{\rm c}}_1 W.
\end{align}
 In view of these last equations, we compute
 \begin{align} \notag
 &(\p_t+\mathring{w}_{\ell}^{\pm}\p_{\ell})\left(
 \mathring{c}_{\pm}^{-2}\mathring{ {F}}_{i1}^{\pm} p ^{\pm}
 +  \mathring{\rho}^{\pm} {F}_{i1}^{\pm}  \right)
 =\mathring{\rho}^{\pm} \mathring{{F}}_{\ell 1}^{\pm} \p_{\ell}^{\mathring{\varPhi}^{\pm} }v_i^{\pm}
 -\mathring{\rho}^{\pm} \mathring{{F}}_{i1}^{\pm} \p_{\ell}^{\mathring{\varPhi}^{\pm} }v_{\ell}^{\pm}
 +\mathring{\rm c}_0  \tilde{f}+\mathring{\rm c}_1 W.
 \end{align}
 Performing operators $\p_{i }^{\mathring{\varPhi}^{\pm}}$ to the  identities above
 and using
 \begin{align*}
 &\mathring{\rho}^{\pm} \mathring{{F}}_{\ell 1}^{\pm} \p_{i}^{\mathring{\varPhi}^{\pm} }\p_{\ell}^{\mathring{\varPhi}^{\pm} }v_i^{\pm}
 -\mathring{\rho}^{\pm} \mathring{{F}}_{i1}^{\pm}\p_{i}^{\mathring{\varPhi}^{\pm} } \p_{\ell}^{\mathring{\varPhi}^{\pm} }v_{\ell}^{\pm}\\
 &\quad =\mathring{\rho}^{\pm} \mathring{{F}}_{i 1}^{\pm} \big[\p_{\ell}^{\mathring{\varPhi}^{\pm} },\p_{i}^{\mathring{\varPhi}^{\pm} }\big]v_{\ell}^{\pm}
 =\mathring{\rm c}_2 \mathrm{D}V=\mathring{\rm c}_2 \mathrm{D} (\mathring{J} W)
=\mathring{\rm c}_2 \mathrm{D}W +\mathring{\rm c}_2 W,
\end{align*}
we have
\begin{align} \label{varsigma.eq}
 (\p_t+\mathring{w}_{\ell}^{\pm}\p_{\ell})\varsigma ^{\pm}=\mathring{\rm c}_1 \mathrm{D}\tilde{f}+ \mathring{\rm c}_1 \tilde{f}+\mathring{\rm c}_2 \mathrm{D}W + \mathring{\rm c}_2 W.
\end{align}
 Apply operator $\mathrm{D}^{\alpha}$ with $|\alpha|\leq s-1$ to equations \eqref{varsigma.eq},
 multiply the resulting identities by $\mathrm{D}^{\alpha} \varsigma^{\pm}$ respectively,
 and take the integration over $\varOmega_t$ to obtain
\begin{align}
 \notag \|\mathrm{D}^{\alpha} \varsigma^{\pm}(t)\|_{L^2 (\varOmega)}^2
  \lesssim\,&
    \left(1+\|\mathrm{D}\mathring{w}\|_{L^{\infty}(\varOmega_t) }\right)
  \|\mathrm{D}^{\alpha}\varsigma^{\pm}\|_{L^{2}(\varOmega_t)}^2
 +\|[\mathrm{D}^{\alpha},\, \mathring{w}_{\ell}^{\pm} ]\p_{\ell} \varsigma^{\pm} \|_{L^{2}(\varOmega_t)}^2\\
 \label{varsigma.est1}&
 +\big\|\mathrm{D}^{\alpha}\big(
  \mathring{\rm c}_1 \mathrm{D}\tilde{f} + \mathring{\rm c}_1 \tilde{f}
  +\mathring{\rm c}_2 \mathrm{D}W + \mathring{\rm c}_2 W\big) \big\|_{L^{2}(\varOmega_t)}^2.
 \end{align}
 Since
 \begin{align}\label{varsigma.dec}
\varsigma^{\pm}=\mathring{\rm c}_1 W+\mathring{\rm c}_1  \mathrm{D}W,
 \end{align}
we have
\begin{align*}
&\|\mathrm{D}^{\alpha}\varsigma^{\pm}\|_{L^{2}(\varOmega_t)}\leq
 \|\mathrm{D}^{\alpha}(\mathring{\rm c}_2 \mathrm{D}W + \mathring{\rm c}_2 W)\|_{L^{2}(\varOmega_t)},\\
&\|[\mathrm{D}^{\alpha},\, \mathring{w}_{\ell}^{\pm} ]\p_{\ell} \varsigma^{\pm} \|_{L^{2}(\varOmega_t)}
 \lesssim
 \big\|\!\left([\mathrm{D}^{\alpha},\, \mathring{\rm c}_2 ] W,\,
 [\mathrm{D}^{\alpha},\, \mathring{\rm c}_2 ] \mathrm{D}W,\,
 [\mathrm{D}^{\alpha},\, \mathring{\rm c}_1 ] \mathrm{D}^2 W \right)\!\big\|_{L^{2}(\varOmega_t)}.
\end{align*}
Estimate \eqref{varsigma.est'} follows by plugging these last inequalities into \eqref{varsigma.est1}
with $\alpha=0$.
Apply the Moser-type calculus inequality \eqref{Moser3} and use $|\alpha|\leq s-1$ to derive
 \begin{align}
 \notag &\| (\mathrm{D}^{\alpha} (\mathring{\rm c}_1 \mathrm{D}\tilde{f} ),\mathrm{D}^{\alpha} (\mathring{\rm c}_1  \tilde{f} ) ) \|_{L^{2}(\varOmega_t)}^2\\
 \notag &\quad \lesssim \| (\mathring{\rm c}_1 \mathrm{D}^{\alpha} \mathrm{D}\tilde{f},\mathring{\rm c}_1 \mathrm{D}^{\alpha}  \tilde{f}   )  \|_{L^{2}(\varOmega_t)}^2
 +\| ([\mathrm{D}^{\alpha} ,\mathring{\rm c}_1 ]\mathrm{D}\tilde{f},[\mathrm{D}^{\alpha} ,\mathring{\rm c}_1 ] \tilde{f}   )  \|_{L^{2}(\varOmega_t)}^2\\
 \notag
 &\quad \lesssim \|\tilde{f} \|_{H^s(\varOmega_t) }^2+ \mathring{\rm C}_{s+1}\|\tilde{f} \|_{L^{\infty} (\varOmega_t)}^2.
 \end{align}
 By virtue of \eqref{Moser3}--\eqref{Moser4}, we obtain
 \begin{align}
 &\notag
 \|([\mathrm{D}^{\alpha},\, \mathring{\rm c}_2 ] W,\,
 [\mathrm{D}^{\alpha},\, \mathring{\rm c}_2 ] \mathrm{D}W  )\|_{L^{2}(\varOmega_t)}^2
 +\| [\mathrm{D}^{\alpha} ,\,\mathring{\rm c}_1 ] \mathrm{D}^2 W   \|_{L^{2}(\varOmega_t)}^2\\
&\notag
 \qquad \lesssim \|W\|_{H^s(\varOmega_t) }^2+ \mathring{\rm C}_{s+2}\|W\|_{L^{\infty} (\varOmega_t)}^2.
 \end{align}
Inserting the  estimates above into \eqref{varsigma.est1} yields \eqref{varsigma.est}.
 This completes the proof.
\qed\end{proof}

Thanks to \eqref{varsigma.est}, we can obtain the estimate of the characteristic variables $W_{1}^{\pm}=p^{\pm}$.
More precisely, according to \eqref{W.def.d} and \eqref{W.nc}, we have
\begin{align*}
\mathring{N}^{\pm}\cdot {F}_{1}^{\pm}
=\frac{| \mathring{N}^{\pm}|^2}{\mathring{\rho}^{\pm} \mathring{{F}}_{1N}^{\pm}  }
(W_1^{\pm}-W_{d+2}^{\pm})-\sum_{j=2}^d \p_j\mathring{\varPhi}^{\pm} W_{d+j+1}^{\pm}
=\frac{| \mathring{N}^{\pm}|^2}{\mathring{\rho}^{\pm} \mathring{{F}}_{1N}^{\pm}  }
W_1^{\pm}+\mathring{\rm c}_1   W_{\rm nc}.
\end{align*}
Combining the last identity with \eqref{varsigma} and recalling \eqref{differential}, we calculate
 \begin{align}
\notag \p_1 \mathring{\varPhi}^{\pm}\varsigma^{\pm}
=\;&
 \begin{matrix}
\p_1\left( \mathring{c}_{\pm}^{-2}\mathring{{F}}_{1N}^{\pm}p^{\pm}
+\mathring{\rho}^{\pm}   \mathring{N}^{\pm}\cdot {F}_{1}^{\pm}\right)
+\sum\limits_{i=2}^d\p_i\left(\p_1\mathring{\varPhi}^{\pm}
\big( \mathring{c}_{\pm}^{-2}\mathring{{F}}_{i1}^{\pm}p^{\pm}
+\mathring{\rho}^{\pm}   {F}_{i1}^{\pm}\big)\right)
\end{matrix}  \\[1mm]
\notag
=\;&
 \frac{  \mathring{c}_{\pm}^{-2} |\mathring{{F}}_{1N}^{\pm}|^2
 +| \mathring{N}^{\pm}|^2}{\mathring{{F}}_{1N}^{\pm}} \p_1 W_1^{\pm}
+\mathring{\rm c}_1 \p_1 W_{\rm nc}+\mathring{\rm c}_1 \mathrm{D}_{\rm tan} W+\mathring{\rm c}_2 W,
\end{align}
which implies
\begin{align} \label{varsigma.id2}
\p_1 W_1^{\pm} = \mathring{\rm c}_1 \varsigma^{\pm}
 +\mathring{\rm c}_1 \p_1 W_{\rm nc}+\mathring{\rm c}_1 \mathrm{D}_{\rm tan} W+\mathring{\rm c}_2 W.
\end{align}
In light of \eqref{varsigma.id2},
we utilize  \eqref{MTT.inequ}, \eqref{Wnc.est}, \eqref{varsigma.est},  \eqref{varsigma.dec},
and the Moser-type calculus inequalities \eqref{Moser3}--\eqref{Moser4} to obtain
 \begin{align}
 \notag &
 \sum_{|\beta|\leq s-1 }\| \p_1 \mathrm{D}_{\rm tan}^{\beta} W_{1}(t) \|_{L^2(\varOmega)}^2 \\
\notag &\quad \lesssim \sum_{|\beta|\leq s-1 }
\|(\mathrm{D}_{\rm tan}^{\beta}\varsigma, \mathrm{D}_{\rm tan}^{\beta}\p_1 W_{\rm nc},
\mathrm{D}_{\rm tan}^{\beta}\mathrm{D}_{\rm tan} W)\|_{L^2(\varOmega)}^2 \\
\notag & \quad\quad\,\,
+ \sum_{|\beta|\leq s-1 }
\big\|\big( [\mathrm{D}_{\rm tan}^{\beta} ,\mathring{\rm c}_1 ]W, [\mathrm{D}_{\rm tan}^{\beta} ,\mathring{\rm c}_1] \mathrm{D}W,
\mathrm{D}_{\rm tan}^{\beta}  (\mathring{\rm c}_2 W)\big)\big\|_{H^1(\varOmega_t)}^2 \\
&\quad\lesssim \VERT W(t)\VERT_{{\rm tan},s}^2
+\|  ( \tilde{f},W)\|_{H^s(\varOmega_t)}^2
 + \mathring{\rm C}_{s+2} \|    ( \tilde{f},W)\|_{L^{\infty}(\varOmega_t)}^2 .
 \label{W1.est}
 \end{align}
Furthermore, we plug \eqref{Wnc.est'} and \eqref{varsigma.est'}
into \eqref{varsigma.id2} to obtain
 \begin{align}
  \| \p_1   W_{1}(t) \|_{L^2(\varOmega)}^2
 \lesssim
\VERT W(t)\VERT_{{\rm tan},1}^2      +  \|  ( \tilde{f},W)\|_{H^1(\varOmega_t)}^2 .
\label{W1.est'}
\end{align}

\subsection{Estimate of the Remaining Characteristic Variables}\label{sec.normal4}

To recover the normal derivatives of the characteristic variables
$W_{jd+i+1}^{\pm}=F_{ij}^{\pm}$ for $i=1,\ldots,d$ and $j=2,\ldots,d$,
motivated by constraints \eqref{inv1},
we introduce quantities $\eta^{\pm}:=(\eta_1^{\pm},\ldots,\eta_d^{\pm})$ with
\begin{align} \label{eta}
&\eta_i^{\pm}:=\mathring{{F}}_{k 1}^{\pm}\p_{k}^{\mathring{\varPhi}^{\pm}} {F}_{i2}^{\pm}
-\mathring{{F}}_{k 2}^{\pm}\p_{k}^{\mathring{\varPhi}^{\pm}} {F}_{i1}^{\pm}.
\end{align}
In addition, for $d=3$, we introduce quantities
$\zeta^{\pm}:=(\zeta_1^{\pm},\zeta_2^{\pm},\zeta_3^{\pm})$ with
\begin{align}
\label{zeta}
&\zeta_i^{\pm}:=\mathring{{F}}_{k 1}^{\pm}\p_{k}^{\mathring{\varPhi}^{\pm}} {F}_{i3}^{\pm}
-\mathring{{F}}_{k 3}^{\pm}\p_{k}^{\mathring{\varPhi}^{\pm}} {F}_{i1}^{\pm}.
\end{align}

We have the following estimates for $\eta^{\pm}$ and $\zeta^{\pm}$.

\begin{lemma}
 If the assumptions in Theorem {\rm \ref{thm2}} are satisfied, then
 \begin{align} \label{eta.est}
& \VERT( \eta^{\pm}, \zeta^{\pm})\VERT_{s-1}^2
  \lesssim
 \|(\tilde{f} ,W)\|_{H^{s}(\varOmega_t) }^2+ \mathring{\rm C}_{s+2}\|(\tilde{f} ,W)\|_{L^{\infty} (\varOmega_t)}^2,
\\ \label{eta.est'}
& \|( \eta^{\pm}, \zeta^{\pm})\|_{L^2(\varOmega)}
 \lesssim\|(\tilde{f} ,W)\|_{H^{1}(\varOmega_t) }.
 \end{align}
\end{lemma}

\begin{proof}
Thanks to \eqref{F.equ}, we deduce the equations for $\eta^{\pm}$ and $\zeta^{\pm}$:
 \begin{align}
 \label{eta.eq}
 &(\p_t+\mathring{w}_{\ell}^{\pm}\p_{\ell})\eta ^{\pm}
 =\mathring{\rm c}_1 \mathrm{D}\tilde{f}+ \mathring{\rm c}_1 \tilde{f}+\mathring{\rm c}_2 \mathrm{D}W + \mathring{\rm c}_2 W,\\[1mm]
 & \label{zeta.eq}
 (\p_t+\mathring{w}_{\ell}^{\pm}\p_{\ell})\zeta ^{\pm}
 =\mathring{\rm c}_1 \mathrm{D}\tilde{f}+ \mathring{\rm c}_1 \tilde{f}+\mathring{\rm c}_2 \mathrm{D}W + \mathring{\rm c}_2 W,
 \end{align}
 where we have used
 \begin{align*}
 \mathring{{F}}_{k i}^{\pm} \mathring{{F}}_{\ell j}^{\pm}\p_{k}^{\mathring{\varPhi}^{\pm}} \p_{\ell}^{\mathring{\varPhi}^{\pm}}
 -\mathring{{F}}_{k j}^{\pm}\mathring{{F}}_{\ell i}^{\pm}\p_{k}^{\mathring{\varPhi}^{\pm}} \p_{\ell}^{\mathring{\varPhi}^{\pm}}
 =\mathring{{F}}_{\ell i}^{\pm} \mathring{{F}}_{k j}^{\pm}
 	\big[\p_{\ell}^{\mathring{\varPhi}^{\pm}}, \p_{k}^{\mathring{\varPhi}^{\pm}}  \big]
 =\mathring{\rm c}_2 \mathrm{D}.
 \end{align*}
 Noting that $\eta^{\pm}=\mathring{\rm c}_1 \mathrm{D}V$ and $\zeta^{\pm}=\mathring{\rm c}_1 \mathrm{D}V$,
 we perform the same analysis as $\varsigma^{\pm}$ in Lemma \ref{lem.varsigma} to deduce \eqref{eta.est}--\eqref{eta.est'}.
This completes the proof.
\qed\end{proof}

According to \eqref{differential},  we compute
\begin{align}
\label{eta.id0}
\eta_i^{\pm}&=
\frac{1}{\p_1 \mathring{\varPhi}^{\pm}}
\big( \mathring{ {F}}_{1N}^{\pm}\p_1 F_{i2}^{\pm}  -\mathring{ {F}}_{2N}^{\pm}\p_1 F_{i1}^{\pm}   \big)
+\sum_{\ell=2}^d
\big(  \mathring{ {F}}_{\ell 1}^{\pm}\p_{\ell} F_{i2}^{\pm}  -\mathring{ {F}}_{\ell 2}^{\pm}\p_{\ell} F_{i1}^{\pm}   \big),
\\
\label{zeta.id0}
\zeta_i^{\pm}&=
\frac{1}{\p_1 \mathring{\varPhi}^{\pm}}
\big( \mathring{ {F}}_{1N}^{\pm}\p_1 F_{i3}^{\pm}  -\mathring{ {F}}_{3N}^{\pm}\p_1 F_{i1}^{\pm}   \big)
+\sum_{\ell=2}^d
\big(  \mathring{ {F}}_{\ell 1}^{\pm}\p_{\ell} F_{i3}^{\pm}  -\mathring{ {F}}_{\ell 3}^{\pm}\p_{\ell} F_{i1}^{\pm}   \big),
\end{align}
which, combined with \eqref{W.def.d} and \eqref{W.nc}, imply
\begin{align}
\label{eta.id1} \p_1 F_{i2}^{\pm}
=\mathring{\rm c}_1 \eta_i^{\pm}
+\mathring{\rm c}_1 \p_1 W_{\rm nc}+\mathring{\rm c}_1 \p_1 W_{1}
+\mathring{\rm c}_1 \mathrm{D}_{\rm tan} W+\mathring{\rm c}_2 W,\\[1mm]
\label{zeta.id1}  \p_1 F_{i3}^{\pm}
=\mathring{\rm c}_1 \zeta_i^{\pm}
+\mathring{\rm c}_1 \p_1 W_{\rm nc}+\mathring{\rm c}_1 \p_1 W_{1}
+\mathring{\rm c}_1 \mathrm{D}_{\rm tan} W+\mathring{\rm c}_2 W.
\end{align}
In view of \eqref{eta.id1}--\eqref{zeta.id1},
we utilize \eqref{Wnc.est}--\eqref{Wnc.est'}, \eqref{W1.est}--\eqref{W1.est'}, \eqref{eta.est}--\eqref{eta.est'},
the Moser-type calculus inequalities  \eqref{Moser3}--\eqref{Moser4},  and \eqref{MTT.inequ}
 to obtain
\begin{align}\label{Wc.est'}
\| \p_1  W_{jd+i+1}(t) \|_{L^2(\varOmega)}^2
 \lesssim  \VERT  W(t)\VERT_{{\rm tan},1}^2+\|  ( \tilde{f},W)\|_{H^1(\varOmega_t)}^2,
\end{align}
and
\begin{align}
&\sum_{|\beta|\leq s-1 }\| \p_1 \mathrm{D}_{\rm tan}^{\beta} W_{jd+i+1}(t) \|_{L^2(\varOmega)}^2\notag\\
 & \quad \lesssim\VERT  W(t)\VERT_{{\rm tan},s}^2   +   \|  ( \tilde{f},W)\|_{H^s(\varOmega_t)}^2
+ \mathring{\rm C}_{s+2} \|    ( \tilde{f},W)\|_{L^{\infty}(\varOmega_t)}^2 ,
\label{Wc.est}
\end{align}
for $i=1,\ldots,d$, and $j=2,\ldots,d$.

\subsection{Proof of Proposition {\rm \ref{lem.normal}}}\label{sec.normal5}

Estimate \eqref{normal.est'} follows by applying \eqref{MTT.inequ} and combining estimates
\eqref{Wnc.est'}, \eqref{S.est'}, \eqref{W1.est'}, and \eqref{Wc.est'}.
Thanks to \eqref{Wnc.id}, \eqref{varsigma.id2}, and \eqref{eta.id1}--\eqref{zeta.id1},
we can combine estimates \eqref{sigma.est}, \eqref{S.est}, \eqref{varsigma.est}, \eqref{eta.est},
and \eqref{Wc.est} to prove  by
induction in $\ell=1,\ldots,s$ that
\begin{align}
&\notag \sum_{k=1}^{\ell}\sum_{|\beta |\leq s-k}
\|\p_1^k\mathrm{D}_{\rm tan}^{\beta} W (t)\|_{L^2(\varOmega)}^2 \\
&\quad \lesssim\VERT  W(t)\VERT_{{\rm tan},s}^2
+ \|  ( \tilde{f},W)\|_{H^s(\varOmega_t)}^2
+ \mathring{\rm C}_{s+2} \|    ( \tilde{f},W)\|_{L^{\infty}(\varOmega_t)}^2 .
\label{normal.est2}
\end{align}
Estimate \eqref{normal.est} follows from \eqref{normal.est2} with $\ell=s$.
Then the proof of Proposition \ref{lem.normal} is complete.

\section{Estimate of the Tangential Derivatives}\label{sec.Tangential}

In this section, we establish the estimate for the tangential derivatives of solutions
of the linearized problem \eqref{ELP2}.

\begin{proposition} \label{lem.tan}
If the assumptions in Theorem {\rm \ref{thm2}} are satisfied, then
\begin{align}
\label{tan.est}
\VERT W(t)\VERT_{{\rm tan}, s}^2
\lesssim   \mathcal{M}_s(t)
+ \|\underline{\mathring{\rm c}}_1\|_{H^{3}(\varOmega_T)}   \VERT W(t)\VERT_{s}^2
\end{align}
for any  constant $\epsilon>0$, where $\varPsi$ is given in \eqref{varPsi.def} and
\begin{align} \label{M.cal}
\mathcal{M}_s(t):=
\left\{
\begin{aligned}
&\| (W,\varPsi,\tilde{f})\|_{H^{1}({\varOmega_t})}^2 &\ \textrm{if } s=1,\\
&\| (W,\varPsi,\tilde{f})\|_{H^{s}({\varOmega_t})}^2 +
\mathring{\rm C}_{s+2}\|(W,\,\varPsi,\,\tilde{f}) \|^2_{H^{3}(\varOmega_t)}&\ \textrm{if } s\geq 2.
\end{aligned}
\right.
\end{align}
\end{proposition}

The rest of this section is concerned with the proof of Proposition \ref{lem.tan}.

\subsection{Prelude}\label{sec.tan1}
Applying operator $\mathrm{D}_{\rm tan}^{\beta}:=\p_t^{\beta_0}\p_2^{\beta_2}\cdots\p_d^{\beta_d}$
with $|\beta|\leq s$  to system \eqref{ELP3.a}, we obtain
\begin{align} \label{tan.id1}
\mathring{\mathcal{A}}_0^{\pm}\p_t\mathrm{D}_{\rm tan}^{\beta}W^{\pm}+\mathring{\mathcal{A}}_j^{\pm}\p_j \mathrm{D}_{\rm tan}^{\beta}W^{\pm}=R^{\pm},
\end{align}
where
\begin{align} \notag
R^{\pm}:=\mathrm{D}_{\rm tan}^{\beta}(\mathring{J}_{\pm}^{\mathsf{T}}\tilde{f}^{\pm})
-\mathrm{D}_{\rm tan}^{\beta}(\mathring{\mathcal{A}}_4^{\pm}W^{\pm})
-[\mathrm{D}_{\rm tan}^{\beta}, \mathring{\mathcal{A}}_0^{\pm}]\p_t W^{\pm}
-[\mathrm{D}_{\rm tan}^{\beta},  \mathring{\mathcal{A}}_j^{\pm}]\p_j W^{\pm}.
\end{align}
Take the scalar product of \eqref{tan.id1} with $\mathrm{D}_{\rm tan}^{\beta} W^{\pm}$ to obtain
\begin{align}
\sum_{\pm}\int_{\varOmega}\mathring{\mathcal{A}}_0^{\pm}\mathrm{D}_{\rm tan}^{\beta}W^{\pm}\cdot \mathrm{D}_{\rm tan}^{\beta}W^{\pm}
=\mathcal{R}_1+\int_{\omega_t} Q,
\label{tan.id2}
\end{align}
where
\begin{align}
\notag \mathcal{R}_1
:=&\sum_{\pm}\int_{\varOmega_t}\mathrm{D}_{\rm tan}^{\beta}
 W^{\pm}\cdot \Big(2R^{\pm}+\big(\p_t \mathring{\mathcal{A}}_0^{\pm}+\p_j\mathring{\mathcal{A}}_j^{\pm} \big)\mathrm{D}_{\rm tan}^{\beta} W^{\pm}\Big),\\
\label{Q.def}
Q:=&\sum_{\pm}\mathring{\mathcal{A}}_{1a}^{\pm}\mathrm{D}_{\rm tan}^{\beta} W^{\pm}\cdot \mathrm{D}_{\rm tan}^{\beta} W^{\pm}
=\underbrace{2[\mathrm{D}_{\rm tan}^{\beta} W_2 \mathrm{D}_{\rm tan}^{\beta} W_{d+2}]}_{Q_1}+Q_2,
\end{align}
with
\begin{align}
\label{Q2.def}
Q_2:=
\left\{
\begin{aligned}
&-2\mathring{\rho}^+\mathring{{F}}_{1N}^+\big[\mathrm{D}_{\rm tan}^{\beta}W_3 \mathrm{D}_{\rm tan}^{\beta} W_5\big]
\quad &&\textrm{if }d=2,\\[1.5mm]
& -2\mathring{\rho}^+\mathring{{F}}_{1N}^+
\big[\mathrm{D}_{\rm tan}^{\beta}W_3 \mathrm{D}_{\rm tan}^{\beta} W_6
+ \mathrm{D}_{\rm tan}^{\beta} W_4 \mathrm{D}_{\rm tan}^{\beta} W_7\big]\quad &&\textrm{if }d=3.
\end{aligned}
\right.
\end{align}
Here and hereafter, for simplicity, we omit the differential symbol of the variables of integration 
when no confusion arises.

A standard computation with an application of the Moser-type calculus inequalities \eqref{Moser3}--\eqref{Moser4}
and the Sobolev embedding
$H^3(\varOmega_t)\hookrightarrow L^{\infty}(\varOmega_t)$ yields
\begin{align}
\mathcal{R}_1\lesssim   \mathcal{M}_s(t).
\label{R1.est}
\end{align}

We introduce the instant tangential energy $\mathcal{E}_{\rm tan}^{\beta}(t)$ as
\begin{align}
\notag \mathcal{E}_{\rm tan}^{\beta}(t):=&
\sum_{\pm}\int_{\varOmega}{\mathcal{A}}_0(\widebar{U}^{\pm},\widebar{\varPhi}^{\pm}) \mathrm{D}_{\rm tan}^{\beta}W^{\pm}\cdot \mathrm{D}_{\rm tan}^{\beta}W^{\pm},
\end{align}
where  $\mathcal{A}_0$ is given in \eqref{A.cal}. Thanks to \eqref{A0.cal}, we have
\begin{align}
\notag \mathcal{E}_{\rm tan}^{\beta}(t)  =&
\sum_{\pm} \Bigg\{\frac{1}{\bar{\rho}^{\pm}\bar{c}_{\pm}^2 }\|\mathrm{D}_{\rm tan}^{\beta}W_1^{\pm} \|_{L^2(\varOmega)}^2
+\frac{1}{\bar{\rho}^{\pm}(\widebar{F}_{11}^{\pm})^2 } \|\mathrm{D}_{\rm tan}^{\beta}(W_1^{\pm}-W_{d+2}^{\pm}) \|_{L^2(\varOmega)}^2\\
&\ \ \qquad  +\sum_{j=2,\,j\neq d+2}^{d^2+d+1}\bar{\rho}^{\pm} \|\mathrm{D}_{\rm tan}^{\beta}W_j^{\pm} \|_{L^2(\varOmega)}^2
+\|\mathrm{D}_{\rm tan}^{\beta}S^{\pm} \|_{L^2(\varOmega)}^2\Bigg\},
\label{E.tan.def}
\end{align}
where $\bar{\rho}^{\pm}:=(\det \widebar{\bm{F}}^{\pm})^{-1}$
and $\bar{c}_{\pm}:= p_{\rho}(\bar{\rho}^{\pm},\widebar{ S }^{\pm})^{1/2}$.

Since $\mathring{\mathcal{A}}_0^{\pm}-\mathcal{A}_0(\widebar{U}^{\pm},\widebar{\varPhi}^{\pm})
$ are smooth functions of $\{(\mathrm{D}^{\alpha}\mathring{V},\, \mathrm{D}^{\alpha}\mathring{\varPsi}): |\alpha|\leq 1 \}$ and vanish at the origin,
we plug \eqref{R1.est} into \eqref{tan.id2} to infer
\begin{align}
\mathcal{E}_{\rm tan}^{\beta}(t)\leq
\,&C \mathcal{M}_s(t)
+C \|\underline{\mathring{\rm c}}_1\|_{L^{\infty}(\varOmega_T)}\VERT W(t)\VERT_{s}^2
+\int_{\omega_t} Q,
\label{tan.est2}
\end{align}
where $\mathcal{M}_s(t)$ and $Q$  are defined by \eqref{M.cal} and \eqref{Q.def}, respectively.

\subsection{Cancellation}\label{sec.tan2}
We are going to show a cancellation for the last term in \eqref{tan.est2}.
By virtue of the boundary conditions \eqref{ELP3.b.2}--\eqref{ELP3.b.3}, we find
\begin{align}
\notag Q_1=\,&2 \mathrm{D}_{\rm tan}^{\beta} [W_{d+2}]\mathrm{D}_{\rm tan}^{\beta}  W_2^++2\mathrm{D}_{\rm tan}^{\beta} [W_2]\mathrm{D}_{\rm tan}^{\beta} W_{d+2}^-\\
\label{Q1} =\,&Q_{1a}+[\mathrm{D}_{\rm tan}^{\beta} ,\,\mathring{\rm c}_0]W \mathrm{D}_{\rm tan}^{\beta}  W_2^++\mathrm{D}_{\rm tan}^{\beta}(\mathring{\rm c}_1\psi )\mathrm{D}_{\rm tan}^{\beta} W
\qquad \textrm{on $\p\varOmega$}
\end{align}
with
\begin{align} \notag
Q_{1a}:=2[\mathring{F}_{11}] \p_{F_{ij}} \varrho(\mathring{\bm{F}}^+)\mathrm{D}_{\rm tan}^{\beta} F_{ij}^+\mathrm{D}_{\rm tan}^{\beta} W_2^+.
\end{align}
Similarly, it follows from \eqref{key1.bas}, \eqref{ELP3.b.2}, and \eqref{ELP3.b.4} that
\begin{align}
\notag Q_2=\,&-2\varrho(\mathring{\bm{F}}^+)\sum_{j=2}^{d} \left(\mathrm{D}_{\rm tan}^{\beta}[W_{d+j+1}]\mathrm{D}_{\rm tan}^{\beta}  W_{j+1}^+
  +\mathrm{D}_{\rm tan}^{\beta} [W_{j+1}]\mathrm{D}_{\rm tan}^{\beta} W_{d+j+1}^-\right)\\
 =\,&Q_{2a}+\mathring{\rm c}_0 [\mathrm{D}_{\rm tan}^{\beta} ,\,\mathring{\rm c}_0] \mathrm{D}_{\rm tan} \psi \mathrm{D}_{\rm tan}^{\beta}  W
+\mathring{\rm c}_0   \mathrm{D}_{\rm tan}^{\beta}(\mathring{\rm c}_1\psi )\mathrm{D}_{\rm tan}^{\beta} W
\ \  \textrm{on $\p\varOmega$} \label{Q2}
\end{align}
with
\begin{align}\notag
Q_{2a}:=
2 \varrho(\mathring{\bm{F}}^+) [\mathring{F}_{11}] \sum_{j=2}^d \mathrm{D}_{\rm tan}^{\beta}  \p_j \psi \mathrm{D}_{\rm tan}^{\beta}  W_{j+1}^+ .
\end{align}
We decompose $Q_{2a}$ further as
\begin{align}\label{Q2a}
Q_{2a}
=\,&Q_{2b}+\sum_{j=2}^d\p_j\big(2\varrho(\mathring{\bm{F}}^+)[\mathring{F}_{11}]\mathrm{D}_{\rm tan}^{\beta}\psi \mathrm{D}_{\rm tan}^{\beta}  W_{j+1}^+\big)
+\mathring{\rm c}_1 \mathrm{D}_{\rm tan}^{\beta} W \mathrm{D}_{\rm tan}^{\beta} \psi
\end{align}
with
\begin{align}\notag
Q_{2b}:=
-2\varrho(\mathring{\bm{F}}^+) [\mathring{F}_{11}]
\sum_{j=2}^d  \mathrm{D}_{\rm tan}^{\beta}\psi \mathrm{D}_{\rm tan}^{\beta} \p_j W_{j+1}^+ .
\end{align}

In order to deduce the cancellation between terms $Q_{1a}$ and $Q_{2b}$,
we need the following lemma.

\begin{lemma}\label{lem.key2}
If $i=1,\ldots,d$, and $j=2,\ldots,d$, then
 \begin{alignat}{3}
 \label{key2}
& \p_0 F_{ij}^{\pm} =\sum_{k=2}^d \mathring{ {F}}_{kj}^{\pm} \p_k v_i^{\pm} +\mathring{\rm c}_0 \tilde{f} +\underline{\mathring{\rm c}}_1 W
 \qquad &&\textrm{ on $\p\varOmega$},
\\ \label{key3b}
&\sum_{j=2}^d\p_j W_{j+1}^+
=-\varrho(\mathring{\bm{F}}^+)^{-1} \p_{F_{ij}}\varrho(\mathring{\bm{F}}^+) \p_0 F_{ij}^+
+\mathring{\rm c}_0 \tilde{f} +\underline{\mathring{\rm c}}_1 W
 \qquad &&\textrm{ on $\p\varOmega$},
\end{alignat}
where $\p_0$ is defined by \eqref{partial0}.
\end{lemma}

\begin{proof}
Considering the restriction of equations \eqref{F.equ} on boundary $\p\varOmega$,
we  utilize \eqref{w1.ring.bdy} and \eqref{bas.c3b} to deduce identities \eqref{key2}.
In the two-dimensional case ($d=2$), relation \eqref{key3b} follows directly from \eqref{key2}.
If $d=3$, then we obtain from \eqref{key2} that
\begin{align}
\notag
\begin{pmatrix}
\p_2 v_2^{\pm} & \p_3 v_2^{\pm} \\[1.5mm]
\p_2 v_3^{\pm} & \p_3 v_3^{\pm}
\end{pmatrix}
\begin{pmatrix}
\mathring{ {F}}_{22}^{\pm} &\mathring{ {F}}_{23}^{\pm}\\[1.5mm]
 \mathring{ {F}}_{32}^{\pm} &\mathring{ {F}}_{33}^{\pm}
\end{pmatrix}
=\begin{pmatrix}
\p_0  { {F}}_{22}^{\pm} &\p_0 { {F}}_{23}^{\pm}\\[1.5mm]
\p_0 { {F}}_{32}^{\pm} &\p_0  { {F}}_{33}^{\pm}
\end{pmatrix}
+\mathring{\rm c}_0 \tilde{f} +\underline{\mathring{\rm c}}_1 W
\qquad \textrm{ on $\p\varOmega$}.
\end{align}
Then we can deduce \eqref{key3b} by virtue of $W_3^+=v_2^+$, $W_4^+=v_3^+$, and
\begin{align*}
-\frac{\p_{F_{ij}}\varrho(\mathring{\bm{F}}^+)\p_0 F_{ij}^+}{\varrho(\mathring{\bm{F}}^+)^2}
=\mathring{ {F}}_{22}^+\p_0 F_{33}^+-\mathring{ {F}}_{23}^+\p_0 F_{32}^+-\mathring{ {F}}_{32}^+\p_0 F_{23}^++\mathring{ {F}}_{33}^+\p_0 F_{22}^+.
\end{align*}
This completes the proof.
\qed\end{proof}

Thanks to identity \eqref{key3b}, we find
\begin{align}
\notag Q_{2b}=\,&
\underbrace{2\varrho(\mathring{\bm{F}}^+)  [\mathring{F}_{11}]
 \mathrm{D}_{\rm tan}^{\beta}   \psi
 \mathrm{D}_{\rm tan}^{\beta} \big(\varrho(\mathring{\bm{F}}^+)^{-1} \p_{F_{ij}}\varrho(\mathring{\bm{F}}^+) \p_0 F_{ij}^+\big)
}_{\large Q_{2c}} \\[2.5mm]
\label{Q2b}&-2\varrho(\mathring{\bm{F}}^+) [\mathring{F}_{11}]
\mathrm{D}_{\rm tan}^{\beta}\psi
\mathrm{D}_{\rm tan}^{\beta}\big(\mathring{\rm c}_0 \tilde{f} +\mathring{\rm c}_1 W\big)
\qquad \quad
\textrm{on $\p\varOmega$}.
\end{align}
Term $Q_{2c}$ can be decomposed   further as
\begin{align}
\notag
Q_{2c}&=
2 [\mathring{F}_{11}] \p_{F_{ij}}\varrho(\mathring{\bm{F}}^+)  \mathrm{D}_{\rm tan}^{\beta}   \psi \, \p_0 \mathrm{D}_{\rm tan}^{\beta}  F_{ij}^+
+\mathring{\rm c}_0 \mathrm{D}_{\rm tan}^{\beta}   \psi \, [\mathrm{D}_{\rm tan}^{\beta},\, \mathring{\rm c}_0 ] \mathrm{D}_{\rm tan} W\\
\notag& =\mathring{\rm c}_0 \mathrm{D}_{\rm tan}^{\beta} \psi [\mathrm{D}_{\rm tan}^{\beta}, \mathring{\rm c}_0 ] \mathrm{D}_{\rm tan} W
+\sum_{j=2}^d\p_j\left\{2 [\mathring{F}_{11}]  \p_{F_{ij}}\varrho(\mathring{\bm{F}}^+) \mathring{v}_j^+ \mathrm{D}_{\rm tan}^{\beta}\psi \mathrm{D}_{\rm tan}^{\beta}  F_{ij}^+
\right\}\\
\notag &\quad
+ \p_t\left\{ 2[\mathring{F}_{11}] \p_{F_{ij}}\varrho(\mathring{\bm{F}}^+)\mathrm{D}_{\rm tan}^{\beta} \psi \mathrm{D}_{\rm tan}^{\beta}  F_{ij}^+\right\}
+ \mathring{\rm c}_1  \mathrm{D}_{\rm tan}^{\beta}   \psi  \,\mathrm{D}_{\rm tan} ^{\beta}W
\\[3mm]
\label{Q2c} &\quad
\underbrace{-2 [\mathring{F}_{11}]  \p_{F_{ij}}\varrho(\mathring{\bm{F}}^+)   \mathrm{D}_{\rm tan}^{\beta}  F_{ij}^+ \mathrm{D}_{\rm tan}^{\beta}   \p_0 \psi}_{Q_{2d}}
+\mathring{\rm c}_0 \mathrm{D}_{\rm tan} ^{\beta}W [\mathrm{D}_{\rm tan}^{\beta}  ,\mathring{\rm c}_0]\mathrm{D}_{\rm tan} \psi .
\end{align}
In view of condition \eqref{ELP3.b.1}, we derive the following desired cancellation:
\begin{align}
\label{cancellation}
Q_{1a}+Q_{2d}
=2 [\mathring{F}_{11}]  \p_{F_{ij}}\varrho(\mathring{\bm{F}}^+)
\mathrm{D}_{\rm tan}^{\beta}  F_{ij}^+ \mathrm{D}_{\rm tan}^{\beta}  (\mathring{\rm c}_1 \psi)
\qquad \textrm{on $\p\varOmega$}.
\end{align}
Combine \eqref{Q1}--\eqref{Q2a} and \eqref{Q2b}--\eqref{cancellation} to obtain
\begin{align} \label{Q.est1}
\int_{\omega_t} Q\,
=\mathcal{R}_2+
\underbrace{2 \int_{\p\varOmega } [\mathring{F}_{11}]  \p_{F_{ij}}\varrho(\mathring{\bm{F}}^+)  \mathrm{D}_{\rm tan}^{\beta}   \psi  \mathrm{D}_{\rm tan}^{\beta}  F_{ij}^+\, }_{\mathcal{R}_3},
\end{align}
where
\begin{align}
\notag \mathcal{R}_2:=&\int_{\omega_t}
[\mathrm{D}_{\rm tan}^{\beta} ,\,\mathring{\rm c}_0]W \mathrm{D}_{\rm tan}^{\beta}  W
+ \int_{\omega_t}  \mathring{\rm c}_0 \mathrm{D}_{\rm tan}^{\beta}(\mathring{\rm c}_1\psi )\mathrm{D}_{\rm tan}^{\beta} W  \\
\notag &+\int_{\omega_t} \mathring{\rm c}_0 [\mathrm{D}_{\rm tan}^{\beta} ,\,\mathring{\rm c}_0] \mathrm{D}_{\rm tan} \psi \mathrm{D}_{\rm tan}^{\beta}  W
+\int_{\omega_t} \mathring{\rm c}_0 \mathrm{D}_{\rm tan}^{\beta} \psi [\mathrm{D}_{\rm tan}^{\beta}, \mathring{\rm c}_0 ] \mathrm{D}_{\rm tan} W\\
\label{R2.def} &
+\int_{\omega_t} \mathring{\rm c}_1   \mathrm{D}_{\rm tan}^{\beta} \psi \mathrm{D}_{\rm tan}^{\beta} W
+\int_{\omega_t} \mathring{\rm c}_0  \mathrm{D}_{\rm tan}^{\beta}   \psi
\mathrm{D}_{\rm tan}^{\beta}\big(\mathring{\rm c}_0 \tilde{f} +\mathring{\rm c}_1 W\big).
\end{align}

\subsection{Estimate of Term $\mathcal{R}_2$}\label{sec.R2}
In this subsection, we deduce the estimate of term $\mathcal{R}_2$ defined by \eqref{R2.def}.

By virtue of assumption \eqref{bas.bound} and the Sobolev embedding,
there exists some positive constant $K_1$ depending on $[\widebar{F}_{11}]$ such that,
if $K\leq K_1$, then
\begin{align*}
[\mathring{ {F}}_{11}]\geq \frac{[\widebar{F}_{11}]}{2}>0\qquad\,\, \textrm{on $\p\varOmega$}.
\end{align*}
It follows from the boundary condition \eqref{ELP3.b.4} that
\begin{align}
\label{key4}
\p_j \psi= -\frac{1}{[\mathring{ {F}}_{11}]} [W_{d+j+1}]+
\underline{\mathring{\rm c}}_1 \psi \qquad\,
\textrm{on $\p\varOmega$, for $j=2,\ldots,d$}.
\end{align}
If we utilize \eqref{key4} to control terms $\mathcal{R}_2$ and $\mathcal{R}_3$, then
the energy estimates break down when $[\widebar{F}_{11}]$ tends to zero.
Hence, identity \eqref{key4} cannot be used in the subsequent analysis for the proof of Theorem \ref{thm2}.
Then we need to exploit new identities and estimates for $\mathrm{D}_{x'} \psi$.
For this purpose, we apply the interpolation argument to
deduce the following lemma, which is motivated by \cite[Proposition 5.2]{T18MR3721411}.
\begin{lemma} \label{lem.key}
	If the assumptions in Theorem {\rm \ref{thm2}} are satisfied, then
	\begin{align} \label{Rj.def}
	R_j:=F_{j}^+\cdot \mathring{N} -\sum_{i=2}^d\mathring{ {F}}_{ij}^+\p_i\psi
	\qquad\,\, \textrm{defined on $\omega_T$, for $j=2,\ldots,d$},
	\end{align}
satisfies
	\begin{align} \label{Rj.est}
	\|\mathrm{D}_{\rm tan}^{\gamma} R_j(t) \|_{H^{s-|\gamma|-1/2}(\p\varOmega)}^2\lesssim  \mathcal{M}_s(t),
	\end{align}
	for all $t\in [0,T]$ and $\gamma \in \mathbb{N}^d$ with $|\gamma|\leq s-1$,
	where $\mathcal{M}_s(t)$ is defined by \eqref{M.cal}.
\end{lemma}

\begin{proof}
	Thanks to  \eqref{bas.c3b} and \eqref{key2}, we have
	\begin{align*}
	\p_0 F_j^+\cdot  \mathring{N}
	&=\sum_{k=2}^d \mathring{ {F}}_{k j}^+ \p_{k} (v^{+}\cdot\mathring{N}  )
	+\sum_{k,i=2}^d v^+_{i} \mathring{ {F}}_{k j}^+ \p_{i} \p_k \mathring{\varphi}
	+\mathring{\rm c}_1 \tilde{f} +\underline{\mathring{\rm c}}_1 W\\
	&=\sum_{i=2}^d \mathring{ {F}}_{i j}^+ \p_{i} (v^{+}\cdot\mathring{N}  )
	+\mathring{\rm c}_1 \tilde{f} +{\mathring{\rm c}}_1 W.
	\end{align*}
	Since, for $k=2,\ldots,d$,
	\begin{align}
	\p_0\p_{k}\mathring{\varphi} =
	\p_{k} (\mathring{v}^+\cdot \mathring{N})+\mathring{v}^+_i\p_i\p_{k} \mathring{\varphi}
	=\p_{k} \mathring{v}^+\cdot \mathring{N}\qquad
	\textrm{on $\p\varOmega$},
	\label{bas.id2}
	\end{align}
we have
\begin{align*}
\p_0 (F_j^+\cdot  \mathring{N} )
=\sum_{i=2}^d \mathring{ {F}}_{i j}^+ \p_{i} (v^{+}\cdot\mathring{N}  )
+\mathring{\rm c}_1 \tilde{f} +{\mathring{\rm c}}_1 W.
\end{align*}
It follows from \eqref{bas.c6} and  \eqref{ELP3.b.1} that
	\begin{align*}
	-\sum_{i=2}^d \p_0 (\mathring{ {F}}_{i j}^+\p_{i} \psi)
	&=
-\sum_{i=2}^d  \mathring{ {F}}_{i j}^+
 \Big(\p_{i}\p_0 \psi  -\sum_{\ell =2}^d \p_i \mathring{v}_{\ell}^+ \p_{\ell} \psi  \Big)
	-\sum_{i,\ell=2}^d \mathring{ {F}}_{\ell  j}^+  \p_{\ell}   \mathring{v}_i^+\p_{i} \psi
	\\
	&=-\sum_{i=2}^d  \mathring{ {F}}_{i j}^+
	\big(\p_{i}(v^{+}\cdot\mathring{N} ) +\mathring{\rm c}_1 \p_i \psi\big)
	+ \mathring{\rm c}_2 \psi .
	\end{align*}
Thanks to \eqref{bas.id2}, we have
\begin{align} \notag
	\p_0 R_j+ \mathring{\rm c}_1 R_j=\mathring{\rm c}_1 \tilde{f} +{\mathring{\rm c}}_1 W
	+\mathring{\rm c}_2 \psi\qquad \textrm{on $\p\varOmega$}.
\end{align}
Using the standard arguments of the energy method yields
	\begin{align*}
	\|\mathrm{D}_{\rm tan}^{\gamma} R_j(t)\|_{H^{m}(\p\varOmega)}\lesssim
	\|\mathring{\rm c}_1 \tilde{f} +{\mathring{\rm c}}_1 W
	+\mathring{\rm c}_2 \psi\|_{H^{m+|\gamma|}(\omega_t)}
	\qquad\, \textrm{for $ m\in \mathbb{N}$}.
	\end{align*}
Applying the interpolation property (see \cite[Lemma 22.3]{T07MR2328004}),
the trace theorem, and the Moser-type calculus inequality, we have
	\begin{align*}
	\|\mathrm{D}_{\rm tan}^{\gamma} R_j(t)\|_{H^{s-|\gamma|-1/2}(\p\varOmega)}&\lesssim
	\|\mathring{\rm c}_1 \tilde{f} +{\mathring{\rm c}}_1 W
	+\mathring{\rm c}_2 \psi\|_{H^{s-1/2}(\omega_t)}\\
	&\lesssim \|\mathring{\rm c}_1 \tilde{f} +{\mathring{\rm c}}_1 W
	+\mathring{\rm c}_2 \varPsi\|_{H^{s}(\varOmega_t)}\lesssim \sqrt{\mathcal{M}_s(t)},
	\end{align*}
where we utilize
$\|\mathring{\rm c}_2\|_{W^{1,\infty}(\varOmega_t) }\lesssim K$
and
$\|(W,\,\varPsi,\,\tilde{f}) \|_{L^{\infty}(\varOmega_t)}\lesssim
\|(W,\,\varPsi,\,\tilde{f}) \|_{H^{3}(\varOmega_t)}$
by the Sobolev embedding theorem.
This completes the proof.
\qed\end{proof}

By virtue of \eqref{Rj.def}, from \eqref{background} and \eqref{varrho.bar},
we obtain the following assertions:
\begin{itemize}
\item If $d=2$, then
\begin{align}
\p_2\psi=(\mathring{ {F}}_{22}^+)^{-1}(F_{12}^+ - \p_2 \mathring{\varphi} F_{22}^+-R_2)
= \varrho(\widebar{\bm{F}}^+)F_{12}^+
 +\underline{\mathring{\rm c}}_1 W +\mathring{\rm c}_0 R_2.\label{psi.d2.id.2D}
\end{align}

\item  If $d=3$, then
\begin{align*}
\begin{pmatrix}
\p_2 \psi\\[1mm] \p_3 \psi
\end{pmatrix}
=\varrho (\mathring{\bm{F}}^+)
\begin{pmatrix}
\mathring{{F}}_{33}^+ &-\mathring{{F}}_{32}^+\\[1mm]
-\mathring{{F}}_{23}^+&\mathring{{F}}_{22}^+
\end{pmatrix}
\begin{pmatrix}
F_{2}^+\cdot \mathring{N}-R_2\\[1mm]
F_{3}^+\cdot \mathring{N}-R_3
\end{pmatrix},
\end{align*}
which implies
	\begin{align}
	\label{psi.d2.id}&\p_2\psi= \varrho(\widebar{\bm{F}}^+)\widebar{F}_{33} F_{12}^++\underline{\mathring{\rm c}}_1 W +\mathring{\rm c}_0 R_2+\mathring{\rm c}_0 R_3,\\
	\label{psi.d3.id}&\p_3\psi= \varrho(\widebar{\bm{F}}^+)\widebar{F}_{22} F_{13}^++\underline{\mathring{\rm c}}_1 W +\mathring{\rm c}_0 R_2+\mathring{\rm c}_0 R_3.
	\end{align}
\end{itemize}
Identities \eqref{psi.d2.id.2D}--\eqref{psi.d3.id} and estimate \eqref{Rj.est} enable us to control term
$\mathcal{R}_2$. More precisely, from \eqref{psi.d2.id.2D}--\eqref{psi.d3.id} and \eqref{ELP3.b.1},
we have
\begin{align}
\label{key4'}
\mathrm{D}_{\rm tan} \psi= \mathring{\rm c}_1 W+\sum_{j=2}^d \mathring{\rm c}_0 R_j \qquad \textrm{on $\p\varOmega$},
\end{align}
where coefficients $\mathring{\rm c}_1$ and $\mathring{\rm c}_0$ are
 independent of $[\widebar{F}_{11}]$.
Assume without loss of generality that $0<\beta'\leq \beta$, $|\beta'|=1$, and $|\beta|\leq s$.
For the last term in $\mathcal{R}_2$,
we employ \eqref{key4'} to obtain
\begin{align}
\notag &\int_{\omega_t} \mathring{\rm c}_0  \mathrm{D}_{\rm tan}^{\beta}   \psi
\mathrm{D}_{\rm tan}^{\beta}   \big(\mathring{\rm c}_0 \tilde{f} +\mathring{\rm c}_1 W\big)\\[-1mm]
\notag &\quad \lesssim
\Big\| \mathring{\rm c}_0  \mathrm{D}_{\rm tan}^{\beta-\beta'}
\big( \mathring{\rm c}_1 W+\sum_{j=2}^d \mathring{\rm c}_0 R_j \big)\Big\|_{H^{1/2}(\omega_t)}
\Big\|\mathrm{D}_{\rm tan}^{\beta}   \big(\mathring{\rm c}_0 \tilde{f} +\mathring{\rm c}_1 W\big)\Big\|_{H^{-1/2}(\omega_t)}\\[-1mm]
\label{R2.est1}&\quad \lesssim
\Big\|  \mathring{\rm c}_1 W
+\sum_{j=2}^d \mathring{\rm c}_0 \widetilde{R}_j  \Big\|_{H^{s}(\varOmega_t)}
\Big\| \mathring{\rm c}_0 \tilde{f} +\mathring{\rm c}_1 W \Big\|_{H^{s}(\varOmega_t)},
\end{align}
where $\widetilde{R}_j$ is the extension of $R_j$ from $\omega_T$ to $\varOmega_T$ satisfying
\begin{align} \label{Rj.est2}
\|\widetilde{R}_j\|_{H^{m}(\varOmega_t)}\lesssim \|{R}_j\|_{H^{m-1/2}(\omega_t)}
\qquad \textrm{for $m=1,\ldots,s$}.
\end{align}
Applying the Moser-type calculus inequality to \eqref{R2.est1}
and using estimates \eqref{Rj.est} and \eqref{Rj.est2}, we obtain
\begin{align}
\notag  \int_{\omega_t} \mathring{\rm c}_0  \mathrm{D}_{\rm tan}^{\beta}   \psi
\mathrm{D}_{\rm tan}^{\beta}   \big(\mathring{\rm c}_0 \tilde{f} +\mathring{\rm c}_1 W\big)
 \lesssim \mathcal{M}_s(t).
\end{align}
As the other terms in \eqref{R2.def}
can be handled similarly, we omit the details and conclude
\begin{align}
\mathcal{R}_2  \lesssim \mathcal{M}_s(t).
\label{R2.est}
\end{align}

\subsection{Estimate of Term $\mathcal{R}_3$ with the Time Derivative 
}\label{sec.tan4.a}
This subsection is devoted to deriving the estimate of term $\mathcal{R}_3$ given in \eqref{Q.est1}
for $\beta=(\beta_0,\beta_2,\ldots,\beta_d)$ satisfying $\beta_0\geq 1$
and $|\beta|\leq s$.

Recalling the definition of background state $(\widebar{U}^{\pm},\widebar{\varPhi}^{\pm})$
in \eqref{background}
and using identity \eqref{key2}, we have
\begin{align} \label{key2b}
\p_t F_{ij}^+&
=  \widebar{F}_{jj} \p_j v_i^+
+\underline{\mathring{\rm c}}_0 \mathrm{D}_{x'} W +\underline{\mathring{\rm c}}_1 W+\mathring{\rm c}_0 \tilde{f}
\quad \textrm{on $\p\varOmega$, for $i,j=2,\ldots,d$},
\end{align}
where $\mathrm{D}_{x'}:=(\p_2,\ldots,\p_d)$.
In light of \eqref{key2b}, we compute
\begin{align}
\notag &[\mathring{F}_{11}] \p_{F_{ij}}\varrho(\mathring{\bm{F}}^+)   \mathrm{D}_{\rm tan}^{\beta}  F_{ij}^+\\[2mm]
&\quad =\notag  [\mathring{F}_{11}] \p_{F_{ij}}\varrho(\mathring{\bm{F}}^+)
\mathrm{D}_{\rm tan}^{\beta-\bm{e}_1}
\Big(\widebar{F}_{jj} \p_j v_i^+
+\underline{\mathring{\rm c}}_0 \mathrm{D}_{x'} W +\underline{\mathring{\rm c}}_1 W+\mathring{\rm c}_0 \tilde{f} \Big)\\
\notag &\quad =-\sum_{j=2}^d\, [\widebar{F}_{11}] \varrho(\widebar{\bm{F}}^+)\mathrm{D}_{\rm tan}^{\beta-\bm{e}_1}  \p_j v_j^+
+\underline{\mathring{\rm c}}_0 \mathrm{D}_{\rm tan}^{\beta-\bm{e}_1} \left( \mathring{\rm c}_0 \mathrm{D}_{x'} W\right) \\
\label{R3.id1}&\quad  \quad
+\mathring{\rm c}_0 \mathrm{D}_{\rm tan}^{\beta-\bm{e}_1}
\big(\underline{\mathring{\rm c}}_0 \mathrm{D}_{x'}   W
+\underline{\mathring{\rm c}}_1 W+\mathring{\rm c}_0 \tilde{f} \big)
\qquad\, \textrm{on $\p\varOmega$}.
\end{align}
Noting from \eqref{ELP3.b.1} that
\begin{align}
\label{psi.t.id}
\p_t \psi=W_2^++\underline{\mathring{\rm c}}_0  \mathrm{D}_{x'}\psi +\underline{\mathring{\rm c}}_1\psi \qquad
\textrm{on $\p\varOmega$},
\end{align}
we have
\begin{align} \label{R3.est1}
\mathcal{R}_3=
2 \int_{\p\varOmega } [\mathring{F}_{11}]  \p_{F_{ij}}\varrho(\mathring{\bm{F}}^+)
  \mathrm{D}_{\rm tan}^{\beta}  F_{ij}^+\mathrm{D}_{\rm tan}^{\beta}   \psi
=\sum_{i=1}^5 \mathcal{R}_{3i},
\end{align}
where
\begin{align*}
\mathcal{R}_{31} &:=-\sum_{j=2}^d\, 2 [\widebar{F}_{11}] \varrho(\widebar{\bm{F}}^+)\int_{\p\varOmega }    \mathrm{D}_{\rm tan}^{\beta-\bm{e}_1} W_2^+  \mathrm{D}_{\rm tan}^{\beta-\bm{e}_1} \p_j v_j^+ ,\\
\mathcal{R}_{32} &:=-\sum_{j=2}^d\, 2 [\widebar{F}_{11}] \varrho(\widebar{\bm{F}}^+)
\int_{\p\varOmega }    \mathrm{D}_{\rm tan}^{\beta-\bm{e}_1}
\left(\underline{\mathring{\rm c}}_0  \mathrm{D}_{x'}\psi \right) \mathrm{D}_{\rm tan}^{\beta-\bm{e}_1} \p_j v_j^+,\\
\mathcal{R}_{33} &:=-\sum_{j=2}^d\, 2 [\widebar{F}_{11}] \varrho(\widebar{\bm{F}}^+)
\int_{\p\varOmega }    \mathrm{D}_{\rm tan}^{\beta-\bm{e}_1}
\left(\underline{\mathring{\rm c}}_1\psi\right) \mathrm{D}_{\rm tan}^{\beta-\bm{e}_1} \p_j v_j^+,\\
\mathcal{R}_{34} &:=\int_{\p\varOmega }    \left(  \underline{\mathring{\rm c}}_0 \mathrm{D}_{\rm tan}^{\beta-\bm{e}_1} \left( \mathring{\rm c}_0 \mathrm{D}_{x'} W\right)
+\mathring{\rm c}_0 \mathrm{D}_{\rm tan}^{\beta-\bm{e}_1}  \mathrm{D}_{x'} (\underline{\mathring{\rm c}}_0  W)  \right)\mathrm{D}_{\rm tan}^{\beta} \psi ,\\[1.5mm]
\mathcal{R}_{35} &:=\int_{\p\varOmega }    \mathring{\rm c}_0 \mathrm{D}_{\rm tan}^{\beta-\bm{e}_1} \big( \underline{\mathring{\rm c}}_1 W+\mathring{\rm c}_0 \tilde{f} \,\big) \mathrm{D}_{\rm tan}^{\beta}\psi.
\end{align*}

Let us first estimate $\mathcal{R}_{32}$ as
\begin{align}
|\mathcal{R}_{32}|
 &\lesssim
 \bigg| \underbrace{\int_{\p\varOmega }
\underline{\mathring{\rm c}}_0  \mathrm{D}_{\rm tan}^{\beta-\bm{e}_1}    \mathrm{D}_{\rm \tan}\psi
 \mathrm{D}_{\rm tan}^{\beta-\bm{e}_1} \mathrm{D}_{x'} W}_{\mathcal{R}_{32}^a}
+ \underbrace{ \int_{\p\varOmega }
[\mathrm{D}_{\rm tan}^{\beta-\bm{e}_1}, \underline{\mathring{\rm c}}_0]    \mathrm{D}_{x'}\psi
\mathrm{D}_{\rm tan}^{\beta-\bm{e}_1}\mathrm{D}_{x'}   W }_{\mathcal{R}_{32}^b} \bigg|.
 \notag
\end{align}
In view of \eqref{key4'}, we employ the classical product estimate
$$
\|uv\|_{H^{1/2}(\mathbb{R}^{d-1})}\lesssim \|u\|_{H^{3/2}(\mathbb{R}^{d-1})}
\|v\|_{H^{1/2}(\mathbb{R}^{d-1})}
$$
to obtain
\begin{align}
\notag |\mathcal{R}_{32}^a| &\lesssim
\Big\|
\underline{\mathring{\rm c}}_0  \mathrm{D}_{\rm tan}^{\beta-\bm{e}_1}\big(
\mathring{\rm c}_1 W+\sum_{j=2}^d\mathring{\rm c}_0  R_j  \big)
 \Big\|_{H^{1/2}(\p\varOmega)}
\left\| \mathrm{D}_{\rm tan}^{\beta-\bm{e}_1} \mathrm{D}_{x'} W\right\|_{H^{-1/2}(\p\varOmega)}\\
&\lesssim
\left\|
\underline{\mathring{\rm c}}_0  \right\|_{H^{3}(\varOmega_t)}
\Big\| \mathrm{D}_{\rm tan}^{\beta-\bm{e}_1}\big(
\mathring{\rm c}_1 W+\sum_{j=2}^d \mathring{\rm c}_0 R_j  \big)
\Big\|_{H^{1/2}(\p\varOmega)}
\left\| \mathrm{D}_{\rm tan}^{\beta-\bm{e}_1}   W\right\|_{H^{1/2}(\p\varOmega)}.
\label{R32a.e1}
\end{align}
Utilize the trace theorem, \eqref{MTT.inequ}, and the Moser-type calculus inequality \eqref{Moser4} to
obtain
\begin{align}
\notag  \left\| \mathrm{D}_{\rm tan}^{\beta-\bm{e}_1}\left(
\mathring{\rm c}_1 W \right)
\right\|_{H^{1/2}(\p\varOmega)}^2 
&\lesssim
\left\| \mathring{\rm c}_1  \mathrm{D}_{\rm tan}^{\beta-\bm{e}_1} W
\right\|_{H^{1}(\varOmega)} ^2
+\left\|   [\mathrm{D}_{\rm tan}^{\beta-\bm{e}_1},\mathring{\rm c}_1] W
\right\|_{H^{2}(\varOmega_t)}^2  \\
&\lesssim
\VERT W(t)\VERT_s^2
+\mathcal{M}_s(t).
\label{R32a.e2}
\end{align}
It follows from \eqref{Rj.est}, the trace theorem, \eqref{MTT.inequ},
and \eqref{Rj.est2} that
\begin{align}
\notag
\left\| \mathrm{D}_{\rm tan}^{\beta-\bm{e}_1}\left(
\mathring{\rm c}_0 R_j  \right)
\right\|_{H^{1/2}(\p\varOmega)}^2
&\lesssim
\left\| \mathring{\rm c}_0 \mathrm{D}_{\rm tan}^{\beta-\bm{e}_1}   R_j
\right\|_{H^{1/2}(\p\varOmega)}^2
+\left\|  [\mathrm{D}_{\rm tan}^{\beta-\bm{e}_1}, \mathring{\rm c}_0]   R_j
\right\|_{H^{1/2}(\p\varOmega)}^2\\
&\lesssim
\mathcal{M}_s(t)
+\left\|  [\mathrm{D}_{\rm tan}^{\beta-\bm{e}_1}, \mathring{\rm c}_0]   \widetilde{R}_j
\right\|_{H^{2}(\varOmega_t)}^2
\lesssim \mathcal{M}_s(t).
\label{R32a.e3}
\end{align}
Plugging \eqref{R32a.e2}--\eqref{R32a.e3} into \eqref{R32a.e1} yields
\begin{align}
 |\mathcal{R}_{32}^a|  \lesssim
\left\|
\underline{\mathring{\rm c}}_0 \right\|_{H^{3}(\varOmega_t)} \VERT W(t)\VERT_s^2
+\mathcal{M}_s(t).
\label{R32a.e0}
\end{align}
For $\mathcal{R}_{32}^b$, we find
\begin{align}
\notag \mathcal{R}_{32}^b=
\,& -\int_{\p\varOmega}
\mathrm{D}_{x'} [\mathrm{D}_{\rm tan}^{\beta-\bm{e}_1}, \underline{\mathring{\rm c}}_0]    \mathrm{D}_{x'}\psi
\mathrm{D}_{\rm tan}^{\beta-\bm{e}_1}  W  \\
  =\,& -\int_{\omega_t }
\p_t\left\{\mathrm{D}_{x'} [\mathrm{D}_{\rm tan}^{\beta-\bm{e}_1}, \underline{\mathring{\rm c}}_0]    \mathrm{D}_{x'}\psi
\mathrm{D}_{\rm tan}^{\beta-\bm{e}_1}  W\right\}  .
\label{R32b.e1}
\end{align}
Hence, it follows from \eqref{Rj.est}, \eqref{key4'}, and \eqref{Rj.est2} that
\begin{align}
\notag |\mathcal{R}_{32}^b| \lesssim \,&
\left\| \p_t \mathrm{D}_{x'} [\mathrm{D}_{\rm tan}^{\beta-\bm{e}_1}, \underline{\mathring{\rm c}}_0]   \mathrm{D}_{x'}\psi
\right\|_{H^{-1/2}(\omega_t)}
\left\|\mathrm{D}_{\rm tan}^{\beta-\bm{e}_1}  W \right\|_{H^{1/2}(\omega_t)} \\[1mm]
\notag &+
\left\|  \mathrm{D}_{x'} [\mathrm{D}_{\rm tan}^{\beta-\bm{e}_1}, \underline{\mathring{\rm c}}_0]   \mathrm{D}_{x'}\psi
\right\|_{H^{1/2}(\omega_t)}
\left\|\p_t\mathrm{D}_{\rm tan}^{\beta-\bm{e}_1}  W \right\|_{H^{-1/2}(\omega_t)} \\
\notag \lesssim \,&
\left\|  W \right\|_{H^{s}(\varOmega_t)}
\Big\|  \mathrm{D}[\mathrm{D}_{\rm tan}^{\beta-\bm{e}_1}, \underline{\mathring{\rm c}}_0]
\Big( \mathring{\rm c}_1 W+  \sum_{j=2}^d\mathring{\rm c}_0 \widetilde{R}_j \Big)  \Big\|_{H^{1}(\varOmega_t)}\\
\lesssim \, &\mathcal{M}_s(t).
\label{R32b.e0}
\end{align}
We decompose $\mathcal{R}_{34}$ as
\begin{align*}
&\int_{\p\varOmega }  \underline{\mathring{\rm c}}_0   \mathrm{D}_{\rm tan}^{\beta} \psi
\mathrm{D}_{\rm tan}^{\beta-\bm{e}_1} \mathrm{D}_{x'}   W\\
&+\int_{\p\varOmega }
\mathrm{D}_{\rm tan}^{\beta} \psi
\left( \mathring{\rm c}_0[\mathrm{D}_{\rm tan}^{\beta-\bm{e}_1}, \mathring{\rm c}_0 ]  \mathrm{D}_{x'} W
 +\mathring{\rm c}_0[\mathrm{D}_{\rm tan}^{\beta-\bm{e}_1}\mathrm{D}_{x'}, \mathring{\rm c}_0 ]   W
 \right).
\end{align*}
The first term in this decomposition can be estimated in the same way as $\mathcal{R}_{32}^a$,
and the second term in this decomposition along with
terms $\mathcal{R}_{33}$ and $\mathcal{R}_{35}$ can be controlled as  $\mathcal{R}_{32}^b$.
In conclusion, we arrive at
\begin{align}
 \sum_{i=2}^5 |\mathcal{R}_{3i}|
\label{R3.est1b} \lesssim  \left\|
\underline{\mathring{\rm c}}_0 \right\|_{H^{3}(\varOmega_t)} \VERT W(t)\VERT_s^2
+\mathcal{M}_s(t).
\end{align}

Let us deduce the estimate of term $\mathcal{R}_{31}$.
In view of \eqref{trace.est2},  we infer
\begin{align}
\notag |\mathcal{R}_{31}|
\leq\,&  2  [\widebar{F}_{11}]  \varrho(\widebar{\bm{F}}^+)
 \|\mathrm{D}_{\rm tan}^{\beta-\bm{e}_1}  W_2^+\|_{H^1(\varOmega)}
 \sum_{j=2}^d  \|\mathrm{D}_{\rm tan}^{\beta-\bm{e}_1}   v_j^+\|_{H^1(\varOmega)} \\
\leq\,& \!
\left\{
\begin{aligned}
&[\widebar{F}_{11}]  \varrho(\widebar{\bm{F}}^+) \|\mathrm{D}_{\rm tan}^{\beta-\bm{e}_1}   (W_2^+,\,W_3^+)\|_{H^1(\varOmega)}^2\ \  &&\textrm{if } d=2,\\[1.5mm]
&\sqrt{2}  [\widebar{F}_{11}]  \varrho(\widebar{\bm{F}}^+) \|\mathrm{D}_{\rm tan}^{\beta-\bm{e}_1}   (W_2^+,\,W_3^+,\,W_4^+)\|_{H^1(\varOmega)}^2\ \  &&\textrm{if } d=3.
\end{aligned}
\right.
\label{R3a.est1}
\end{align}
We now make the estimate for the term on the right-hand side of \eqref{R3a.est1}.
Since $|\beta|\leq s$,
we apply inequality \eqref{MTT.inequ} to obtain
\begin{align}
\sum_{j=2}^{d+1} \|\mathrm{D}_{\rm tan}^{\beta-\bm{e}_1}   W_j^+\|_{L^2(\varOmega)}^2
\lesssim \|W\|_{H^s(\varOmega_t)}^2.
\label{R3a.est1a}
\end{align}
According to definition \eqref{E.tan.def} for the instant tangential energy
$\mathcal{E}_{\rm tan}^{\beta}(t)$, we have
\begin{align}
 \sum_{\ell=2}^d  \sum_{j=2}^{d+1}
\|\p_{\ell} \mathrm{D}_{\rm tan}^{\beta-\bm{e}_1}   W_j^+\|_{L^2(\varOmega)}^2
 \lesssim
\left\{
\begin{aligned}
&\mathcal{E}_{\rm tan}^{\beta-\bm{e}_1+\bm{e}_2}(t)  &  &\textrm{if } d=2,\\[1mm]
&\mathcal{E}_{\rm tan}^{\beta-\bm{e}_1+\bm{e}_2}(t)+\mathcal{E}_{\rm tan}^{\beta-\bm{e}_1+\bm{e}_3}(t) &  &\textrm{if }  d=3.
\end{aligned}
\right.
\label{R3a.est1b}
\end{align}
As for the normal derivatives in \eqref{R3a.est1},
we utilize \eqref{Wnc.id}  to derive
\begin{align}
\notag
\begin{pmatrix}
0\\
\p_1 W_{\rm nc}^{\pm}\\ 0
\end{pmatrix}=\,&
\mp\bm{B} \mathcal{A}_0(\widebar{U}^{\pm},\,\widebar{\varPhi}^{\pm}) \partial_t W^{\pm}
\mp\sum_{j=2}^d \bm{B} \mathcal{A}_j(\widebar{U}^{\pm},\,\widebar{\varPhi}^{\pm})  \partial_j W^{\pm}
+\underline{\mathring{\rm c}}_1 \mathrm{D}_{\rm tan} W\\
&
-\mathring{\bm{B}}^{\pm}\mathring{\mathcal{A}}_{1b}^{\pm}\partial_1 W^{\pm}
-\mathring{\bm{B}}^{\pm}\mathring{\mathcal{A}}_4^{\pm}W^{\pm}
+\mathring{\bm{B}}^{\pm}\mathring{J}_{\pm}^{\mathsf{T}}\tilde{f}^{\pm},
\label{normal.id1}
\end{align}
where $\bm{B}(U,\varPhi)$ is defined by \eqref{B.bm}.
By virtue of identities \eqref{A2.cal}--\eqref{A3.cal}, we can compute the following assertions:
\begin{itemize}
 \item For $d=2$, the second and third components of
 $$\bm{B} \mathcal{A}_0(\widebar{U}^{+},\,\widebar{\varPhi}^{+}) \partial_t W^+
 +\bm{B} \mathcal{A}_2(\widebar{U}^{+},\,\widebar{\varPhi}^{+})  \partial_2 W^+$$
 are
$ \frac{1}{\bar{\rho}^+(\widebar{F}_{11}^+)^2 }\p_t(W_4^+-W_1^+)$  and
 $-\frac{1}{\widebar{F}_{11}^+ }\p_t W_5^+$, respectively.
\item For $d=3$, the second, third, and fourth components of
$$\bm{B} \mathcal{A}_0(\widebar{U}^{+},\,\widebar{\varPhi}^{+}) \partial_t W^+
+\bm{B} \mathcal{A}_2(\widebar{U}^{+},\,\widebar{\varPhi}^{+})  \partial_2 W^+
+\bm{B} \mathcal{A}_3(\widebar{U}^{+},\,\widebar{\varPhi}^{+})  \partial_3 W^+$$
are
$\frac{1}{\bar{\rho}^+(\widebar{F}_{11}^+)^2 }\p_t(W_5^+-W_1^+),$
$-\frac{1}{\widebar{F}_{11}^+ }\p_t W_6^+,$  and
$-\frac{1}{\widebar{F}_{11}^+ }\p_t W_7^+$, respectively.
\end{itemize}
Using \eqref{E.tan.def}, \eqref{normal.id1}, and the assertions above, we conclude
\begin{align}
 \sum_{j=2}^{d+1}
\|\p_{1} \mathrm{D}_{\rm tan}^{\beta-\bm{e}_1}   W_j^+\|_{L^2(\varOmega)}^2
\notag \leq \,&
   \big\|\mathrm{D}_{\rm tan}^{\beta-\bm{e}_1}\big(
-\mathring{\bm{B}}\mathring{\mathcal{A}}_{1b}\partial_1 W+\mathring{\bm{B}} \mathring{\mathcal{A}}_4 W
+\mathring{\bm{B}} \mathring{J}^{\mathsf{T}} \tilde{f}\big)\big\|_{L^2(\varOmega)}^2\\[-1.5mm]
&+\|\mathrm{D}_{\rm tan}^{\beta-\bm{e}_1}(\underline{\mathring{\rm c}}_1 \mathrm{D}_{\rm tan} W)\|_{L^2(\varOmega)}^2
+\frac{\mathcal{E}_{\rm tan}^{\beta}(t)}{\bar{\rho}^+(\widebar{F}_{11}^+)^2}.
\label{R3a.est1c}
\end{align}
Employ \eqref{MTT.inequ} and the Moser-type calculus inequality \eqref{Moser3} to derive
\begin{align}
\label{R3a.est1d} \|\mathrm{D}_{\rm tan}^{\beta-\bm{e}_1}(\underline{\mathring{\rm c}}_1 \mathrm{D}_{\rm tan} W)(t)\|_{L^2(\varOmega)}^2
\lesssim \|\underline{\mathring{\rm c}}_1  \|_{L^{\infty}(\varOmega_T)} \VERT W(t)\VERT_{s}^2
+ \mathcal{M}_s(t) .
\end{align}
Plug \eqref{Wnc.est1b}--\eqref{Wnc.est1e}, \eqref{Wnc.est1c}--\eqref{sigma.est},
 and \eqref{R3a.est1d} into \eqref{R3a.est1c},
insert the resulting estimate and \eqref{R3a.est1a}--\eqref{R3a.est1b} into \eqref{R3a.est1},
and use \eqref{tan.est2}, \eqref{Q.est1}, \eqref{R2.est}, \eqref{R3.est1}, and
\eqref{R3.est1b} to obtain
\begin{align}
\notag \mathcal{E}_{\rm tan}^{\beta}(t)\leq
\;&
 C\mathcal{M}_s(t)+
C\|\underline{\mathring{\rm c}}_1 \|_{H^{3}(\varOmega_T)} \VERT W(t)\VERT_{s}^2\\[1mm]
&
+\left\{
\begin{aligned}
&\frac{[\widebar{F}_{11}]}{\widebar{F}_{11}^+} \mathcal{E}_{\rm tan}^{\beta}(t)
+C\mathcal{E}_{\rm tan}^{\beta-\bm{e}_1+\bm{e}_2}(t)
&\ & \textrm{if } d=2,\\
&\sqrt{2}\,\frac{[\widebar{F}_{11}]}{\widebar{F}_{11}^+} \mathcal{E}_{\rm tan}^{\beta}(t)
+C\mathcal{E}_{\rm tan}^{\beta-\bm{e}_1+\bm{e}_2}(t)
+C\mathcal{E}_{\rm tan}^{\beta-\bm{e}_1+\bm{e}_3}(t)&\  & \textrm{if } d=3.
\end{aligned}\right.
\notag
\end{align}
Since $\widebar{F}_{11}^+>\widebar{F}_{11}^->0$, we always see that
$[\widebar{F}_{11}]/\widebar{F}_{11}^+<1$.
Moreover, it follows from \eqref{thm.H1} that $[\widebar{F}_{11}]/\widebar{F}_{11}^+<\frac{1}{2}$
for dimension $d=3$.
Thus, we can obtain
\begin{align}
\notag \mathcal{E}_{\rm tan}^{\beta}(t)\lesssim
\;& \mathcal{M}_s(t)+
\|\underline{\mathring{\rm c}}_1 \|_{H^{3}(\varOmega_T)}\VERT W(t)\VERT_{s}^2\\[1mm]
& +\left\{
\begin{aligned}
& \mathcal{E}_{\rm tan}^{\beta-\bm{e}_1+\bm{e}_2}(t)
&& \textrm{if } d=2,\\[1mm]
& \mathcal{E}_{\rm tan}^{\beta-\bm{e}_1+\bm{e}_2}(t)
+  \mathcal{E}_{\rm tan}^{\beta-\bm{e}_1+\bm{e}_3}(t)&\qquad\quad  & \textrm{if } d=3,
\end{aligned}\right.
\label{E.tan.est3}
\end{align}
for all $\beta=(\beta_0,\beta_2,\ldots,\beta_d)\in\mathbb{N}^d$ with $|\beta|\leq s$ and $\beta_0\geq 1$.
Inequality \eqref{E.tan.est3} reduces the estimate
of each instant tangential energy
 to that with one less time derivative.
Therefore, we are led to estimate $\mathcal{R}_3$ for the case containing at least one space derivative.

\subsection{Estimate of Term $\mathcal{R}_3$ with the $x_2$-Derivative}\label{sec.tan4.b}

In this subsection, we make the estimate of $\mathcal{R}_3$ defined in \eqref{Q.est1}
 for the case  when
$\beta_2\geq 1$ and $|\beta|\leq s$.

Computing from \eqref{varrho} that
\begin{align}
\notag
&\p_{F_{ij}} \varrho (\widebar{\bm{F}}^+ )\mathrm{D}_{\rm tan}^{\beta} F_{ij}^+\\
\label{varrho.id2}&\quad=
\left\{
\begin{aligned}
&-\varrho(\widebar{\bm{F}}^+ )^2  \mathrm{D}_{\rm tan}^{\beta} F_{22}^+       &\ & \textrm{if } d=2, \\[1mm]
&-\varrho(\widebar{\bm{F}}^+ )^2
\left( \widebar{F}_{33} \mathrm{D}_{\rm tan}^{\beta} F_{22}^+
+\widebar{F}_{22} \mathrm{D}_{\rm tan}^{\beta} F_{33}^+
 \right)     &\ & \textrm{if } d=3,
\end{aligned}
\right.
\end{align}
and using \eqref{psi.d2.id.2D}--\eqref{psi.d2.id}, we deduce
\begin{align}
\mathcal{R}_3=
\underbrace{2 \int_{\p\varOmega } [\widebar{F}_{11}]  \p_{F_{ij}}\varrho(\widebar{\bm{F}}^+)
\mathrm{D}_{\rm tan}^{\beta}  F_{ij}^+\mathrm{D}_{\rm tan}^{\beta}   \psi
}_{\widetilde{\mathcal{R}}_{31} +\widetilde{\mathcal{R}}_{32} }
+\underbrace{\int_{\p\varOmega }\underline{\mathring{\rm c}}_0
\mathrm{D}_{\rm tan}^{\beta} F_{ij}^{+} \mathrm{D}_{\rm tan}^{\beta} \psi
}_{\widetilde{\mathcal{R}}_{33} }  ,
\label{R3.est2}
\end{align}
where
\begin{align}
\notag  \widetilde{\mathcal{R}}_{31}&:=
\left\{
\begin{aligned}
&{    -2  [\widebar{F}_{11}]   \varrho(\widebar{\bm{F}}^+)^3  \int_{\p\varOmega }\mathrm{D}_{\rm tan}^{\beta-\bm{e}_2} F_{12}^+ \mathrm{D}_{\rm tan}^{\beta}  F_{22}^+}
\  && \textrm{if } d=2, \\
&-2  [\widebar{F}_{11}]   \varrho(\widebar{\bm{F}}^+)^2   \int_{\p\varOmega }\mathrm{D}_{\rm tan}^{\beta-\bm{e}_2} F_{12}^+
\Big(\mathrm{D}_{\rm tan}^{\beta}  F_{33}^+ + \frac{\widebar{F}_{33}}{\widebar{F}_{22}}\mathrm{D}_{\rm tan}^{\beta}  F_{22}^+\Big)
\  && \textrm{if } d=3,
\end{aligned}
\right.\\
\notag \widetilde{\mathcal{R}}_{32}&:=
 \int_{\p\varOmega }
\mathring{\rm c}_0  \mathrm{D}_{\rm tan}^{\beta-\bm{e}_2}
 \Big( \underline{\mathring{\rm c}}_1 W +\sum_{\ell=2}^d\mathring{\rm c}_0 R_{\ell}\Big)
 \mathrm{D}_{\rm tan}^{\beta}  F_{ij}^+   .
\end{align}
Similar to the derivation of estimates \eqref{R32a.e1}--\eqref{R32a.e0}, we can obtain
\begin{align}
  | \widetilde{\mathcal{R}}_{32}|+| \widetilde{\mathcal{R}}_{33}| \lesssim
  \big\|\underline{\mathring{\rm c}}_1\big\|_{H^{3}(\varOmega_t)}
\VERT  W\VERT_{s}^2+ \mathcal{M}_s(t).
\label{R3.est2b}
\end{align}
Utilizing inequality \eqref{trace.est2} leads to
\begin{align}
\notag |\widetilde{\mathcal{R}}_{31}|
\leq \;&2  [\widebar{F}_{11}]   \varrho(\widebar{\bm{F}}^+)^3
\| \mathrm{D}_{\rm tan}^{\beta-\bm{e}_2} F_{12}^+\|_{H^{1}(\varOmega) }
\| \mathrm{D}_{\rm tan}^{\beta-\bm{e}_2} F_{22}^+\|_{H^{1}(\varOmega) }\\
\leq \;&  [\widebar{F}_{11}]   \varrho(\widebar{\bm{F}}^+)^3
\| \mathrm{D}_{\rm tan}^{\beta-\bm{e}_2} (F_{12}^+,F_{22}^+)\|_{H^{1}(\varOmega) }^2
\qquad\qquad  \textrm{if } d=2.
\label{R3.est2c.2D}
\end{align}
Moreover, for $d=3$, we have
\begin{align}
\notag |\widetilde{\mathcal{R}}_{31}|
\leq \;&2  [\widebar{F}_{11}]   \varrho(\widebar{\bm{F}}^+)^2
\| \mathrm{D}_{\rm tan}^{\beta-\bm{e}_2} F_{12}^+\|_{H^{1}(\varOmega) }\\[1mm]
\notag &\times \Big( \| \mathrm{D}_{\rm tan}^{\beta-\bm{e}_2} F_{33}^+\|_{H^{1}(\varOmega) }
+\frac{\widebar{F}_{33}}{\widebar{F}_{22}} \| \mathrm{D}_{\rm tan}^{\beta-\bm{e}_2} F_{22}^+\|_{H^{1}(\varOmega) }  \Big) \\
\leq \;& [\widebar{F}_{11}]   \varrho(\widebar{\bm{F}}^+)^2
\Big(1+\frac{\widebar{F}_{33}^2}{\widebar{F}_{22}^2}\Big)^{1/2}\,
\| \mathrm{D}_{\rm tan}^{\beta-\bm{e}_2} (F_{12}^+,F_{22}^+,F_{33}^+)\|_{H^{1}(\varOmega) }^2.
\label{R3.est2c}
\end{align}

To estimate the terms on the right-hand side of \eqref{R3.est2c.2D}--\eqref{R3.est2c},
we compute from \eqref{eta.id0}--\eqref{zeta.id0} that
\begin{align}
\eta_i^{\pm}&=
\pm \widebar{{F}}_{11}^{\pm} \p_1 F_{i2}^{\pm}-\widebar{F}_{22}\p_2F_{i1}^{\pm} +\underline{\mathring{\rm c}}_1\mathrm{D}_xW
+\mathring{\rm c}_2W,
\label{eta.id2}\\[1mm]
 \zeta_i^{\pm}&=
\pm \widebar{{F}}_{11}^{\pm} \p_1 F_{i3}^{\pm}-\widebar{F}_{33}\p_3F_{i1}^{\pm} +\underline{\mathring{\rm c}}_1\mathrm{D}_xW
+\mathring{\rm c}_2W.
\label{zeta.id2}
\end{align}
By virtue of identities \eqref{eta.id2}--\eqref{zeta.id2}, estimates \eqref{eta.est}--\eqref{eta.est'}, and
\begin{align}
F_{11}^+=\frac{1}{\bar{\rho}^{+} \widebar{F}_{11}^+ }(W_1^+-W_{d+2}^+)+ \underline{\mathring{\rm c}}_1 W,
\label{F11.id}
\end{align}
 we obtain the following two assertions:
\begin{itemize}
\item If $d=2$, then
\begin{align*}
&\|\mathrm{D}_{\rm tan}^{\beta-\bm{e}_2} \mathrm{D}_x (F_{12}^+,\,F_{22}^+)\|_{L^2(\varOmega)}^2\\[1.5mm]
& \quad \leq \|\mathrm{D}_{\rm tan}^{\beta}  (F_{12}^+,\,F_{22}^+)\|_{L^2(\varOmega)}^2
+\frac{\widebar{F}_{22}^2}{(\bar{\rho}^{+})^2(\widebar{F}_{11}^+)^4}
\|\mathrm{D}_{\rm tan}^{\beta}  (W_1^+-W_4^+)\|_{L^2(\varOmega)}^2\\[0.5mm]
& \quad \quad\; +\frac{\widebar{F}_{22} ^2}{ (\widebar{F}_{11}^+)^2}
\|\mathrm{D}_{\rm tan}^{\beta}  F_{21}^+\|_{L^2(\varOmega)}^2
+C   \|\underline{\mathring{\rm c}}_1  \|_{L^{\infty}(\varOmega_T)} \VERT W\VERT_{s}^2
+ C \mathcal{M}_s(t) ,
\end{align*}
which, combined with \eqref{R3.est2c.2D}, leads to
\begin{align}
|\widetilde{\mathcal{R}}_{31}|\leq
\widebar{C}_0 \mathcal{E}_{\rm tan}^{\beta}(t)
+C   \|\underline{\mathring{\rm c}}_1  \|_{L^{\infty}(\varOmega_T)} \VERT W\VERT_{s}^2
+C  \mathcal{M}_s(t) ,
\label{R3.est2d.2D}
\end{align}
where
\begin{align}
\label{C0.def}&\widebar{C}_0:=
\begin{aligned}
\max( 1,\, \frac{ (\widebar{F}_{11}^+)^2}{\widebar{F}_{22}^2} )\frac{ [\widebar{F}_{11}]}{\widebar{F}_{11}^+}
\end{aligned}.
\end{align}
\item If $d=3$, then
\begin{align*}
&\|\mathrm{D}_{\rm tan}^{\beta-\bm{e}_2} \mathrm{D}_x (F_{12}^+,\,F_{22}^+,\,F_{33}^+)\|_{L^2(\varOmega)}^2\\[2mm]
& \leq \|\mathrm{D}_{\rm tan}^{\beta}  (F_{12}^+,\,F_{22}^+,\,F_{33}^+)\|_{L^2(\varOmega)}^2
+\|\mathrm{D}_{\rm tan}^{\beta-\bm{e}_2+\bm{e}_3}  (F_{12}^+,\,F_{22}^+,\,F_{33}^+)\|_{L^2(\varOmega)}^2
\\[1.5mm]
&\quad\; +\frac{\widebar{F}_{22}^2}{(\bar{\rho}^{+})^2(\widebar{F}_{11}^+)^4}
\|\mathrm{D}_{\rm tan}^{\beta}  (W_1^+-W_5^+)\|_{L^2(\varOmega)}^2
+\frac{\widebar{F}_{33}^2}{ (\widebar{F}_{11}^+)^2}
\|\mathrm{D}_{\rm tan}^{\beta-\bm{e}_2+\bm{e}_3}  F_{31}^+\|_{L^2(\varOmega)}^2\\[1mm]
&\quad\; +\frac{\widebar{F}_{22} ^2}{ (\widebar{F}_{11}^+)^2}
\|\mathrm{D}_{\rm tan}^{\beta}  F_{21}^+\|_{L^2(\varOmega)}^2
+C   \|\underline{\mathring{\rm c}}_1  \|_{L^{\infty}(\varOmega_T)} \VERT W\VERT_{s}^2
+ C \mathcal{M}_s(t) ,
\end{align*}
which, along with \eqref{R3.est2c}, yields
\begin{align}
\notag |\widetilde{\mathcal{R}}_{31}|\leq
\;& \widebar{C}_1 \mathcal{E}_{\rm tan}^{\beta}(t)
+\widebar{C}_2 \mathcal{E}_{\rm tan}^{\beta-\bm{e}_2+\bm{e}_3}(t) \\
&+C  \|\underline{\mathring{\rm c}}_1  \|_{L^{\infty}(\varOmega_T)} \VERT W\VERT_{s}^2
+ C \mathcal{M}_s(t) ,
\label{R3.est2d}
\end{align}
where
\begin{align}
\widebar{C}_1:= \,&
{ 
\begin{aligned}
\Big(1+\frac{\widebar{F}_{33}^2}{\widebar{F}_{22}^2}\Big)^{1/2}  
\max(1,\, \frac{\widebar{F}_{22}^2}{ (\widebar{F}_{11}^+)^2})
\frac{\widebar{F}_{11}^+[\widebar{F}_{11}]}{\widebar{F}_{22}\widebar{F}_{33}}
\end{aligned}
}
\label{C1.def}, \\[1mm]
\label{C2.def}
\widebar{C}_2:=\,&
{ 
\begin{aligned}
\Big( 1+\frac{\widebar{F}_{33}^2}{\widebar{F}_{22}^2}\Big)^{1/2}
\max( 1,\, \frac{\widebar{F}_{33}^2}{ (\widebar{F}_{11}^+)^2 })
\frac{ \widebar{F}_{11}^+[\widebar{F}_{11}]}{\widebar{F}_{22}\widebar{F}_{33}}
\end{aligned}
}.
\end{align}
\end{itemize}
Plugging estimates \eqref{R3.est2b}, \eqref{R3.est2d.2D}, and \eqref{R3.est2d} into \eqref{R3.est2},
and using \eqref{tan.est2}, \eqref{Q.est1}, and  \eqref{R2.est}, we deduce
\begin{align}
\notag \mathcal{E}_{\rm tan}^{\beta}(t)\leq
\,&
 C
\|\underline{\mathring{\rm c}}_1  \|_{H^{3}(\varOmega_T)}   \VERT W(t)\VERT_{s}^2
+C \mathcal{M}_s(t)  \\[2mm]
&+\left\{
\begin{aligned}
& \widebar{C}_0\, \mathcal{E}_{\rm tan}^{\beta}(t)  &\qquad\quad &\textrm{if } d=2,\\
&  \widebar{C}_1\, \mathcal{E}_{\rm tan}^{\beta}(t)+ \widebar{C}_2\, \mathcal{E}_{\rm tan}^{\beta-\bm{e}_2+\bm{e}_3}(t)  & &\textrm{if } d=3.
\end{aligned}
\right.
\label{E.tan.est4b}
\end{align}

For $d=3$,
it follows from \eqref{thm.H1} that $\widebar{C}_1<1$,
so that estimate \eqref{E.tan.est4b} implies
\begin{align}
  \mathcal{E}_{\rm tan}^{\beta}(t)\leq
C
\|\underline{\mathring{\rm c}}_1  \|_{H^{3}(\varOmega_T)}    \VERT W(t)\VERT_{s}^2
 +C \mathcal{M}_s(t) +\frac{\widebar{C}_2}{1-\widebar{C}_1} \mathcal{E}_{\rm tan}^{\beta-\bm{e}_2+\bm{e}_3}(t)
\label{E.tan.est4}
\end{align}
for all $\beta\in\mathbb{N}^3$ with $|\beta|\leq s$ and $\beta_2\geq 1$.

\vspace*{5mm}
\noindent {\bf Proof of Proposition \ref{lem.tan} for $d=2$}.\quad
In the two-dimensional case,
if  \eqref{thm.H1} holds, then $\widebar{C}_0<1$.
From \eqref{E.tan.est4b},  we have
\begin{align}
\mathcal{E}_{\rm tan}^{\beta}(t)\lesssim
\|\underline{\mathring{\rm c}}_1  \|_{H^{3}(\varOmega_T)}  \VERT W(t)\VERT_{s}^2
+\mathcal{M}_s(t),
\label{E.tan.est4.2D}
\end{align}
for  all $\beta\in\mathbb{N}^2$ with $|\beta|\leq s$ and $\beta_2\geq 1$.
Combining  \eqref{E.tan.est4.2D} and \eqref{E.tan.est3},
we can conclude \eqref{E.tan.est4.2D}  for all $\beta \in\mathbb{N}^2$ with $|\beta|\leq s$.
The proof
for case $d=2$ is complete.
\qed

\subsection{Estimate of Term $\mathcal{R}_3$ with the $x_3$-Derivative
}\label{sec.tan4.c}
For the three-dimensional case ($d=3$), 
in order to prove \eqref{tan.est}, it suffices to obtain the estimate of $\mathcal{R}_3$ defined in \eqref{Q.est1}  for $\beta_3\geq 1$ and $|\beta|\leq s$.
For this purpose,
we utilize \eqref{psi.d3.id} and \eqref{varrho.id2}  to  deduce
\begin{align}
\mathcal{R}_3=
\underbrace{2 \int_{\p\varOmega } [\widebar{F}_{11}]  \p_{F_{ij}}\varrho(\widebar{\bm{F}}^+)
	\mathrm{D}_{\rm tan}^{\beta}  F_{ij}^+\mathrm{D}_{\rm tan}^{\beta}   \psi
}_{\widehat{\mathcal{R}}_{31} +\widehat{\mathcal{R}}_{32} }
+\underbrace{\int_{\p\varOmega }\underline{\mathring{\rm c}}_0
	\mathrm{D}_{\rm tan}^{\beta} F_{ij}^{+} \mathrm{D}_{\rm tan}^{\beta} \psi
}_{\widehat{\mathcal{R}}_{33} }  ,
\label{R3.est3}
\end{align}
where
\begin{align}
\notag  \widehat{\mathcal{R}}_{31}&:=
-2  [\widebar{F}_{11}]   \varrho(\widebar{\bm{F}}^+)^2   \int_{\p\varOmega }\mathrm{D}_{\rm tan}^{\beta-\bm{e}_3} F_{13}^+
\Big(\frac{\widebar{F}_{22}}{\widebar{F}_{33}} \mathrm{D}_{\rm tan}^{\beta}  F_{33}^+ + \mathrm{D}_{\rm tan}^{\beta}  F_{22}^+\Big),\\
\notag \widehat{\mathcal{R}}_{32}&:=
 \int_{\p\varOmega }
\mathring{\rm c}_0  \mathrm{D}_{\rm tan}^{\beta-\bm{e}_3}
\Big( \underline{\mathring{\rm c}}_1 W +\sum_{\ell=2}^d\mathring{\rm c}_0 R_{\ell}\Big)
\mathrm{D}_{\rm tan}^{\beta}  F_{ij}^+   .
\end{align}
Similar to the derivation of estimates \eqref{R32a.e1}--\eqref{R32a.e0}, we can deduce
\begin{align}
| \widehat{\mathcal{R}}_{32}|+| \widehat{\mathcal{R}}_{33}|
\lesssim
  \big\|\underline{\mathring{\rm c}}_1\big\|_{H^{3}(\varOmega_t)}
\VERT W(t) \VERT_{s}^2+\mathcal{M}_s(t).
\label{R3.est3b}
\end{align}

In view of inequality \eqref{trace.est2}, we have
\begin{align}
\notag |\widehat{\mathcal{R}}_{31}|
\leq \;&2  [\widebar{F}_{11}]   \varrho(\widebar{\bm{F}}^+)^2
\| \mathrm{D}_{\rm tan}^{\beta-\bm{e}_3} F_{13}^+\|_{H^{1}(\varOmega) }\\
\notag &\times \Big( \frac{\widebar{F}_{22}}{\widebar{F}_{33}} \| \mathrm{D}_{\rm tan}^{\beta-\bm{e}_3} F_{33}^+\|_{H^{1}(\varOmega) }
+\| \mathrm{D}_{\rm tan}^{\beta-\bm{e}_3} F_{22}^+\|_{H^{1}(\varOmega) }  \Big)\\
\leq \;&[\widebar{F}_{11}]   \varrho(\widebar{\bm{F}}^+)^2
\Big( 1+\frac{\widebar{F}_{22}^2}{\widebar{F}_{33}^2}\Big)^{1/2}
\| \mathrm{D}_{\rm tan}^{\beta-\bm{e}_3} (F_{13}^+,F_{22}^+,F_{33}^+)\|_{H^{1}(\varOmega) }^2.
\label{R3.est3c}
\end{align}
Use identities \eqref{eta.id2}--\eqref{F11.id} and estimates \eqref{eta.est}--\eqref{eta.est'} to derive
\begin{align*}
&\|\mathrm{D}_{\rm tan}^{\beta-\bm{e}_3}  \mathrm{D}_x (F_{13}^+,\,F_{22}^+,\,F_{33}^+)\|_{L^2(\varOmega)}^2\\[3mm]
&\quad  \leq
\|\mathrm{D}_{\rm tan}^{\beta}   (F_{13}^+,\,F_{22}^+,\,F_{33}^+)\|_{L^2(\varOmega)}^2
+\|\mathrm{D}_{\rm tan}^{\beta-\bm{e}_3+\bm{e}_2}   (F_{13}^+,\,F_{22}^+,\,F_{33}^+)\|_{L^2(\varOmega)}^2\\[1.5mm]
&\quad\quad\,\, +\frac{\widebar{F}_{33}^2}{(\bar{\rho}^{+})^2(\widebar{F}_{11}^+)^4}
\|\mathrm{D}_{\rm tan}^{\beta}  (W_1^+-W_5^+)\|_{L^2(\varOmega)}^2
+\frac{\widebar{F}_{22} ^2}{ (\widebar{F}_{11}^+)^2}
\|\mathrm{D}_{\rm tan}^{\beta-\bm{e}_3+\bm{e}_2}  F_{21}^+\|_{L^2(\varOmega)}^2
\\[1mm]
&\quad\quad\,\, +\frac{\widebar{F}_{33}^2}{ (\widebar{F}_{11}^+)^2}
\|\mathrm{D}_{\rm tan}^{\beta}  F_{31}^+\|_{L^2(\varOmega)}^2
+ C \|\underline{\mathring{\rm c}}_1  \|_{L^{\infty}(\varOmega_t)} \VERT W\VERT_{s}^2
+C  \mathcal{M}_s(t),
\end{align*}
which, along with \eqref{tan.est2}, \eqref{Q.est1}, \eqref{R2.est}, \eqref{R3.est3}, 
and \eqref{R3.est3c}--\eqref{R3.est3d}, yields
\begin{align}
\notag \mathcal{E}_{\rm tan}^{\beta}(t)\leq
\;& \widebar{C}_3\, \mathcal{E}_{\rm tan}^{\beta}(t)
+\widebar{C}_4\, \mathcal{E}_{\rm tan}^{\beta-\bm{e}_3+\bm{e}_2}(t)
\\
&+  C  \|\underline{\mathring{\rm c}}_1  \|_{L^{\infty}(\varOmega_T)} \VERT W\VERT_{s}^2
+ C \mathcal{M}_s(t) ,
\label{R3.est3d}
\end{align}
where
\begin{align}
\widebar{C}_3:=\,&
{ 	\begin{aligned}
\Big( 1+\frac{\widebar{F}_{22}^2}{\widebar{F}_{33}^2}\Big)^{1/2}
\max( 1,\, \frac{\widebar{F}_{33}^2}{ (\widebar{F}_{11}^+)^2 } )
\frac{ \widebar{F}_{11}^+[\widebar{F}_{11}]}{\widebar{F}_{22}\widebar{F}_{33}}
\end{aligned}}
\label{C3.def}
,\\[1mm]
\label{C4.def}\widebar{C}_4:=\,&
{ 
\begin{aligned}
\Big( 1+\frac{\widebar{F}_{22}^2}{\widebar{F}_{33}^2}\Big)^{1/2}
\max( 1,\, \frac{\widebar{F}_{22}^2}{ (\widebar{F}_{11}^+)^2})
\frac{ \widebar{F}_{11}^+[\widebar{F}_{11}]}{\widebar{F}_{22}\widebar{F}_{33}}
\end{aligned}}
.
\end{align}
Noting from \eqref{thm.H1} that $\widebar{C}_3<1$,
we have
\begin{align}
 \mathcal{E}_{\rm tan}^{\beta}(t)\leq
C\|\underline{\mathring{\rm c}}_1\|_{H^{3}(\varOmega_T)} \VERT W(t)\VERT_{s}^2
+C  \mathcal{M}_s(t) +\frac{\widebar{C}_4}{1-\widebar{C}_3} \mathcal{E}_{\rm tan}^{\beta-\bm{e}_3+\bm{e}_2}(t),
\label{E.tan.est5}
\end{align}
for all  $\beta\in\mathbb{N}^3$ with $\beta_3\geq 1$ and $|\beta|\leq s$.

\vspace*{5mm}
\noindent {\bf Proof of Proposition \ref{lem.tan} for $d=3$}.\quad
Combine \eqref{E.tan.est4}  and \eqref{E.tan.est5} to infer
\begin{align}
\notag \mathcal{E}_{\rm tan}^{\beta}(t)\leq
\,&
C
 \|\underline{\mathring{\rm c}}_1\|_{H^{ 3}(\varOmega_T)} \VERT W(t)\VERT_{s}^2 +C \mathcal{M}_s(t) +\frac{\widebar{C}_2 \widebar{C}_4}{(1-\widebar{C}_1)(1-\widebar{C}_3)}\, \mathcal{E}_{\rm tan}^{\beta}(t),
\notag
\end{align}
which yields
\begin{align}
 \mathcal{E}_{\rm tan}^{\beta}(t)\lesssim
\|\underline{\mathring{\rm c}}_1\|_{H^{ 3}(\varOmega_T)}\VERT W(t)\VERT_{s}^2
+\mathcal{M}_s(t)
\label{E.tan.est6}
\end{align}
for all $\beta\in\mathbb{N}^3$ with $\beta_3\geq 1$ and $|\beta|\leq s$,
provided
\begin{align} \notag
\widebar{C}_2 \widebar{C}_4<(1-\widebar{C}_1)(1-\widebar{C}_3).
\end{align}
This last condition is equivalent to \eqref{thm.H1} because of $\widebar{C}_1\widebar{C}_3=\widebar{C}_2\widebar{C}_4$.
Combining \eqref{E.tan.est3}, \eqref{E.tan.est4},  and \eqref{E.tan.est6}, 
we  deduce  \eqref{E.tan.est6} for all $\beta\in\mathbb{N}^3$ with $|\beta|\leq s$.
Therefore, we complete the proof for $d=3$.
\qed


\section{Proof of Theorem \ref{thm2}}\label{sec.Proof}

This subsection is dedicated to the proof of the main theorem of this paper, Theorem \ref{thm2}.

Combine estimates \eqref{normal.est}--\eqref{normal.est'} and \eqref{tan.est} to obtain
\begin{align}\notag
\VERT W(t)\VERT_s^2
\lesssim
\|\underline{\mathring{\rm c}}_1\|_{H^{3}(\varOmega_T)}
\VERT W(t)\VERT_{s}^2 + \mathcal{M}_s(t),
\end{align}
where $\mathcal{M}_s(t)$ is defined by \eqref{M.cal}.
Thanks to \eqref{bas.bound},
we apply the Moser-type calculus inequality \eqref{Moser2}
and take ${K}>0$ sufficiently small to obtain
\begin{align} \label{W.est1}
\VERT W(t)\VERT_s^2
\lesssim   \mathcal{M}_s(t).
\end{align}

It follows from definitions \eqref{varPsi.def}--\eqref{VERT.norm} that
\begin{align} \notag
\VERT \varPsi(t)\VERT_s^2
=\sum_{k+|\beta|\leq s}\int_0^{\infty} |\p_1^k\chi(\pm x_1)|^2\mathrm{d}x_1  \int_{\mathbb{R}^{d-1} } | \mathrm{D}_{\rm tan}^{\beta} \psi(t,x')|^2\mathrm{d} x',
\end{align}
which, along with  \eqref{chi},  leads to
\begin{align} \label{varPsi.est0}
\VERT \varPsi(t)\VERT_s^2\sim \sum_{|\beta|\leq s }\|\mathrm{D}_{\rm tan}^{\beta} \psi(t)\|_{L^2(\p\varOmega)}^2.
\end{align}
Integrate \eqref{varPsi.est0} over $(-\infty,T)$ to obtain
\begin{align} \label{varPsi.est}
\|\varPsi\|_{H^{s}(\varOmega_T)} \sim \|\psi\|_{H^{s}(\omega_T)}.
\end{align}
Similarly, we see from \eqref{bas.relation} that
\begin{align}  \label{tame.est2b}
\|\mathring{\varPsi}\|_{H^{s}(\varOmega_T)} \sim \|\mathring{\varphi}\|_{H^{s}(\omega_T)}.
\end{align}
In view of \eqref{MTT.inequ}, \eqref{Rj.est}, \eqref{key4'}, and \eqref{W.est1},
we have
\begin{align}
\notag \sum_{|\beta|\leq s }\|\mathrm{D}_{\rm tan}^{\beta} \psi(t)\|_{L^2(\p\varOmega)}^2
&\lesssim  \| \psi\|_{H^s(\omega_t )}^2
+\sum_{|\beta|=s-1}\Big\| \mathrm{D}_{\rm tan}^{\beta} \Big(\mathring{\rm c}_1 W
+\sum_{j=2}^d\underline{\mathring{\rm c}}_0 R_j \Big)\Big\|_{L^2(\p\varOmega)}^2\\
\label{psi.est}
& \lesssim   \VERT W(t)\VERT_{s}^2
+ \mathcal{M}_s(t)\lesssim \mathcal{M}_s(t),
\end{align}
which, along with  \eqref{varPsi.est0}, yields
\begin{align}
\VERT (W,\varPsi)(t)\VERT_1^2
\lesssim \;&    \int_0^t \VERT (W,\varPsi)(\tau)\VERT_1^2\,\mathrm{d}\tau
+ \| \tilde{f}\|_{H^{1}({\varOmega_t})}^2 ,\notag \\
\notag
\VERT (W,\varPsi)(t)\VERT_s^2
\lesssim  \;&    \int_0^t \VERT (W,\varPsi)(\tau)\VERT_s^2\,\mathrm{d}\tau
+ \| \tilde{f}\|_{H^{s}({\varOmega_t})}^2\\
& +
\| (\mathring{V},  \mathring{\varPsi})\|_{H^{s+2}(\varOmega_T)}^2
\|(W,\,\varPsi,\,\tilde{f}) \|^2_{H^{3}(\varOmega_t)}\quad {\textrm{for }s\geq 3.}
\notag
\end{align}
Applying Gr\"{o}nwall's inequality to the estimates above implies
\begin{align}
 \label{tame.est3'} \VERT (W,\varPsi)(t)\VERT_1^2
\lesssim  \; &  
  \| \tilde{f}\|_{H^{1}({\varOmega_t})}^2,  \\[1mm]
  \VERT (W,\varPsi)(t)\VERT_s^2
\lesssim \; &   
\| \tilde{f}\|_{H^{s}({\varOmega_t})}^2
+\| (\mathring{V},  \mathring{\varPsi})\|_{H^{s+2}(\varOmega_T)}^2
\|(W,\,\varPsi,\,\tilde{f}) \|^2_{H^{3}(\varOmega_t)}
\ \ {\textrm{for }s\geq 3.}
\label{tame.est3}
\end{align}
Since $W$ and $\psi$ vanish in the past,
we integrate \eqref{tame.est3'}--\eqref{tame.est3} over $[0,T]$  to deduce
\begin{align}
\label{tame.est3b'}\|(W,\varPsi) \|_{H^1(\varOmega_T)}^2
\lesssim  \;&   
  \| \tilde{f}\|_{H^{1}({\varOmega_T})}^2,\\[1mm]
   \|(W,\varPsi) \|_{H^s(\varOmega_T)}^2
\lesssim  \;&  
\| \tilde{f}\|_{H^{s}({\varOmega_T})}^2
  +\| (\mathring{V},  \mathring{\varPsi})\|_{H^{s+2}(\varOmega_T)}^2
\|(W, \varPsi, \tilde{f}) \|^2_{H^{3}(\varOmega_T)}
\ \  {\textrm{for }s\geq 3.}
 \label{tame.est3b}
\end{align}
Utilizing \eqref{tame.est3b} with $s=3$ and \eqref{bas.bound},
we take ${K}>0$ sufficiently small to derive
\begin{align}\label{tame.est3c}
 \|(W,\varPsi) \|_{H^3(\varOmega_T)}^2
 \lesssim \| \tilde{f}\|_{H^{3}({\varOmega_T})} ^2.
\end{align}
Insert \eqref{tame.est3c} into \eqref{tame.est3b} to find
\begin{align}
 \|(W,\varPsi) \|_{H^s(\varOmega_T)}^2 \leq  C(K_0,T)
\Big\{\| \tilde{f}\|_{H^{s}({\varOmega_T})}^2 + \| (\mathring{V},  \mathring{\varPsi})\|_{H^{s+2}(\varOmega_T)}^2
\|\tilde{f} \|^2_{H^{3}(\varOmega_T)}\Big\}. \label{tame.est3d}
\end{align}

Recalling $V^{\pm}=\mathring{J}_{\pm} W^{\pm}$ ({\it cf.}\;\eqref{W.def.d}),
we employ the Moser-type calculus inequality \eqref{Moser3}, \eqref{MTT.inequ},
and the Sobolev embedding theorem
to obtain
\begin{align}
\|V \|_{H^s(\varOmega_T)}^2
\lesssim \; &\sum_{|\alpha|\leq s }\Big( \|\mathring{J} \mathrm{D}^{\alpha}W \|_{L^2(\varOmega_T)}^2
+ \|[ \mathrm{D}^{\alpha},\,\mathring{J}]W \|_{L^2(\varOmega_T)}^2
\Big)\notag \\
\lesssim \;&    \|W \|_{H^s(\varOmega_T)}^2
+ \| (\mathring{V},  \mathring{\varPsi})\|_{H^{s+1}(\varOmega_T)}^2  \|W \|_{H^{3}(\varOmega_T)}^2 .
\label{tame.est4}
\end{align}
Combining \eqref{varPsi.est} with \eqref{tame.est3c}--\eqref{tame.est4} yields
\begin{align}
\notag &  \|V\|_{H^s(\varOmega_T)}^2+ \|\psi \|_{H^s(\omega_T)}^2 \\&
\qquad \leq  C(K_0,T)
\Big\{ \| \tilde{f}\|_{H^{s}({\varOmega_T})}^2
+\| (\mathring{V},  \mathring{\varPsi})\|_{H^{s+2}(\varOmega_T)}^2
\|\tilde{f} \|^2_{H^{3}(\varOmega_T)}
\Big\}. \label{tame.est5}
\end{align}
Thanks to \eqref{Rj.est}, \eqref{key4'},  and \eqref{tame.est5},
we can obtain
\begin{align}
\notag &  \|V\|_{H^s(\varOmega_T)}^2+ \|\psi \|_{H^{s+1/2}(\omega_T)}^2 \\
&\qquad\leq  C(K_0,T)
\Big\{ \| \tilde{f}\|_{H^{s}({\varOmega_T})}^2
+\| (\mathring{V},  \mathring{\varPsi})\|_{H^{s+2}(\varOmega_T)}^2
\|\tilde{f} \|^2_{H^{3}(\varOmega_T)}
\Big\}. \label{tame.est6}
\end{align}

It follows from \eqref{f.tilde} that
\begin{align*}
 \| \tilde{f}\|_{H^{m}({\varOmega_T})}^2
 \lesssim  \| {f}\|_{H^{m}({\varOmega_T})}^2
 + \| \mathring{\rm c}_1 \mathrm{D}V_{\natural} \|_{H^{m}({\varOmega_T})}^2
  + \| \mathring{\rm c}_1 V_{\natural} \|_{H^{m}({\varOmega_T})}^2.
\end{align*}
By virtue of \eqref{V.natural.est},
we employ the Moser-type calculus inequality \eqref{Moser3} and the Sobolev embedding theorem to obtain
\begin{align*}
 \| \tilde{f}\|_{H^{m}({\varOmega_T})}^2
 \lesssim \;&  \| f\|_{H^{m}({\varOmega_T})}^2+
 \|g\|_{H^{m+1/2}({\omega_T})}^2  +  \| (\mathring{V},  \mathring{\varPsi})\|_{H^{m+1}(\varOmega_T)}^2\|g\|_{H^{7/2}({\omega_T})}^2.
\end{align*}
Insert the estimate with $m=s$ and $m=3$ above into \eqref{tame.est6}
and use \eqref{tame.est2b} to deduce the tame estimate \eqref{thm2.est}.
Moreover, we can easily derive \eqref{thm2.est2} from \eqref{tame.est3b'}.
This completes the proof of Theorem \ref{thm2}.


\medskip
 
\begin{appendices}
\section[Proof of Proposition \ref{pro1.1}]{Proof of Proposition \ref{pro1.1} } \label{App.A}
Assume that $[{ S }]=0$ on $\varGamma(t)$. Taking the scalar product of the last identity in \eqref{BC.E.general}  with $N$  and utilizing \eqref{inv3.E} yield
\begin{align*}
&|N|^2\left(p(\rho^+,S^+)-p(\rho^-,S^+)\right)\\
&\quad =|N|^2[p]= \rho^+ F_{\ell N}^+[F_{\ell N}]
 =[\rho F_{\ell N}F_{\ell N}]=\sum_{j=1}^{d}(\rho^+ F_{\ell N}^+)^2[ {\rho}^{-1}].
\end{align*}
Then we infer from \eqref{c.def} and \eqref{F.N} that
$$
[\rho]=[p]=0,
$$
which, combined with \eqref{BC.E.general}, gives
\begin{align} \label{p1.rem1}
  F_{\ell N}^+[F_{\ell}]=0.
\end{align}
Plug \eqref{inv3.E}  into \eqref{inv4.E} to obtain
\begin{align} \label{p3.rem1}
F_{kN}^{+}[F_{ij}]-F_{jN}^+[F_{ik}]=0 \qquad\,\,\textrm{for $i,j,k=1,\ldots,d$}.
\end{align}

{For $d=2$}, from \eqref{p1.rem1}--\eqref{p3.rem1}, we have
\begin{align*}
(F_{1N}^{+})^2[F_{i2}]+(F_{2N}^{+})^2[F_{i2}]
=F_{2N}^{+} \left( F_{1N}^{+}[F_{i1}]+F_{2N}^{+}[F_{i2}] \right)=0,
\end{align*}
which, along with \eqref{F.N}, yields $[F_{i2}]=0$ for $i=1,2$. Then we utilize \eqref{p3.rem1} again  to obtain
$[\bm{F}]=0$ on $\varGamma(t)$.

{For $d=3$}, relations \eqref{p3.rem1} are equivalent to
\begin{align*}
( F_{1N}^+,\; F_{2N}^+,\;  F_{3N}^+)^{\mathsf{T}}\times
([F_{i1}] ,\;  [F_{i2}] ,\;   [F_{i3}])^{\mathsf{T}}=0\qquad\,\textrm{for $i=1,2,3$},
\end{align*}
which implies
\begin{align}\label{p2.rem1}
[F_{ij}] =\omega_i   F_{jN}^+
\end{align}
for some scalar functions $\omega_i$ and for all $i,j=1,2,3$.
We plug \eqref{p2.rem1} into \eqref{p1.rem1} and utilize \eqref{F.N} to deduce that
$\omega_i\equiv 0$ for all $i=1,\ldots,d$.
Then it follows from \eqref{p2.rem1} that $[\bm{F}]=0$ on $\varGamma(t)$.

In view of the second condition in \eqref{BC.E.general},
we find that $[U]=0$ on $\varGamma(t)$, {\it i.e.}, solution $U$ is continuous across front $\varGamma(t)$.
Therefore, there is no thermoelastic contact discontinuity for the case $[S]=0$.
This completes the proof of Proposition \ref{pro1.1}.

 \section[Proof of Proposition \ref{pro1}]{Proof of Proposition \ref{pro1}} \label{App.B}

 We omit indices $\pm$ in several places below to avoid overloaded expressions.

 \vspace*{2mm}
 \noindent\textbf{1}:\;{\it Proof of \eqref{rho.relation}}.\ \
 In the original variables,
 we see from \eqref{sym.form.3} that
 \begin{align*}
 (\p_t+v_{\ell} \p_{\ell}) \det\bm{F}&=\frac{\p \det\bm{F}}{\p F_{ij}} (\p_t+v_{\ell} \p_{\ell})  F_{ij}
=\det\bm{F} (\bm{F}^{-1})_{ji} F_{\ell j}\p_{\ell} v_i 
\\ & =\det\bm{F} \delta_{\ell, i} \p_{\ell} v_i=\det\bm{F} \p_{i} v_i,
 \end{align*}
which, combined with the first equation in \eqref{conser.laws}, yields
$$
  (\p_t+v_{\ell} \p_{\ell}) (\rho\det\bm{F})=0.
$$
After transformation \eqref{transform}, we find
$$
(\p_t+w_{\ell} \p_{\ell})(\rho\det \bm{F})=0, 
$$
where
\begin{align*}
 w_1:=\frac{1}{\p_1 \varPhi}\Big(v_1-\p_t\varPhi -\sum_{j=2}^{d}v_j \p_j\varPhi \Big),
 \qquad
 w_i:=v_i  \quad\textrm{for $i=2,\cdots,d$}.
 \end{align*}
Since $w_1|_{x_1=0}=0$ resulting from \eqref{TCD0.b},
we can obtain  identity \eqref{rho.relation} by the standard energy method.

\vspace*{2mm}
\noindent\textbf{2}:\;{\it Proof of \eqref{inv1}}.\ \
A straightforward calculation shows that solutions of \eqref{vec.form} satisfy (see, e.g., the proof of {\sc Qian--Zhang} \cite[Proposition 1]{QZ10MR2729321})
 \begin{align*}
 (\partial_t+v_{\ell} \p_{\ell})( {F}_{\ell k}\partial_{\ell}{F}_{ij}-{F}_{\ell j}\partial_{\ell}{F}_{i k})
 =\partial_{m }v_i({F}_{\ell k}\partial_{\ell}{F}_{m  j}-{F}_{\ell j}\partial_{\ell}{F}_{m  k}).
 \end{align*}
 After transformation \eqref{transform}, we have
 \begin{align*}
 (\partial_t+w_{\ell} \p_{\ell})M_{k,i,j}=\partial_{m }^{\varPhi}v_i M_{k,m, j}
 \end{align*} 
 with $M_{k,i,j}:={F}_{\ell k}\partial_{\ell}^{\varPhi}{F}_{ij}-{F}_{\ell j}\partial_{\ell}^{\varPhi}{F}_{ik}$.
 Here we recall  the differentials with respect to \eqref{transform} from definition \eqref{differential}.
 Similar to the proof of {\sc Hu--Wang} \cite[Lemma\;A.2]{HW11MR2737829}, we can use integration by parts and $w_1|_{x_1=0}=0$ to obtain \eqref{inv1}.

\vspace*{3mm}
 \noindent\textbf{3}:\;{\it Proof of \eqref{inv3} and \eqref{inv5}}.\ \
In the original variables,  system \eqref{sym.form}  gives
 \begin{align}
 \label{app.A1}&(\partial_t+v_{\ell} \p_{\ell})(\rho {F}_{ij})+\rho {F}_{ij} {\p_{\ell} v_{\ell}}-\rho {F}_{\ell j}\p_{\ell} v_i=0.
 \end{align}
After transformation \eqref{transform}, equations \eqref{app.A1}  become
 \begin{align}
 \label{app.A2}(\partial_t+w_{\ell} \p_{\ell} )(\rho {F}_{ij})+\rho {F}_{ij} \partial_{\ell }^{\varPhi} v_{\ell}-\rho {F}_{\ell j} \partial_{\ell}^{\varPhi} v_i=0.
 \end{align}
 By virtue of \eqref{TCD0.b}, we have
 \begin{align}\notag
 (\partial_t+w_{\ell} \p_{\ell})\p_i\varphi=\p_i v\cdot N \qquad\,\, \textrm{on $\p\varOmega$, for $i=2,\ldots,d$}. 
 \end{align}
 Then it follows from the restriction of \eqref{app.A2}  on $\p\varOmega$ that
 \begin{align}\label{app.A3}
 (\partial_t+w_{\ell} \p_{\ell} )(\rho {F}_{jN})+\rho {F}_{jN} \sum_{\ell=2}^d \p_{\ell} v_{\ell}=0\qquad\, \textrm{on $\p\varOmega$}.
 \end{align}
 Since $w_1|_{x_1=0}=0$ and $[v]=0$,  we can derive \eqref{inv3} and \eqref{inv5} by employing the method of characteristics.

 \vspace*{3mm}
\noindent\textbf{4}:\;{\it Proof of \eqref{inv4}}.\ \
It follows from  \eqref{app.A3} that
\begin{align}
\notag &(\partial_t+w_{\ell} \p_{\ell} )(\rho {F}_{kN} F_{ij} -\rho F_{jN} F_{ik})
-\rho {F}_{kN} (\partial_t+w_{\ell} \p_{\ell} ) F_{ij}\\
\notag &\quad +\rho {F}_{jN} (\partial_t+w_{\ell} \p_{\ell} ) F_{ik}+
\sum_{\ell=2}^d \p_{\ell} v_{\ell} (\rho {F}_{kN}F_{ij} -\rho F_{jN}F_{ik})=0 
\qquad \textrm{on $\p \varOmega$}.
\end{align}
Since
 \begin{align*}
 (\partial_t+w_{\ell} \p_{\ell} ) F_{ij}=F_{\ell j}\p_{\ell}^{\varPhi}v_i=
 \frac{\p_1 v_i}{\p_1 \varPhi} F_{jN}+\sum_{\ell=2}^d F_{\ell j}\p_{\ell} v_i,
 \end{align*}
we have
 \begin{align}
 \notag (\partial_t+w_{\ell} ^+\p_{\ell} ) [I_{k,i,j}]+\sum_{\ell=2}^d \p_{\ell} v_i^+ [ I_{j,\ell,k}]
+  \sum_{\ell=2}^d \p_{\ell} v_{\ell}^+  [I_{k,i,j}]=0 \qquad \textrm{on $\p \varOmega$},
 \end{align}
 for $I_{k,i,j}:=\rho {F}_{kN} F_{ij} -\rho F_{jN}  F_{ik}$.
 Since \eqref{inv4} holds at the initial time, {\it i.e.}, $[I_{k,i,j}]=0$ at $t=0$ for $i,j,k=1,\ldots,d$,
 we employ the standard argument of the energy method to derive that \eqref{inv4} is satisfied for all $t\in[0,T]$.

 \vspace*{3mm}
\noindent\textbf{5}:\;{\it Proof of \eqref{inv2}}.\ \
It suffices to prove \eqref{inv2.E} in the original variables. 
We note that \eqref{rho.relation.E}--\eqref{inv1.E} hold  in virtue of \eqref{rho.relation}--\eqref{inv1}
so that
\begin{align*}
\p_{\ell}(\rho F_{\ell k})&= \p_{\ell}((\det \bm{F})^{-1} F_{\ell k})\\
&=(\det \bm{F})^{-1} \p_{\ell}  F_{\ell k}-  (\det \bm{F})^{-2} F_{\ell k}\frac{\p \det \bm{F}}{\p F_{i j}} \p_{\ell} F_{i j}\\
&=(\det \bm{F})^{-1} \left( \p_{\ell}  F_{\ell k}-  (\bm{F}^{-1})_{j i} F_{\ell k}  \p_{\ell} F_{i j} \right)\\[1mm]
&=(\det \bm{F})^{-1} \left( \p_{\ell}  F_{\ell k}-   (\bm{F}^{-1})_{ j i}  F_{\ell j}    \p_{\ell} F_{i k} \right)\\[1mm]
&=(\det \bm{F})^{-1} \left( \p_{\ell}  F_{\ell k}-   \delta_{\ell,i} \p_{\ell} F_{i k} \right)=0.
\end{align*}
This completes the proof of Proposition \ref{pro1}.
\end{appendices}



\bigskip
\noindent {\bf Compliance with ethical standards}

\vspace*{2mm}
\noindent {\bf Conflict of interest}  The authors declare that they have no conflict of interest.


{  \small
\begin{thebibliography}{}
	%
	%
	
	
	\bibitem{A89MR976971} Alinhac, S.:
	{Existence d'ondes de rar{\'e}faction pour des syst{\`e}mes quasi-lin{\'e}aires hyperboliques multidimensionnels}. 
	{\it Commun. Partial Differ. Eqs.}  {\bf 14}, 173--230  (1989).
	{http://doi.org/10.1080/03605308908820595}
	
	\bibitem{AG07MR2304160} Alinhac, S., G{{\'e}}rard, P.: 
	{{\it Pseudo-Differential Operators and the {N}ash-{M}oser Theorem}}. Translated from the 1991 French original by Stephen S. Wilson.
	American Mathematical Society: Providence (2007).
	
	{https://doi.org/10.1090/gsm/082}
	
	
	\bibitem{B-GS07MR2284507}
	Benzoni-Gavage, S. , Serre, D.: 
	{{\it Multidimensional Hyperbolic Partial Differential Equations: First-Order Systems and Applications}}.
	Oxford University Press: Oxford (2007).
	
	{https://doi.org/10.1093/acprof:oso/9780199211234.001.0001}
	
	\bibitem{CF18MR3791458}
	Chen, G.-Q., Feldman, M.:
	{{\it The Mathematics of Shock Reflection-Diffraction and von Neumann's Conjectures}}.
	Annals of Mathematics Studies, {\bf 197}. 
	Princeton University Press: Princeton, NJ  (2018).
	
	{https://doi.org/10.2307/j.ctt1jktq4b}  
	
	
	\bibitem{CSW19MR3925528}  Chen, G.-Q., Secchi, P., Wang, T.: 
	{Nonlinear stability of relativistic vortex sheets in three-dimensional Minkowski spacetime}. 
	{\it Arch. Ration. Mech. Anal.} {\bf  232}, 591--695 (2019).
	{https://doi.org/10.1007/s00205-018-1330-5}
	
	
	
	\bibitem{CW08MR2372810}  Chen, G.-Q., Wang, Y.-G.: 
	{Existence and stability of compressible current-vortex sheets in three-dimensional magnetohydrodynamics}. 
	{\it Arch. Ration. Mech. Anal.}  {\bf 187}, 369--408 (2008).
	
	{https://doi.org/10.1007/s00205-007-0070-8}
	
	\bibitem{CW12MR3289359}  Chen, G.-Q., Wang, Y.-G.: 
	{Characteristic discontinuities and free boundary problems for hyperbolic conservation laws}. 
	In: Holden, H., Karlsen, K.H. (eds.) {\it Nonlinear Partial Differential Equations}, 53--81. Springer: Heidelberg (2012).
	
	{https://doi.org/10.1007/978-3-642-25361-4}
	
	\bibitem{CHW17MR3628211}  Chen, R.M., Hu, J., Wang, D.: 
	{Linear stability of compressible vortex sheets in two-dimensional elastodynamics}. {\it Adv. Math.}  {\bf 311}, 18--60 (2017).
	
	{https://doi.org/10.1016/j.aim.2017.02.014}
	
	
	
	\bibitem{CHW19}  Chen, R.M., Hu, J., Wang, D.: 
	{Linear stability of compressible vortex sheets in 2D elastodynamics: variable coefficients}. {\it Math. Ann.} {\bf 376},  863--912 (2020).
	
	{https://doi.org/10.1007/s00208-018-01798-w}
	
	
	\bibitem{CT18MR3799089} Christoforou, C., Tzavaras, A.E.:  
	{Relative entropy for hyperbolic-parabolic systems and application to the constitutive theory of thermoviscoelasticity}.  {\it Arch. Ration. Mech. Anal.}  {\bf 229}, 1--52 (2018).
	
	{https://doi.org/10.1007/s00205-017-1212-2}
	
	\bibitem{CGT18MR3910194} Christoforou, C., Galanopoulou, M., Tzavaras, A.E.:  
	{A symmetrizable extension of polyconvex thermoelasticity and applications to zero-viscosity limits and weak-strong uniqueness}. {\it Commun. Partial Differ. Eqs.} {\bf 43}, 1019--1050 (2018).
	{https://doi.org/10.1080/03605302.2018.1456551}
	
	
	
	\bibitem{CGT19MR4026976} Christoforou, C., Galanopoulou, M., Tzavaras, A.E.:  
	{Measure-valued solutions for the equations of polyconvex adiabatic thermoelasticity}. {\it Discrete Contin. Dyn. Syst.} {\bf 39}, 6175--6206 (2019).
	
	{https://doi.org/10.3934/dcds.2019269}
	
	
	\bibitem{C88MR936420} Ciarlet, P.G.: {{\it  Mathematical Elasticity. Vol. I: Three-Dimensional Elasticity}}. North-Holland Publishing Co.: Amsterdam (1988).
	
	{https://www.sciencedirect.com/bookseries/studies-in-mathematics-and-its-applications/vol/20/suppl/C}
	
	\bibitem{CN63MR153153}  Coleman, B.D., Noll, W.:
	{The thermodynamics of elastic materials with heat conduction and viscosity}.
	{\it Arch. Ration. Mech. Anal.} {\bf 13}, 167--178 (1963).
	{https://doi.org/10.1007/BF01262690}
	
	
	\bibitem{CS04MR2095445} Coulombel, J.-F., Secchi, P.: 
	{The stability of compressible vortex sheets in two space dimensions}. {\it Indiana Univ. Math. J.}  {\bf 53}, 941--1012 (2004).
	
	{https://doi.org/10.1512/iumj.2004.53.2526}
	
	\bibitem{CS08MR2423311} Coulombel, J.-F., Secchi, P.: 
	{Nonlinear compressible vortex sheets in two space dimensions}. {\it Ann. Sci. {\'E}c. Norm. Sup{\'e}r. (4)}  {\bf 41}, 85--139 (2008).
	
	{https://doi.org/10.24033/asens.2064}
	
	\bibitem{D86MR846895} Dafermos, C.M.:  
	{Quasilinear hyperbolic systems with involutions}. {\it Arch. Ration. Mech. Anal.} {\bf 94}, 373--389 (1986).
	
	{https://doi.org/10.1007/BF00280911}
	
	
	\bibitem{D16MR3468916} Dafermos, C.M.: 
	{{\it Hyperbolic Conservation Laws in Continuum Physics}}. 4th edition, Springer-Verlag, Berlin (2016).
	
	{https://doi.org/10.1007/978-3-662-49451-6}
	
	
	\bibitem{HW11MR2737829} Hu, X., Wang, D.: 
	{Global existence for the multi-dimensional compressible viscoelastic flows}. {\it J. Differ. Eqs.}  {\bf 250}, 1200--1231 (2011).
	
	{https://doi.org/10.1016/j.jde.2010.10.017}
	
	
	\bibitem{LLZ08MR2393434}  Lei, Z., Liu, C., Zhou, Y.: 
	{Global solutions for incompressible viscoelastic fluids}. {\it Arch. Ration. Mech. Anal.} {\bf 188} 371--398 (2008).
	
	{https://doi.org/10.1007/s00205-007-0089-x}
	
	\bibitem{M01MR1842775} M{{\'e}}tivier, G.: Stability of multidimensional shocks. 
	In: Freist{\"u}hler, H., Szepessy, A. (eds.) {{\it Advances in the Theory of Shock Waves}}, pp.\;25--103. 
	Birkh{\"a}user: Boston (2001).
	{https://doi.org/10.1007/978-1-4612-0193-9}
	
	\bibitem{MTT15MR3306348} Morando, A., Trakhinin, Y., Trebeschi, P.: 
	{Well-posedness of the linearized problem for {MHD} contact discontinuities}. {\it J. Differ. Eqs.}  {\bf 258}, 2531--2571 (2015).
	{https://doi.org/10.1016/j.jde.2014.12.018}
	
	
	
	\bibitem{MTT18MR3766987} Morando, A., Trakhinin, Y., Trebeschi, P.: 
	{Local existence of {MHD} contact discontinuities}. {\it Arch. Ration. Mech. Anal.}  {\bf 228}, 691--742 (2018).
	
	{https://doi.org/10.1007/s00205-017-1203-3}
	
	
	\bibitem{QZ10MR2729321}   Qian, J., Zhang, Z.: 
	{Global well-posedness for compressible viscoelastic fluids near equilibrium}. 
	{\it Arch. Ration. Mech. Anal.}  {\bf 198}, 835--868 (2010).
	
	{https://doi.org/10.1007/s00205-010-0351-5}
	
	\bibitem{S96MR1405665} Secchi, P.: 
	{Well-posedness of characteristic symmetric hyperbolic systems}.
	{\it Arch. Ration. Mech. Anal.}  {\bf 134}, 155--197 (1996).
	
	{https://dx.doi.org/10.1007/BF00379552}
	
	\bibitem{S70MR0290095} Stein, E.M.: 
	{\it Singular Integrals and Differentiability Properties of Functions}. 
	Princeton University Press: Princeton, N.J. (1970).
	
	{https://www.jstor.org/stable/j.ctt1bpmb07}
	
	\bibitem{T07MR2328004} Tartar, L.:  
	{\it An Introduction to {S}obolev Spaces and Interpolation Spaces}. Springer: Berlin (2007).
	{https://doi.org/10.1007/978-3-540-71483-5}
	
	\bibitem{T09MR2481071} Trakhinin, Y.: 
	{The existence of current-vortex sheets in ideal compressible magnetohydrodynamics}. 
	{\it Arch. Ration. Mech. Anal.}   {\bf  191}, 245--310 (2009).
	
	{https://doi.org/10.1007/s00205-008-0124-6}
	
	\bibitem{T18MR3721411} Trakhinin, Y.: 
	{Well-posedness of the free boundary problem in compressible elastodynamics}. 
	{\it J. Differ. Eqs.}  {\bf 264}, 1661--1715 (2018).
	
	{https://doi.org/10.1016/j.jde.2017.10.005}
	
	\bibitem{TT60MR0118005}
	Truesdell, C., Toupin, R.: 
	{The classical field theories}.
	With an appendix on tensor fields by J.~L.~Ericksen.
	In: Fl\"{u}gge S. (eds.)
	{\it Handbuch der {P}hysik, {B}d.~{III}/1}, pp.\;226--793; appendix, pp.\;794--858. Springer: Berlin (1960).
	
	{https://doi.org/10.1007/978-3-642-45943-6}
	
\end{thebibliography}
}
\end{document}